%% file: main.tex
\documentclass[indent=latex,theme=default,review=false]{acmhomework}

\input{macros.tex}

\title{Logical Structure on Inverse Functor Categories}

\author{Marcelo Fiore}
\affiliation{University of Cambridge}
\email{marcelo.fiore@cl.cam.ac.uk}

\author{Krzysztof Kapulkin}
\affiliation{University of Western Ontario}
\email{kkapulki@uwo.ca}

\author{Yufeng Li}
\affiliation{University of Cambridge}
\email{yufeng.li@cl.cam.ac.uk}

\begin{document}

\begin{abstract}
  Inspired by recent work on the categorical semantics of dependent type theories, 
  we investigate the following question:  When is logical structure (crucially, 
  dependent-product and subobject-classifier structure) induced from a category to 
  categories of diagrams in it?
  Our work offers several answers, providing a variety of conditions on both the 
  category itself and the indexing category of diagrams.
  Additionally, motivated by homotopical considerations, we investigate the case when 
  the indexing category is equipped with a class of weak equivalences and study 
  conditions under which the localization map induces a structure-preserving functor 
  between presheaf categories.
\end{abstract}

\maketitle

\input{intro.tex}
\input{gpd-logic.tex}
\input{gluing-logic.tex}
\input{colim-cats.tex}
\input{glue-diagrams.tex}
\input{diagrams-logic.tex}

\printbibliography

\end{document}

%% file: macros.tex
\newcommand{\oM}{\overline{M}}

\newcommand{\M}{M}

\newcommand{\om}{\overline{m}}

\newcommand{\m}{m}

\newcommand{\tphi}{\widetilde{\varphi}}
\newcommand{\trho}{\widetilde{\rho}}
\newcommand{\txtev}{\textsf{ev}}

\DeclareMathOperator{\ocosk}{\overline{cosk}}
\DeclareMathOperator{\ores}{\overline{res}}
\DeclareMathOperator{\Gl}{Gl}

%% file: intro.tex
\section*{Introduction}
Given a category $\McE$, seen as a universe of discourse, with some logical structure 
one is typically interested in inducing the corresponding logical structure on the
presheaf category $\McE^\McI$ of $\McI$-shaped diagrams in $\McE$.
Here, by logical structure we specifically mean: categorical products and
coproducts, dependent sums (i.e.~left adjoints to pullback functors), 
dependent products (i.e.~right adjoint to pullback functors), as well as 
subobject classifiers.

In a diagram category, categorical products and coproducts (being limits and colimits) 
and dependent sums (being given by postcomposition) are naturally constructed pointwise.
While, on the other hand, the construction of dependent products and subobject 
classifiers is generally involved.

When, besides dependent products, one allows extra structure on $\McE$ (typically, admitting
sufficiently many limits) dependent products in $\McE^\McI$ may be constructed.
For example, it is shown in~\cite[Theorem 2.12]{sv10} that if $\McE$ is finitely
complete and has products indexed by  the class of all maps $\mor(\McI)$ in
$\McI$, then dependent products in $\McE$ give rise to dependent products in
$\McE^\McI$.
In general, as $\mor(\McI)$ can be large, this puts a stringent completeness requirement 
on $\McE$.
Therefore, it is natural to investigate the possibility of lessening such 
assumptions 
on $\McE$.

Particularly interesting examples in which not enough limits may exist for arbitrary 
indexing categories $\McI$ are those where the universe of discourse consists of: 
$(i)$~``finite structures'' (such as the topos $\FinSet$ of finite sets and functions); 
and, $(ii)$~``syntactic structures'' (such as the free topos~\cite{sl80} or classifying 
categories for various flavours of (dependent) type theories~\cite{kl21,shu14}).

A concrete counterexample of a diagram category in which exponentials (and hence 
dependent products) fail to exist, even though they exist in the universe of 
discourse, is the category of finitely branching forests $\FinSet^{\omega}$.
Indeed, given diagrams $X,Y \in \FinSet^\omega$, if the exponential 
$Y^X \in \FinSet^\omega$ were to exist, at each $n \in \omega$, the component $(Y^X)n$ 
should consist of compatible families of functions $(Xm \to Ym)_{n < m}$ which, in general, 
need not be finite (as there are infinitely many $m \in \omega$ such that $n < m$ for any 
fixed $n \in \omega$).
This counterexample shows that in order to reduce the requirements on $\McE$ while still 
ensuring the existence of exponentiable objects and, more generally, \emph{powerful} maps
(viz.~exponentiable objects in slice categories) in the category of diagrams $\McE^\McI$, 
one needs to put restrictions on the indexing category $\McI$.


In the first part of the paper, we explore the construction of dependent products
and subobject classifiers in $\McE^\McI$ induced by their counterparts in $\McE$.  We do 
this by placing various classes of restrictions on $\McI$ that are successively generalized 
by proceeding in a modular fashion.  This is structured as follows.
%
  %

  We begin in~\Cref{sec:gpd-logic} with simple diagram shapes given by groupoids.
  Here, as expected, dependent products and subobject classifiers are constructed 
  pointwise.
  However, 
  for the dependent product, rather than directly working with groupoidal diagrams of the 
  form $\McE^\bG$, for $\bG$ a groupoid, we instead work with the generalization 
  $\McE^{\McC^{-1}\McC}$, where $\McC$ is any category and $\McC^{-1}\McC$ is the 
  homotopical category obtained by inverting all maps (this generalizes the groupoidal 
  case because a category $\McC$ is a groupoid if and only if the localization 
  $\McC \to \McC^{-1}\McC$ is an equivalence).
  In so doing, 
  and in connection to our subsequent development in the second part of the paper, 
  we both construct dependent products in
  $\McE^{\McC^{-1}\McC}$ and also show that they are preserved by the inclusion
  $\McE^{\McC^{-1}\McC} \to \McE^\McC$.
  A type-theoretic version of this result was previously proved 
  in~\cite[Proposition 5.13]{kl21}.
  %
  %
  Of course, these 
  constructions 
  fail to even encompass the simple case of arrow categories.
  These we consider next in their 
  natural generalization 
  as \emph{Artin-gluing} categories.
  Thus, as a first step towards encompassing shapes with non-invertible arrows, 
  in~\Cref{sec:gluing-logic}, we consider logical structure in Artin-gluing
  categories.
  %
  For $[n]$ the free category generated by $n$ composable arrows, the iteration of 
  the construction of logical structure in arrow and Artin-gluing categories 
  in~\Cref{sec:gluing-logic} gives rise to a compatible family of logical structures 
  in each $\McE^{[n]^\op}$.
  In passing from the finite to the infinite case, one need look at logical structure 
  on 
  $\McE^{\omega^\op} \simeq \McE^{\colim_n [n]^\op} \simeq \lim_n\McE^{[n]^\op}$; 
  intuitively, by assembling the family of compatible logical structures in each 
  $\McE^{[n]^\op}$ to induce corresponding logical structures in $\McE^{\omega^\op}$.
  This naturally leads to the more general question of inducing logical structures in 
  the limit of categories $\lim_j\bD_j$ from compatible logical structures in each 
  $\bD_j$.
  A solution to this problem is provided in~\Cref{sec:limit-logic}.
  %
  By~\cite{shu15}, one observes that the categories built up inductively in a
  colimiting process via repeatedly applying the Artin-gluing construction into
  groupoidal categories, encompasses the notion of \emph{inverse} category (a special 
  case of Reedy categories~\cite{bm10}) which play an important role in homotopical
  category theory.
  Motivated by this, 
  in~\Cref{sec:iter-glue}, 
  we introduce a framework that encompasses categories obtained from iterated 
  Artin gluing.
  \Cref{sec:diagrams-logic}
  then combines 
  the results of~\Cref{sec:gpd-logic,sec:gluing-logic,sec:limit-logic}. 
  in this framework. 
  %

In the second part of the paper, we re-examine the development of~\Cref{sec:gpd-logic}.
Specifically, the results there prompt the following question: Rather than specifying 
\emph{all} maps in $\McC$ as weak equivalences, are there conditions on a collection 
of maps $\McW \hookrightarrow \McC$ specifying weak equivalences such that 
$\McE^{\McW^{-1}\McC} \to \McE^\McC$ preserves dependent products?  That is, such that 
the dependent product in the homotopical diagram category agrees with the dependent 
product in the diagram category.
%
%
One answer, other than the case $\McW = \McC$ 
proven in~\Cref{sec:gpd-logic}, is provided by \cite[Proposition 2.10]{sv10}, which
requires $\McE^{\McW^{-1}\McC} \to \McE^\McC$ to be dense and fully faithful.
Here, for $\McC$ built up inductively in a colimiting process as per the framework 
of~\Cref{sec:iter-glue}, we provide an alternative answer to this question phrased 
only with respect to the relation between the weak equivalences $\McW$ and the 
ambient category $\McC$.
This is the content of the highly technical~\Cref{sec:homotopical-Pi}.


%% file: gpd-logic.tex
\section{Logical Structure in Groupoids}\label{sec:gpd-logic}
As explained in the introduction, our investigation into the construction of the
subobject classifier and dependent product in diagram categories $\McE^\McI$
starts with a warm-up by considering the simplest case where $\McI$ is a
groupoid.
%

\subsection{Subobject Classifiers}\label{subsec:gpd-Omega}
Let $\bG$ be a groupoid and $\McE$ be a category.
The goal for this part is to show in \Cref{prop:gpd-Omega-ptwise} that if $\McE$
has a subobject classifier and truth map then so does $\McE^\bG$, and these
logical structures are constructed pointwise provided that $\bG$ is connected or
$\McE$ has an initial object.

\begin{lemma}\label{lem:gpd-conn-mono-ptwise}
  If $\bG$ is connected or $\McE$ admits an initial object then a map in
  $\McE^\bG$ is a mono exactly when each of its components are.
\end{lemma}
\begin{proof}
  If $\bG$ is connected or $\McE$ has an initial object, for each $x \in \bG$, the
  functor $\ev_x = x^* \colon \McE^\bG \to \McE = \McE^{\mbbo{1}}$ given by
  restriction along $x \colon \mbbo{1} \to \bG$ admits a left adjoint $\Delta_x$.
  Examining the left Kan extension formula, one sees that for each $e \in \McE$,
  the functor $\Delta_x e$ maps each $y \in \bG$ to $e$ if $y$ is in the same
  connected component with $x$ and if there is $y \in \bG$ disconnected from
  $x$ then $\Delta_x e$ maps $y$ to the initial object 0 in $\McE$.

  The existence of a left adjoint as above shows that
  $\alpha \colon G \hookrightarrow F \in \McE^\bG$ is a mono just in case each
  of its component $\alpha_x \colon Gx \hookrightarrow Fx \in \McE$ is a mono.
  To see this, suppose $\alpha$ is a mono.
  To show each component
  $\alpha_x = \ev_x\alpha \colon Gx = \ev_xG \to Fx = \ev_xF$ is a mono is to
  show $(\ev_x\alpha)_* \colon \McE(e, \ev_xG) \to\McE(e, \ev_xF)$ is an
  injection.
  But because $\Delta_x \dashv \ev_x$ and $\alpha \colon G \hookrightarrow F$ is
  a mono, $(\ev_x\alpha) = \alpha_x$ is a mono.
  \begin{equation*}
    \begin{tikzcd}
      \McE(e, \ev_xG) \ar[r, "\cong"] \ar[d, "(\ev_x\alpha)_*"', hook]
      &
      \McE^\bG(\Delta_xe, G) \ar[d, "\alpha_*", hook]
      \\
      \McE(e, \ev_xF) \ar[r, "\cong"']
      &
      \McE^\bG(\Delta_xe, F)
    \end{tikzcd}
  \end{equation*}
\end{proof}
\begin{proposition}\label{prop:gpd-Omega-ptwise}
  Suppose either $\bG$ is connected or $\McE$ has an initial object 0.
  Then, subobject classifier and truth map in $\McE^\bG$ are given by the
  constant subobject classifier diagram and constant truth map.

  That is, if $\Omega \in \McE$ is the subobject classifier with
  $\true \colon 1 \to \Omega \in \McE$ being the truth map then the constant
  diagram $\Omega \in \McE^\bG$ and the constant natural transformation
  $\true \colon 1 \to \Omega^\bG \in \McE$ is serves as the subobject classifier
  and truth map in $\McE^\bG$.
\end{proposition}
\begin{proof}
  Assume there is a mono $\alpha \colon F \hookrightarrow G \in \McE^\bG$ so
  that by \Cref{lem:gpd-conn-mono-ptwise} each of its components
  $\alpha_x \colon Fx \hookrightarrow Gx \in \McE$ for $x \in \bG$ are monos.
  So, each $\alpha_x$ admits a characteristic map
  $\chi_x \colon Gx \to \Omega \in \McE$.
  Furthermore, for $g \colon x \cong x' \in \bG$, the pullback of
  $\true \colon 1 \to \Omega \in \McE$ along
  \begin{tikzcd}[cramped]
    Gx \ar[r,"Gg","\cong"'] & Gx' \ar[r,"\alpha_{x'}"] & \Omega
  \end{tikzcd}
  is $\alpha_x \colon Fx \hookrightarrow Gx$
  \begin{equation*}
    \begin{tikzcd}
      Fx  \ar[r, "Fg", "\cong"'] \ar[d, hook, "\alpha_x"']
      \ar[rd, "\lrcorner"{pos=0}, phantom]
      &
      Fx' \ar[d, hook, "\alpha_{x'}"{description}] \ar[r, "!"]
      \ar[rd, "\lrcorner"{pos=0}, phantom]
      &
      1 \ar[d, "\true", hook]
      \\
      Gx \ar[r, "\cong", "Gg"']
      &
      Gx' \ar[r, "\chi_{x'}"']
      &
      \Omega
    \end{tikzcd}
  \end{equation*}
  And so by uniqueness of the characteristic map, $\chi_x = Gg \cdot \chi_{x'}$.
  In other words, this shows that one has a natural transformation
  $\chi = (\chi_x)_x \colon G \to \Omega \in \McE^\bC$, giving rise to
  \begin{equation*}
    \begin{tikzcd}
      F \ar[d, "\alpha"', hook] \ar[r, "!"]
      \ar[rd, "\lrcorner"{pos=0}, phantom]
      &
      1 \ar[d, "\true", hook]
      \\
      G \ar[r, "\chi"']
      &
      \Omega
    \end{tikzcd} \in \McE^\bG
  \end{equation*}
  which is a pullback because each of its components are.

  And clearly $\chi$ is the only possible characteristic map $G \to \Omega$ for
  $\alpha \colon F \hookrightarrow G$, because all other characteristic maps
  $G \to \Omega$ would have components $\chi_x$ for each $x \in \bG$.
\end{proof}

\subsection{Dependent Products}\label{subsec:gpd-Pi}
In this part, we show that dependent products in categories of diagrams indexed
by groupoids are likewise constructed pointwise.
However, we note that a groupoid $\bG$ is equivalent to $\bG^{-1}\bG$, the
category obtained by formally inverting all arrows in $\bG$ and generalise this
observation by not working with groupoids $\bG$ but with categories $\bC$ and
the groupoid $\bG \coloneqq \bC^{-1}\bC$ obtained by formally inverting all arrows in $\bC$.
We structure the construction in this part in such a way to also show that the
dependent product of homotopical diagrams is always their dependent product
viewed as ordinary diagrams.
Specifically, fixing a category $\McE$ and denoting by
$\gamma \colon \McC \to \bG = \McC^{-1}\McC$ to be the localisation,
$\gamma^* \colon \McE^\bG \hookrightarrow \McE^\McC$ is the inclusion of the
full subcategory of the functors that send all maps in $\McC$ to isomorphisms in
$\McE$ into the functor category $\McE^\McC$.
We show in \Cref{thm:all-invert-Pi-htpy} that for
$h \colon B \to A \in \McE^{\bG} = \McE^{\McC^{-1}\McC}$ and
$k \colon C \to B \in \McE^{\bG}/B$, one has an isomorphism
$\gamma^*(\Pi_BC) \simeq \Pi_{\gamma^*B}{\gamma^*C}$.
A type-theoretic version of this result is given in \cite[Proposition
5.13]{kl21}.

\begin{definition}{}
  A map $f \colon c \to d$ in a category $\McC$ is \emph{powerful} if pullback
  along $f$ exists and admits a right adjoint.
\end{definition}

\begin{construction}\label{constr:gpd-Pi}
  Fix a map $h \colon B \to A \in \McE^{\bG}$ such that each component is powerful.
  For each $k \colon C \to B \in \McE^{\bG}/B$, define
  $\Pi(B,k) \colon \Pi(B,k) \to A \in \McE^{\bG}/B$ whose actions an objects are
  $\Pi(B,C)x \coloneqq \Pi_{Bx}Cx$.
  Because $\bG$ is a groupoid, for each $g \colon x \to x' \in \bG$, the bottom
  square is a pullback.
  So, by the adjunction $h_{x'}^* \dashv \Pi_{Bx}'$, one may define
  $\Pi(B,C)g \colon \Pi_{Bx}Cx \to \Pi_{Bx'}Cx'$ as the unique map such that
  \begin{equation*}
    \begin{tikzcd}
      Cx
      \ar[rdd, "Cg"']
      \ar[rd, bend right=10, "k_x"{description}]
      &
      Bx \times_{Ax} \Pi_{Bx}{Cx}
      \ar[l, "\ev"']
      \ar[d]
      \ar[rr]
      \ar[rrd, phantom, "\lrcorner"{description, pos=0}]
      &
      &
      \Pi_{Bx}{Cx}
      \ar[d]
      \ar[rdd, "{\Pi(B,C)g}"]
      \\
      &
      Bx
      \ar[rr, "h_x"{description}]
      \ar[rdd, "Bg" {description, pos=0.25}]
      \ar[rrrdd, phantom, "\lrcorner"{description, pos=0.01}]
      &
      &
      Ax
      \ar[rdd, "Ag" {pos=0.25, description}]
      \\
      &
      Cx'
      \ar[rd, bend right=10, "k_{x'}"{description}]
      &
      Bx' \times_{Ax'} \Pi_{Bx'}{Cx'}
      \ar[l, "\ev"'{pos=0.7}, crossing over]
      \ar[d]
      \ar[rr, crossing over]
      \ar[rrd, phantom, "\lrcorner"{description, pos=0}]
      \ar[uul, leftarrow, crossing over, "{(Bg, \Pi(B,C)g)}" {description, pos=0.75}]
      &
      &
      \Pi_{Bx'}{Cx'}
      \ar[d]
      \\
      &
      &
      Bx'
      \ar[rr, "h_{x'}"{description}]
      &
      &
      Ax'
    \end{tikzcd}
  \end{equation*}
  Moreover, using component-wise counits
  $Bx \times_{Ax} \Pi_{Bx}(-) \dashv \id$, define a family
  $\epsilon_{C} \coloneqq (\epsilon_{C,x} \colon Bx \times_{Ax} \Pi(B,C)x \to
  Cx)_{x \in \bG}$ with each
  $\epsilon_{C,x} \coloneqq \ev \colon Bx \times_{Ax} \Pi_{Bx}{Cx} \to Cx$.
\end{construction}

\begin{lemma}\label{lem:gpd-Pi}
  \Cref{constr:gpd-Pi} defines a functor $\Pi(B,C) \colon \bG \to \McE$ and a natural transformation
  $\epsilon_C$.
  Moreover, the construction is functorial in $C$.
\end{lemma}
\begin{proof}
  It is easy to observe that each $\Pi(B,C)$ is a functor $\bG \to \McE$ and
  that $\epsilon_C$ is a natural transformation.

  Next, we note $\Pi(B,C)$ is natural in $C \in \McE^\bG/B$.
  If there is $k' \colon C' \to B \in \McE^\bG/B$ and $f \colon C \to C'$
  over $B$ then functoriality of each $\Pi_{Bx}$ defines maps
  $\Pi(B,C)x = \Pi_{Bx}{Cx} \xrightarrow{\Pi_{Bx}{f_x}} \Pi_{Bx}{C'x} =
  \Pi(B,C')x$.
  To observe naturality is to note that
  $\textcolor{blue0}{\Pi(B,C')g \cdot \Pi_{Bx}{f_x}} =
  \textcolor{red0}{\Pi_{Bx'}{f_{x'}} \cdot \Pi(B,C)g}$ for arbitrary
  $g \colon x \to x' \in \bG$.
  Under $h_{x'}^* \dashv \Pi_{Bx'}$, the transpose of
  $\textcolor{blue0}{\Pi(B,C')g \cdot \Pi_{Bx}{f_x}}$ is given by
  \begin{align*}
    \textcolor{magenta0}{\ev} \cdot \textcolor{blue0}{(Bg, \Pi(B,C')g) \cdot
    h^*_x(\Pi_{Bx}{f_x})}
  \end{align*}
  while the transpose of
  $\textcolor{red0}{\Pi_{Bx'}{f_{x'}} \cdot \Pi(B,C)g}$ is given by
  \begin{align*}
    \textcolor{magenta0}{\ev} \cdot \textcolor{red0}{h^*_{x'}(\Pi_{Bx'}{f_{x'}}) \cdot (Bg, \Pi(B,C)g)}
  \end{align*}
  By construction,
  $\textcolor{magenta0}{\ev} \cdot \textcolor{blue0}{(Bg, \Pi(B,C')g)} = C'g
  \cdot \ev$ and by naturality of $\ev \colon B \times_A \Pi_B(-) \to \id$,
  it follows that
  $\textcolor{magenta0}{\ev} \cdot \textcolor{red0}{(Bg,
    h^*_{x'}(\Pi_{Bx'}{f_{x'}}))} = f_{x'} \cdot \ev$.
  Hence,
  $\textcolor{magenta0}{\ev} \cdot \textcolor{red0}{h^*_{x'}(\Pi_{Bx'}{f_{x'}})
    \cdot (Bg, \Pi(B,C)g)} = f_{x'} \cdot \ev \cdot \textcolor{red0}{(Bg,
    \Pi(B,C)g)}$.
  By naturality of $f$ and $\ev$ once again, the result follows.
  \begin{equation*}
    \begin{tikzcd}[column sep=tiny]
      {C'x} && {Bx \times_{Ax} \Pi_{Bx}{C'x}} &&&& {\Pi_{Bx}{C'x}} \\
      & Cx && {Bx \times_{Ax} \Pi_{Bx}{Cx}} && {\Pi_{Bx}{Cx}} \\
      \\
      & Bx && Bx && Ax \\
      && {Cx'} && {Bx' \times_{Ax'} \Pi_{Bx'}{Cx'}} && {\Pi_{Bx'}{Cx'}} \\
      & {C'x'} && {Bx' \times_{Ax'} \Pi_{Bx'}{C'x'}} &&&&& {\Pi_{Bx'}{C'x'}} \\
      && {Bx'} && {Bx'} && {Ax'}
      \arrow[from=5-5, to=7-7, phantom, "\lrcorner"{pos=0}]
      \arrow[from=2-4, to=4-6, phantom, "\lrcorner"{pos=0}]
      \arrow[from=1-3, to=2-6, phantom, "\lrcorner"{pos=0}]
      \arrow[from=4-4, to=7-5, "Bg" {pos=0.15, description}]
      \arrow[from=4-6, to=7-7, "Ag" {description}]
      \arrow[from=7-5, to=7-7, "h_{x'}" {description}]
      \arrow[from=4-4, to=4-6, "h_x" {description}]
      \arrow[from=5-5, to=7-5]
      \arrow[from=5-7, to=7-7, "\Pi_{Bx'}{k_{x'}}" {description, pos=0.25}]
      \arrow[from=2-6, to=4-6, "\Pi_{Bx}{k_x}"']
      \arrow[from=2-4, to=2-6]
      \arrow[from=2-2, to=4-2, "k_x" {description}]
      \arrow[from=5-3, to=7-3, "k_{x'}" {description, pos=0.25}]
      \arrow[Rightarrow, no head, from=4-2, to=4-4]
      \arrow[Rightarrow, no head, from=7-3, to=7-5]
      \arrow[from=4-2, to=7-3, "Bg" {description, pos=0.25}]
      \arrow[from=1-1, to=6-2, "C'g"']
      \arrow[from=1-3, to=4-4]
      \arrow[from=2-2, to=1-1, "f_x" {description}]
      \arrow[from=1-1, to=4-2, "k'_x" {description}]
      \arrow[from=6-2, to=7-3, "k'_{x'}"']
      \arrow[from=1-3, to=1-1, "\ev"']
      \arrow[from=2-6, to=1-7, "\Pi_{Bx}{f_x}" {description}, color=blue0]
      \arrow[from=2-4, to=1-3, "h_x^*(\Pi_{Bx}f_x)"', color=blue0]
      \arrow[from=5-7, to=6-9, "\Pi_{Bx'}{f_{x'}}" {description}, color=red0]
      \arrow[from=1-7, to=4-6, "\Pi_{Bx}{k_x}"]
      \arrow[from=6-9, to=7-7, "\Pi_{Bx'}{k_x}"]
      \arrow[from=1-7, to=6-9, "{\Pi(B,C')g}", color=blue0]
      \arrow[from=1-3, to=1-7]
      \arrow[from=2-4, to=4-4]
      \arrow[from=5-5, to=5-7, crossing over]
      \arrow[from=2-6, to=5-7, crossing over, "{\Pi(B,C)g}"{description, pos=0.6}, color=red0]
      \arrow[from=5-5, to=6-4, crossing over, "h_{x'}^*(\Pi_{Bx'}{f_{x'}})"{description}, color=red0]
      \arrow[from=5-5, to=5-3, crossing over, "\ev" {description, pos=0.4}]
      \arrow[from=2-4, to=5-5, crossing over, "{(Bg, \Pi(B,C)g)}" {description}, color=red0]
      \arrow[from=2-4, to=2-2, crossing over, "\ev" {description}]
      \arrow[from=1-3, to=6-4, crossing over, "{(Bg, \Pi(B,C')g)}" {description}, color=blue0]
      \arrow[from=5-3, to=6-2, crossing over, "f_{x'}"{description, pos=0.25}]
      \arrow[from=6-4, to=6-2, crossing over, "\ev"' {pos=0.25}, color=magenta0]
      \arrow[from=6-4, to=6-9, crossing over]
      \arrow[from=6-4, to=7-5, crossing over]
      \arrow[from=2-2, to=5-3, crossing over, "Cg" {description}]
    \end{tikzcd}
  \end{equation*}

  Finally, functoriality of $\epsilon_C$ in $C$ amounts to the functoriality of
  each $\Pi_{Bx}$ because
  $\epsilon_{C',x} \cdot (Bx \times_{Ax} \Pi(B,f)_x)
  = \ev \cdot \textcolor{blue0}{h_x^*(\Pi_{Bx}{fx})} = f_x \cdot \ev = f_x \cdot \epsilon_{C,x}$.
\end{proof}

\begin{lemma}\label{lem:gpd-Pi-htpy}
  For $\Pi(B,C) \in \McE^\bG/A$ as constructed in \Cref{constr:gpd-Pi}, it image
  $\gamma^*(\Pi(B,C)) \in \McE^\McC/\gamma^*A$ under the inclusion
  $\gamma \colon \McE^\bG \hookrightarrow \McE^\McC$ is the dependent product
  $\Pi_{\gamma^*B}\gamma^*C \in \McE^\McC/\gamma^*A$.
\end{lemma}
\begin{proof}
  It suffices to prove that $\gamma^*(\Pi(B,C))$ has the same universal property
  as $\Pi_{\gamma^*B}\gamma^*C$ by showing that $\gamma^*(\Pi(B,C))$ represents
  the functor
  \begin{equation*}
    \sfrac{\McE^\McC}{\gamma^*B}(- \times_{\gamma^*A} \gamma^*B, \gamma^*C)
    \colon (\sfrac{\McE^\McC}{\gamma^*A})^\op
    \to
    \Set
  \end{equation*}

  Fix an object $d \colon D \to \gamma^*A \in \McE^\McC/\gamma^*A$ and a natural
  transformation $t \colon D \to \gamma^*(\Pi(B,C)) \in \McE^\McC/\gamma^*A$.
  Such a natural transformation is a compatible family
  $(t_x \colon Dx \to \Pi_{Bx}Cx \in \McE/Ax)_{x \in \McC}$, in that for each
  $g \colon x \to x' \in \McC$, one has $t_{x'} \cdot Dg = \Pi(B,C)g \cdot t_x$.
  Thus, pulling back along $h_{x'}$ and noting that the bottom face is a
  pullback because $Bg$ and $Ag$ are isomorphisms by the fact that
  $A,B \in \McE^\bG$ are homotopic,
  %
  \begin{equation*}
    \begin{tikzcd}
      {Bx \times_{Ax} Dx} & {Bx \times_{Ax} \Pi_{Bx}{Cx}} & Cx & {\Pi_{Bx}{Cx}} & Dx \\
      && Bx & Ax \\
      & {Bx' \times_{Ax'} Dx'} & {Bx' \times_{Ax'} \Pi_{Bx'}{Cx'}} & {Cx'} & {\Pi_{Bx'}{Cx'}} & {Dx'} \\
      &&& {Bx'} & {Ax'}
      \arrow["{t_{x'}}"{description}, from=3-6, to=3-5]
      \arrow["{t_x}"{description}, from=1-5, to=1-4]
      \arrow["Dg"{description}, from=1-5, to=3-6]
      \arrow["Ag"', "\cong", from=2-4, to=4-5]
      \arrow[from=3-5, to=4-5]
      \arrow[from=1-4, to=2-4]
      \arrow["{d_{x'}}"{description}, from=3-6, to=4-5]
      \arrow["{d_x}"{description}, from=1-5, to=2-4]
      \arrow["{h_{x'}}"{description}, from=4-4, to=4-5]
      \arrow["Bg"'{pos=0.2}, "\cong"{pos=0.2}, from=2-3, to=4-4]
      \arrow[from=3-4, to=4-4]
      \arrow["\ev"{description}, from=1-2, to=1-3]
      \arrow[from=1-3, to=2-3]
      \arrow[from=3-3, to=4-4]
      \arrow[from=1-2, to=2-3]
      \arrow["{Bx \times_{Ax} t_x}", from=1-1, to=1-2]
      \arrow["{Bx' \times_{Ax'} t_{x'}}", from=3-2, to=3-3]
      \arrow[from=3-2, to=4-4]
      \arrow[from=1-1, to=2-3]
      \arrow["{(Bg,Dg)}"{description}, from=1-1, to=3-2]
      \arrow["{h_x}"{description, pos=0.25}, from=2-3, to=2-4]
      \arrow["\ev"{description}, from=3-3, to=3-4, crossing over]
      \arrow["{(Bg, \Pi(B,C)g)}"{description}, from=1-2, to=3-3, crossing over]
      \arrow["{\Pi(B,C)g}"{description, pos=0.7}, from=1-4, to=3-5, crossing over]
      \arrow["Cg"{description, pos=0.3}, from=1-3, to=3-4, crossing over]
    \end{tikzcd}
  \end{equation*}
  Because $g \colon x \to x' \in \McC$ is an arbitrary map, the left side of the
  above diagram shows that
  $(t_x^\dagger \colon Bx \times_{Ax} Dx \to Cx \in \McE/Bx)_{x \in \McC}$
  assembles to form a natural transformation
  $t^\dagger \colon \gamma^*B \times_{\gamma^*A} D \to \gamma^*C \in \McE^\McC/B$.

  Likewise, given
  $s \colon \gamma^*B \times_{\gamma^*A} D \to C \in \McE^\McC/B$ given by a
  compatible family
  $(s_x \colon Bx \times_{Ax} Dx \to Cx \in \McE/Bx)_{x \in \McC}$, its
  pointwise transposes
  $(s_x^\dagger \colon Dx \to \Pi_{Bx}Cx = \Pi(B,C) \in \McE/Ax)_{x \in \McC}$
  assemble to form a natural transformation $s^\dagger \colon D \to \gamma^*(\Pi(B,C))$.
  And because transposes are taken pointwise, one has a bijection
  \begin{align*}
    \sfrac{\McE^\McC}{\gamma^*A}(D, \gamma^*(\Pi(B,C))) \cong
    \sfrac{\McE^\McC}{\gamma^*A}(\gamma^*B \times_{\gamma^*A} D, \gamma^*C)
  \end{align*}
  It is moreover easy to see that this bijection constructed by taking pointwise
  transposes is natural in $D$.
  Therefore, $\gamma^*(\Pi(B,C)) = \Pi_{\gamma^*B}{\gamma^*C}$.
\end{proof}

\begin{lemma}\label{lem:gpd-Pi-ptwise}
  $\Pi(B,C) \in \McE^\bG/A$ as constructed in \Cref{constr:gpd-Pi} is the
  dependent product $\Pi_BC$.
\end{lemma}
\begin{proof}
  By \Cref{lem:gpd-Pi}, one observes that \Cref{constr:gpd-Pi} defines a functor
  $\Pi(B,-) \colon \McE^\bG/B \to \McE^\bG/A$ equipped with a map
  $\epsilon \colon B \times_A \Pi(B,-) \to \id$.
  To realise the transpose $k^* \dashv \Pi(B,-)$, it suffices to observe that
  \begin{align*}
    \sfrac{\McE^\bG}{A}(D, \Pi(B,C)) \xrightarrow{k^*} \sfrac{\McE^\bG}{B}(B \times_A D, B \times_A
    \Pi(B,C)) \xrightarrow{(\epsilon_C)_*} \sfrac{\McE^\bG}{B}(B \times_A D, C)
  \end{align*}
  defines a bijection for each $d \colon D \to A \in \McE^\bG/A$.
  This is because a natural transformation
  $f \colon D \to \Pi(B,C) \in \McE^\bG$ over $A$ is mapped to
  $((\epsilon_C)_* \cdot k^*)f$, which has component
  $Bx \times_{Ax} Dx \xrightarrow{Bx \times_{Ax} f_x} Bx \times_{Ax}
  \Pi_{Bx}{Cx} \xrightarrow{\ev} Cx$ at $x \in \bG$.
  This is exactly under the transpose of $f_x$ under $h_x^* \dashv \Pi_{Bx}$.
\end{proof}

Hence, by \Cref{lem:gpd-Pi,lem:gpd-Pi-htpy,lem:gpd-Pi-ptwise}, we have shown:
\begin{theorem}[{cf. \cite[Proposition 5.13]{kl21}}]\label{thm:all-invert-Pi-htpy}
  If $h \colon B \to A \in \McE^\bG$ is such that each
  $h_x \colon Bx \to Ax \in \McE$ is powerful then
  $\Pi_Bk \colon \Pi_BC \to A \in \McE^\bG/A$ exists for all
  $k \colon C \to B \in \McE^\bG/B$.
  Moreover, the dependent product
  $\Pi_{\gamma^*B}\gamma^*k \colon \Pi_{\gamma^*B}{\gamma^*C} \to \gamma^*A \in
  \McE^\McC/\gamma^*A$ exists and is isomorphic to
  $\gamma^*(\Pi_Bk) \colon \gamma^*(\Pi_BC) \to \gamma^*A \in
  \McE^\McC/\gamma^*A$.
\end{theorem}


%% file: gluing-logic.tex
\section{Logical Structure in Gluing Categories}\label{sec:gluing-logic}
Having dealt with the case of logical structures in $\McE^\McI$ where $\McI$
contains no non-trivial maps in \Cref{sec:gpd-logic}, one would next like to
investigate the case where $\McI$ has one single non-trivial arrow (i.e. the
free walking arrow category).
In this section, however, instead of taking $\McI = \set{\bullet \to \bullet}$,
we work instead with \emph{Artin gluing categories}, which generalise the
observation that arrow categories are simply comma categories of the identity
functor.


\begin{definition}\label{def:artin-glue}
  The \emph{Artin gluing category} of a functor $F \colon \McC \to \McE$, also
  known as its \emph{gluing category}, $\Gl(F)$, is defined as the comma
  category $\Gl(F) \coloneqq \McE \downarrow F$.
\end{definition}
Indeed, we see that the arrow category $\McE^\to$ is equivalent to
$\Gl(\id_\McE\relax) = \id_\McE \downarrow \McE$.

\subsection{Subobject Classifiers}
In this section, we construct the subobject classifier in $\Gl(F)$.
Its construction is motivated by the following example of the construction of
subobject classifier in co-presheaf categories.

\begin{example}[{\cite[\S I.4]{mm94}}]\label{ex:glu-Omega-ex}
  For $\McJ$ a (small) indexing category, the
  co-presheaf category $\Set^{\McJ}$ admits a subobject classifier $\Omega$
  taking each $j \in \McJ$ to the sub-co-presheaves of the representable at
  $j$.
  Each such sub-co-presheaf is equivalently a sieve on $j$: a subset $S$ of all
  the arrows whose domain is $j$ such that if $j \xrightarrow{\alpha'} j' \in S$
  then $j \xrightarrow{\alpha} j' \xrightarrow{\alpha'} j'' \in S$ for all
  composable maps $\alpha'$.

  %
  %
  %
  It is clear that the collage of the terminal profunctor $\McJ
  \slashedrightarrow \mbbo{1}$ taking each $(\bullet, j)$, where
  $\bullet$ is the unique object of the singleton category $\mbbo{1}$ and $j \in
  \McJ$, to the singleton set, is the category
  $\McJ^\lhd$ obtained by formally adjoining an initial object 0 to $\McJ$.
  Now, suppose $S$ is a sieve on 0 in $\McJ^\lhd$.
  Then, for each $j \in \McJ$ the restriction $S|_j \coloneqq \set{j
    \xrightarrow{\alpha} j' ~|~ 0 \xrightarrow{!} j \xrightarrow{\alpha}
    j'}$ is a sieve on $j$ in $\McJ$.
  By initiality of the 0, it is also clear that if $j_1 \xrightarrow{\phi}
  j_2$ then $(S|_{j_1})|_\phi \coloneqq \set{j_2 \xrightarrow{\alpha} j' ~|~ j_1
    \xrightarrow{\phi} j_2 \xrightarrow{\alpha} j' \in S|_{j_1}} = S|_{j_2}$.
  Hence, each sieve
  $S$ on 0 gives rise to a comaptible family of sieves $(S|_j)|_{j \in \McJ} \in
  \lim_{j \in \McJ} \Omega j$, where
  $\Omega$ is the subobject classifier of $\Set^\McJ$.
  Moreover, in the case that $\id_0 \in S$ then each
  $S|_j$ is the maximal sieve on $j$.
  Therefore, each sieve $S$ on 0 is an element $S \in 2 \times \lim_j \Omega
  j$ such that if $\pi_1 S = 1$ (i.e. $\id_0 \in S$) then $S|_j(\alpha)$ is the maximal sieve on
  $j$ for each $j \in \McJ$.

  Denote by $\chi_\true \colon \lim_j \Omega j \to
  2$ the characteristic function picking out the family of sieves whose
  component at each $j \in \McJ$ is the maximal sieve on $j$.
  Then, the sieve condition says that a sieve on $0 \in \McJ^\lhd$ is a pair $S
  \in 2 \times \lim_j \Omega_j$ such that if $\pi_1S = 1$ (i.e. $\id_0 \in
  S$) then $\chi_\true(\pi_2S) = 1$ (i.e. each
  $S|_j$ is the maximal sieve on $j$).

  From this, one observes that the sieves on $0 \in \McJ^\lhd$ is given by the
  equaliser
  \begin{equation*}
    \begin{tikzcd}[column sep=large]
      {\Omega^\lhd 0} & {2 \times \lim_j \Omega j} & {2 \times 2} & {2}
      \arrow["\wedge", from=1-3, to=1-4]
      \arrow["{\id \times \chi_{\true}}", from=1-2, to=1-3]
      \arrow["{\pi_1}"', bend right=10, from=1-2, to=1-4]
      \arrow[hook, "i", from=1-1, to=1-2]
    \end{tikzcd}
  \end{equation*}

  But note that the category of cones over $\McJ$-shaped diagrams in $\Set$ is
  given by
  $\Set^{\McJ^\lhd} \simeq \Gl(\Set^\McJ \xrightarrow{\lim\relax} \Set)$ by the
  universal property of the limit.
  By the equaliser defining $\Omega^\lhd 0$ above, one obtains a map
  $\Omega^\lhd 0 \hookrightarrow 2 \times \lim_j \Omega j \to \lim_j \Omega j$,
  which is equivalent to a cone over $(\Omega j)_j$ in $\Set$ at
  $\Omega^\lhd 0$.
  As a $\McJ^\lhd$-shaped diagram, this cone takes each object of $\McJ^\lhd$ to
  the sieves on said object, thus constructing the subobject classifier of
  $\Set^{\McJ^\lhd}$.
  Under the equivalence of categories
  $\Set^{\McJ^\lhd} \simeq \Gl(\Set^\McJ \xrightarrow{\lim\relax} \Set)$, the
  subobject classifier of the gluing category
  $\Gl(\Set^\McJ \xrightarrow{\lim\relax} \Set)$ is then
  $\Omega^\lhd 0 \hookrightarrow 2 \times \lim_j \Omega j \to \lim_j \Omega j$.
  %
  %
  %
  %
\end{example}
%

We now proceed to generalise the observation made in \Cref{ex:glu-Omega-ex} to
arbitrary gluing categories $\Gl(F \colon \McC \to \McE)$ in
\Cref{thm:glue-Omega}.
For this, we work under the following assumptions:
\begin{assumption}\label{asm:F-omega-cond} $ $
  \begin{itemize}
    \item Both $\McC$ and $\McE$ are equipped with subobject classifiers,
    respectively called $\Omega_\McC$ and$\Omega_\McE$.
    \item $\McE$ has all finite limits and $F$ preserves finite limits.
  \end{itemize}
\end{assumption}

First, we construct the subobject classifier and truth map in $\Gl(F)$ in the
following \Cref{constr:glue-true} and proceed to prove various properties about
them.

\begin{construction}\label{constr:glue-true}
  Because $\McE$ has all finite limits, the following equaliser exists
  \begin{equation}\label{eqn:E-Omega-eq}\tag{\textsc{$\Omega_{\Gl(F)}$-eq}}
    \begin{tikzcd}[column sep=large]
      {\Omega_{\Gl(F)}} & {\Omega_\McE \times F\Omega_\McC} & {\Omega_\McE \times \Omega_\McE} & {\Omega_\McE}
      \arrow["\wedge", from=1-3, to=1-4]
      \arrow["{\id \times \chi_{F(\true)}}", from=1-2, to=1-3]
      \arrow["{\pi_1}"', bend right=10, from=1-2, to=1-4]
      \arrow[hook, "i", from=1-1, to=1-2]
    \end{tikzcd}
  \end{equation}
  where $\chi_{F(\true)} \colon F\Omega_\McC \to \Omega_\McE$ is the
  classifying map of $F(\true) \colon F1 \hookrightarrow F\Omega_\McC$
  (which is a mono because $F$ preserves pullbacks so it preserves monos).
  And because $F$ preserves finite limits, one has an isomorphism
  $1_\McE \cong F1_\McC$.
  Define the map $\true_{\Gl(F)} \colon 1 \to \Omega_{\Gl(F)}$ as the map induced by
  $(1 \xrightarrow{\true} \Omega, 1 \cong F1 \xrightarrow{F(\true)} F\Omega)
  \colon 1 \to \Omega_\McE \times F\Omega_\McC$.
\end{construction}

\begin{lemma}\label{lem:glue-true}
  The map
  $(1 \xrightarrow{\true} \Omega_\McE, 1 \cong F1 \xrightarrow{F(\true)}
  F\Omega_\McC) \colon 1 \to \Omega_\McE \times F\Omega_\McC$ from
  \Cref{constr:glue-true} equalises
  $\pi_1, \wedge \cdot (\id \times \chi_{F(\true)}) \colon \Omega_\McE \times
  F\Omega_\McC \rightrightarrows \Omega_\McE$.
  So, the map $\true_{\Gl(F)}$ in \Cref{constr:glue-true} actually exists.
  Moreover, one has
  \begin{equation}\label{eqn:gl-Omega}\tag{\textsc{gl-$\Omega$}}
    \begin{tikzcd}
      1 \ar[rr, "\cong"] \ar[d, "\true_{\Gl(F)}"'] && F1 \ar[d, "F(\true)"] \\
      \Omega_{\Gl(F)} \ar[r, hookrightarrow] & \Omega_\McE \times F\Omega_\McC \ar[r, "\pi_2"'] & F\Omega_\McC
    \end{tikzcd} \in \McE
  \end{equation}
\end{lemma}
\begin{proof}
  Put $\varphi \coloneqq 1 \cong F1 \xrightarrow{F(\true)} F\Omega_\McC$ from
  \Cref{constr:glue-true}.
  Then, $\pi_1 \cdot (\true, \varphi) = \true \colon 1 \to \Omega_\McE$.
  Also,
  $\chi_{F(\true)} \cdot \varphi = 1 \cong F1 \xrightarrow{F(\true)} F\Omega_C
  \xrightarrow{\chi_{F(\true)}} \Omega = \true$ by definition of the
  characteristic function $\chi_{F(\true)}$.
  Therefore,
  $\wedge \cdot (\id \times \chi_{F(\true)}) \cdot (\true, \varphi) = \wedge
  (\true, \true) = \true$.
  Hence, $\true_{\Gl(F)}$ as above does indeed exist.

  In particular, the map
  $1 \xrightarrow{\true_{\Gl(F)}} \Omega_{\Gl(F)} \hookrightarrow \Omega_\McE \times
  F\Omega_C \xrightarrow{\pi_2} F\Omega_\McC = \varphi = 1 \cong F1
  \xrightarrow{F(\true)} F\Omega_\McC$.
  So the square in the statement above commutes.
\end{proof}

\begin{lemma}\label{lem:true-pb}
  For $\true_{\Gl(F)} \colon 1 \to \Omega_{\Gl(F)}$ from \Cref{constr:glue-true}, one
  has the pullback
  \begin{equation*}
    \begin{tikzcd}
      1 \cong F1 \ar[d, "\true_{\Gl(F)}"', hookrightarrow]
      \ar[r, "="]
      \ar[rd, "\lrcorner"{pos=0}, phantom]
      &
      1 \ar[d, "\true", hookrightarrow]
      \\
      \Omega_{\Gl(F)}
      \ar[r]
      &
      \Omega_\McE
    \end{tikzcd}
  \end{equation*}
  where the map $\Omega_{\Gl(F)} \to \Omega_\McE$ in the bottom row are the maps from
  from \Cref{eqn:E-Omega-eq}.
\end{lemma}
\begin{proof}
  We check the universal property.
  Suppose there is a map $\varphi \colon e \to \Omega_{\Gl(F)}$ such that the solid
  arrows commute.
  \begin{equation}\label{eqn:e-eqn}\tag{\textsc{$e$-eqn}}
    \begin{tikzcd}[column sep=huge]
      e
      \ar[rrrrd, "!", bend left=10]
      \ar[ddr, "\varphi"', bend right]
      \ar[rd, "!", dashed]
      \\
      &
      1 \cong F1 \ar[d, "\true_{\Gl(F)}"{description}, hookrightarrow]
      \ar[rrr, "="]
      \ar[rd, "{(\true, F(\true))}"{description}]
      &
      &
      &
      1 \ar[d, "\true", hookrightarrow]
      \\
      &
      \Omega_{\Gl(F)}
      \ar[r, hookrightarrow, "i"']
      &
      \Omega_\McE \times F\Omega_\McC
      \ar[r, "\id \times \chi_{F(\true)}"']
      &
      \Omega_\McE \times \Omega_\McE
      \ar[r, "\wedge"']
      &
      \Omega_\McE
    \end{tikzcd}
  \end{equation}
  It suffices to show that $\varphi = \true_{\Gl(F)} \cdot !$.
  Composing \begin{tikzcd}[cramped]
    e \ar[r, "\varphi"] & \Omega_{\Gl(F)} \ar[r, hookrightarrow] & \Omega_\McE \times F\Omega_\McC
    \ar[r, "{\id \times \chi_{F(\true)}}"] &[1em] \Omega_\McE \times \Omega_\McE
  \end{tikzcd}
  with the two projections, one obtains maps
  \begin{equation*}
    \varphi_1 =
    \begin{tikzcd}[cramped]
      e \ar[r, "\varphi"] & \Omega_{\Gl(F)} \ar[r, hookrightarrow] & \Omega_\McE \times F\Omega_\McC
      \ar[r, "\pi_1"] & \Omega_\McE
    \end{tikzcd}
  \end{equation*}
  and
  \begin{equation*}
    \varphi_2 =
    \begin{tikzcd}[cramped]
      e \ar[r, "\varphi"] & \Omega_{\Gl(F)} \ar[r, hookrightarrow] & \Omega_\McE \times F\Omega_\McC
      \ar[r, "\pi_2"] & F\Omega_\McC \ar[r, "\chi_{F(\true)}"] & \Omega_\McE
    \end{tikzcd}
  \end{equation*}
  But because the bottom row of \Cref{eqn:e-eqn} is also
  $\Omega_{\Gl(F)} \hookrightarrow \Omega_\McE \times F\Omega_\McC \xrightarrow{\pi_1}
  \Omega$, the commutativity of \Cref{eqn:e-eqn} indicates that
  $\varphi_1 = e \xrightarrow{!} 1 \xrightarrow{\true} \Omega_\McE$.
  Hence, the subobject of $e$ obtained by pulling back
  $\true \colon 1 \to \Omega_\McE$ along
  $\varphi_1 = e \xrightarrow{!} ! \xrightarrow{\true} \Omega_\McE$ is exactly
  $\id \colon e \to e$ itself.
  Now, pulling back $\true \colon 1 \to \Omega_\McE$ along $\varphi_2$ gives
  \begin{equation*}
    \begin{tikzcd}
      p \ar[d, hookrightarrow] \ar[rrr] \ar[rrrd, "\lrcorner"{pos=0}, phantom]
      &
      &
      &
      F1 \cong 1
      \ar[d, "F(\true)"{description}, hookrightarrow]
      \ar[r, "\cong"]
      \ar[rd, "\lrcorner"{pos=0}, phantom]
      &
      1
      \ar[d, "\true", hookrightarrow]
      \\
      e \ar[r, "\varphi"']
      &
      \Omega_{\Gl(F)} \ar[r, hookrightarrow, "i"']
      &
      \Omega_\McE \times F\Omega_\McC
      \ar[r, "\pi_2"']
      &
      F\Omega_\McC \ar[r, "\chi_{F(\true)}"']
      &
      \Omega_\McE
    \end{tikzcd}
  \end{equation*}
  And so $\wedge \cdot (\varphi_1, \varphi_2) \colon e \to \Omega_\McE$ is the
  characteristic map for the fibre product $p \times_e e \hookrightarrow e$.
  But also by \Cref{eqn:e-eqn}, one has
  $\wedge \cdot (\varphi_1, \varphi_2) = e \xrightarrow{!} 1 \xrightarrow{\true}
  \Omega_\McE$, which is the characteristic map for $\id \colon e \to e$.
  Therefore, $p \times_e e \hookrightarrow e$ is the identity.
  In particular, this means $p \hookrightarrow e$ is the identity.

  Hence, we have established
  $\pi_2 \cdot i \cdot \varphi = F(\true) \cdot ! = \pi_2 \cdot (\true,
  F(\true)) \cdot ! \colon e \rightrightarrows F\Omega_\McC$ and
  $\pi_1 \cdot i \cdot \varphi = \true \cdot ! = \pi_1 \cdot (\true, F(\true))
  \cdot ! \colon e \rightrightarrows \Omega_\McE$ for $\pi_1,\pi_2$ the
  projection maps of $\Omega_\McE \times F\Omega_\McC$.
  This means that $i \cdot \varphi = (\true, F(\true)) \cdot ! = i \cdot \true_{\Gl(F)} \cdot !$.
  But $i$ is a mono, so $\varphi = \true_{\Gl(F)} \cdot !$.

  Clearly, $! \colon e \to 1$ is the unique dashed map that makes all of
  \Cref{eqn:e-eqn} commutes, so the result follows.
\end{proof}

Now, we construct the indicator map in the following \Cref{constr:gl-indicator}
and verify in \Cref{thm:glue-Omega} that the constructed indicator map along
with the subobject classifier and truth map in \Cref{constr:glue-true} does
indeed have their requisite logical properties.
\begin{construction}\label{constr:gl-indicator}
  Suppose one has a monomorphism
  $(g \colon b \hookrightarrow a \in \McE, k \colon y \hookrightarrow x \in
  \McC) \in \Gl(F)$
  \begin{equation*}
    \begin{tikzcd}
      b \ar[r, dashed, hookrightarrow, "g"] \ar[d, "\beta"']
      &
      a \ar[d, "\alpha"]
      \\
      Fy \ar[r, dashed, hookrightarrow, "Fk"']
      &
      Fx
    \end{tikzcd}
  \end{equation*}
  Then, one has indicator $\chi_g \colon a \to \Omega_\McE$ and
  $\chi_k \colon x \to \Omega_\McC$.
  This gives rise to a cospan
  $F\Omega_\McC \xleftarrow{F\chi_k} Fx \xleftarrow{\alpha} a
  \xrightarrow{\chi_g} \Omega_\McE$ that induces a map
  $\chi_\beta \colon a \to \Omega_{\Gl(F)}$.
\end{construction}

\begin{lemma}\label{lem:gl-indicator}
  The map
  $(\chi_g, F\chi_k \cdot \alpha) \colon a \to \Omega_\McE \times F\Omega_\McC$
  from \Cref{constr:gl-indicator} equalises
  $\wedge \cdot \id \times \chi_{F(\true)}, \pi_1 \colon \Omega_\McE \times
  F\Omega_\McC \rightrightarrows \Omega_\McE$ so that the map
  $\chi_\beta \colon a \to \Omega_{\Gl(F)}$ from \Cref{constr:gl-indicator} exists.

  Furthermore, the following diagrams commute
  \begin{equation*}
    \begin{tikzcd}
      a \ar[d, "\chi_\beta"'] \ar[rr, "\alpha"]
      &
      &
      Fx \ar[d, "F\chi_k"]
      \\
      \Omega_{\Gl(F)} \ar[r, "i"', hook]
      &
      \Omega_\McE \times F\Omega_\McC
      \ar[r, "\pi_2"']
      &
      F\Omega_\McC
    \end{tikzcd}
    \hspace{5em}
    \begin{tikzcd}
      b \ar[r, "!"] \ar[d, "g"'] & 1 \ar[d, "\true_{\Gl(F)}"]
      \\
      a \ar[r, "\chi_\beta"'] & \Omega_{\Gl(F)}
    \end{tikzcd}
  \end{equation*}
\end{lemma}
\begin{proof}
  Clearly,
  $\pi_1 \cdot (\chi_g, \alpha \cdot F_{\chi_k}) = \chi_g \colon a \to
  \Omega_\McE$ is the characteristic map for $g \colon b \hookrightarrow a$.
  And
  $\wedge \cdot (\id \times \chi_{F(\true)}) \cdot (\chi_g, F_{\chi_k} \cdot
  \alpha) = \wedge \cdot (\chi_g, \chi_{F(\true)} \cdot F_{\chi_k} \cdot \alpha)
  \colon a \to \Omega_\McE \times \Omega_\McE \to \Omega_\McE$.
  %
  \begin{equation*}
    \begin{tikzcd}
      b &&& 1 \\
      & {a \times_{Fx} Fy} \\
      && Fy &&& F1 &&& 1 \\
      a &&& {\Omega_\McE} \\
      & a &&& {\Omega_\McE} \\
      && Fx &&& {F\Omega_\McC} &&& {\Omega_\McE}
      \arrow["\alpha"', from=5-2, to=6-3]
      \arrow["{\chi_g}"{description, pos=0.6}, from=5-2, to=5-5]
      \arrow["{F\chi_k}"', from=6-3, to=6-6]
      \arrow["{\chi_{F(\true)}}"', from=6-6, to=6-9]
      \arrow["\true", from=3-9, to=6-9]
      \arrow["{F(\true)}"{description}, hook, from=3-6, to=6-6]
      \arrow["\cong"{description}, from=3-6, to=3-9]
      \arrow[from=3-3, to=3-6]
      \arrow[from=2-2, to=3-3]
      \arrow["{=}"{description}, from=4-1, to=5-2]
      \arrow["{=}"{description}, from=4-4, to=5-5]
      \arrow["{\chi_g}"{description}, from=4-1, to=4-4]
      \arrow["g"{description}, hook, from=1-1, to=4-1]
      \arrow["\true"{pos=0.3}, from=1-4, to=4-4]
      \arrow["{!}", from=1-1, to=1-4]
      \arrow["\lrcorner"{pos=0}, phantom, from=3-6, to=6-9]
      \arrow["\lrcorner"{pos=0}, phantom, from=3-3, to=6-6]
      \arrow["\lrcorner"{pos=0, rotate=-45}, phantom, from=2-2, to=6-3]
      \arrow["\lrcorner"{pos=0}, phantom, from=1-1, to=4-4]
      \arrow[hook, from=2-2, to=5-2, crossing over]
      \arrow["Fk"{description}, hook, from=3-3, to=6-3, crossing over]
      \arrow["{(g,\beta)}"{description}, from=1-1, to=2-2]
    \end{tikzcd}
  \end{equation*}
  But because $g \colon b \hookrightarrow a$ is a mono, the map
  $(g, \beta) \colon b \hookrightarrow a \times_{Fx} Fy$ induced by
  $\alpha g = (Fk) \beta$ is a mono.
  Therefore, the fibre product $b \times_a (a \times_{Fx} Fy) \to a$ is
  $g \colon b \hookrightarrow a$.
  In other words, pulling back $\true \colon 1 \to \Omega_\McE$ along
  $\wedge \cdot (\chi_g, \chi_{F(\true)} \cdot F_{\chi_k} \cdot \alpha) \colon a
  \to \Omega_\McE \times \Omega_\McE \to \Omega_\McE$ gives
  $g \colon b \hookrightarrow a$, which is also the pullback of $\true$ along
  $\chi_g$.
  This shows that
  $\pi_1 \cdot (\chi_g, \alpha \cdot F_{\chi_k}) = \chi_g = \wedge \cdot (\id
  \times \chi_{F(\true)}) \cdot (\chi_g, F_{\chi_k} \cdot \alpha) \colon a \to
  \Omega_\McE \times F\Omega_\McC \rightrightarrows \Omega_\McE$.

  Also, note that by construction,
  $i \cdot \chi_\beta = (\chi_g, F_{\chi_k} \cdot \alpha)$ and so
  $\pi_2 \cdot i \cdot \chi_\beta = F_{\chi_k} \cdot \alpha$.
  Finally, note that
  $i \cdot \true_{\Gl(F)} \cdot ! = i \cdot \chi_\beta \cdot g \colon b
  \rightrightarrows \Omega_{\Gl(F)} \hookrightarrow \Omega_\McE \times F\Omega_\McC$.
  This is because
  $i \cdot \true_{\Gl(F)} = (\true, F(\true)) \colon 1 \to \Omega_\McE \times
  F\Omega_\McC$ while
  $i \cdot \chi_\beta = (\chi_g, F_{\chi_k} \cdot \alpha) \colon a \to
  \Omega_\McE \times F\Omega_\McC$.
  And so
  $\pi_1 \cdot i \cdot \true_{\Gl(F)} \cdot ! = \true \cdot ! = \chi_g \cdot g =
  \pi_1 \cdot i \cdot \chi_\beta \cdot g$ for
  $\pi_1 \colon \Omega_\McE \times F\Omega_\McC \to \Omega_\McE$.
  And for $\pi_2 \colon \Omega_\McE \times F\Omega_\McC \to \Omega_\McE$, one has
  $\pi_2 \cdot i \cdot \true_{\Gl(F)} \cdot ! = F(\true) \cdot ! =
  F_{\chi_k} \cdot \alpha \cdot g = \pi_2 \cdot i \cdot \chi_\beta \cdot g$ because
  \begin{equation*}
    \begin{tikzcd}
      b &&& {F1 \cong 1} \\
      & Fy \\
      a &&& {F\Omega_\McC} \\
      & Fx
      \arrow["g"{description}, hook, from=1-1, to=3-1]
      \arrow["\alpha"{description}, from=3-1, to=4-2]
      \arrow["\beta"{description}, from=1-1, to=2-2]
      \arrow["{F\chi_k}"{description}, from=4-2, to=3-4]
      \arrow["{F_{\chi_k} \cdot \alpha}"{description, pos=0.6}, from=3-1, to=3-4]
      \arrow["{F(\true)}"{description}, from=1-4, to=3-4]
      \arrow["{!}"{description}, from=1-1, to=1-4]
      \arrow["\lrcorner"{pos=0}, phantom, from=1-1, to=3-4]
      \arrow["{!}"{description}, from=2-2, to=1-4]
      \arrow["Fk"{description, pos=0.3}, hook, from=2-2, to=4-2, crossing over]
    \end{tikzcd}
  \end{equation*}
\end{proof}

\begin{theorem}\label{thm:glue-Omega}
  Under \Cref{asm:F-omega-cond}, the gluing category $\Gl(F)$ has a subobject
  classifier and a truth map, given as in \Cref{eqn:gl-Omega}.
  Explicitly, \Cref{asm:F-omega-cond} requires that $\McC,\McE$ to be equipped
  wit subobject classifiers, respectively $\Omega_\McC$ and $\McE$, and that
  $\McE$ admits all finite limits while $F$ is left exact.

  In this case, the bottom row of \Cref{eqn:gl-Omega} gives the subobject
  classifier in $\Gl(F)$ and the vertical maps gives the canonical truth map.
  Given a monomorphism
  \begin{equation*}
    \begin{tikzcd}
      b \ar[r, dashed, hookrightarrow, "g"] \ar[d, "\beta"']
      &
      a \ar[d, "\alpha"]
      \\
      Fy \ar[r, dashed, hookrightarrow, "Fk"']
      &
      Fx
    \end{tikzcd} \in \Gl(F)
  \end{equation*}
  its characteristic map is given by the dashed maps
  \begin{equation}\label{eqn:chi-glue}\tag{\textsc{$\chi$-glue}}
    \begin{tikzcd}
      a \ar[d, "\chi_\beta"', dashed] \ar[rr, "\alpha"]
      &
      &
      Fx \ar[d, "F\chi_k", dashed]
      \\
      \Omega_{\Gl(F)} \ar[r, "i"', hook]
      &
      \Omega_\McE \times F\Omega_\McC
      \ar[r, "\pi_2"']
      &
      F\Omega_\McC
    \end{tikzcd}
  \end{equation}
  where $\chi_\beta$ is from \Cref{constr:gl-indicator}.
\end{theorem}
\begin{proof}
  By \Cref{lem:gl-indicator}, \Cref{eqn:chi-glue} commutes.
  It is moreover a pullback because by \Cref{lem:true-pb}, the right square
  below is a pullback and by the definition of $\chi_g$, the large rectangle
  below is a pullback.
  \begin{equation}\label{eqn:chi-beta-pb}\tag{\textsc{$\chi_\beta$-pb}}
    \begin{tikzcd}
      b \ar[r, "!"] \ar[d, "g"', hook]
      \ar[rd, "\lrcorner"{pos=0}, phantom]
      &
      1 \cong F1 \ar[d, "\true_{\Gl(F)}"{description}, hookrightarrow]
      \ar[rr, "="]
      \ar[rd, "\lrcorner"{pos=0}, phantom]
      &
      &
      1 \ar[d, "\true", hookrightarrow]
      \\
      a \ar[r, "\chi_\beta"] \ar[rrr, "\chi_g"', bend right=20]
      &
      \Omega_{\Gl(F)}
      \ar[r, hookrightarrow, "i"]
      &
      \Omega_\McE \times F\Omega_\McC
      \ar[r, "\pi_1"]
      &
      \Omega_\McE
    \end{tikzcd}
  \end{equation}
  Therefore, we have
  \begin{equation}\label{eqn:pb-glue}\tag{\textsc{pb-glue}}
    \begin{tikzcd}[row sep=large]
      b && 1 \\
      a && {\Omega_{\Gl(F)}} \\
      & Fy && {F1 \cong 1} \\
      & Fx && {F\Omega_\McC}
      \arrow["{F\chi_k}"{description}, from=4-2, to=4-4]
      \arrow["Fk"{description}, hook, from=3-2, to=4-2]
      \arrow["\true"{description}, from=3-4, to=4-4]
      \arrow["{\chi_\beta}"{description}, from=2-1, to=2-3]
      \arrow["\alpha"{description}, from=2-1, to=4-2]
      \arrow["g"{description}, hook, from=1-1, to=2-1]
      \arrow["{\true_{\Gl(F)}}"{description}, from=1-3, to=2-3]
      \arrow["{!}"{description}, from=1-1, to=1-3]
      \arrow[from=2-3, to=4-4]
      \arrow["{=}"{description}, from=1-3, to=3-4]
      \arrow["\beta"{description, pos=0.7}, from=1-1, to=3-2, crossing over]
      \arrow["{!}"{description}, from=3-2, to=3-4, crossing over]
      \arrow["\lrcorner"{pos=0}, phantom, from=1-1, to=2-3]
      \arrow["\lrcorner"{pos=0}, phantom, from=3-2, to=4-4]
    \end{tikzcd}
  \end{equation}
  where $\beta$ the unique map between the pullbacks.

  Now, suppose there are maps $\overline{\chi_\beta} \colon a \to \Omega_{\Gl(F)}$ and
  $\overline{\chi_k} \colon x \to \Omega_\McC$ such that $\chi_\beta$ replaced
  with $\overline{\chi_\beta}$ and $F\chi_k$ replaced with $F\overline{\chi_k}$
  in \Cref{eqn:pb-glue} still gives rise to the back face being a pullback in
  $\McE$ and the front face being the image of a pullback in $\McC$ under $F$.
  Then, $k \colon y \hookrightarrow x$ is a pullback of
  $\true \colon 1 \hookrightarrow \Omega_\McC$ along
  $\overline{\chi_k} \colon x \to \Omega_\McC$, which means
  $\overline{\chi_k} = \chi_k$.
  Also, by \Cref{eqn:chi-beta-pb}, $g \colon b \hookrightarrow a$ would be a
  pullback of $\true \colon 1 \hookrightarrow \Omega_\McE$ along
  $a \xrightarrow{\overline{\chi_\beta}} \Omega_{\Gl(F)} \xrightarrow{i}
  \Omega_\McE \times F\Omega_\McC \xrightarrow{\pi_2} \Omega_\McE$.
  Hence, $\pi_1 \cdot i \cdot \overline{\chi_\beta} = \chi_g$.
  And because the bottom face of \Cref{eqn:pb-glue} commutes, this means that
  $\pi_2 \cdot i \cdot \overline{\chi_\beta} = F\overline{\chi_k} \cdot \alpha =
  F\chi_k \cdot \alpha$.
  Hence,
  $i \cdot \overline{\chi_\beta} = (\chi_g, F\chi_k \cdot \alpha) = i \cdot
  \chi_\beta$, which is to say $\chi_\beta = \overline{\chi_\beta}$.
  Thus, the characteristic map is unique.
\end{proof}

\subsection{Dependent Products}
In this section, we construct dependent products in the gluing category
$\Gl(F)$.
For motivation, we first consider the construction of dependent products in the
category of $\Set$-valued $\spncat$s.

\begin{example}\label{ex:glu-Pi-ex}
  Fix a map of $\spncat$s $k \colon Y \to X \in \Set^{\spncat}$ and a map
  $f \colon Z \to Y \in \Set^{\spncat}/Y$ so that the goal is to construct
  the dependent product $\Pi_Yf \colon \Pi_YZ \to X \in \Set^{\spncat}/X$.
  %
  \begin{equation*}
    \begin{tikzcd}
      {Z_0} \\
      \\
      & {Z_2} \\
      {Y_0} &&& {X_0} \\
      && {Z_1} \\
      & {Y_2} &&& {X_2} \\
      \\
      && {Y_1} &&& {X_1}
      \arrow["{Z_{20}}"{description}, from=3-2, to=1-1]
      \arrow["{f_0}"{description, pos=0.7}, from=4-1, to=4-4]
      \arrow["{Y_{20}}", from=6-2, to=4-1]
      \arrow["{f_2}"{description}, from=6-2, to=6-5]
      \arrow["{Y_{21}}"', from=6-2, to=8-3]
      \arrow["{X_{20}}"', from=6-5, to=4-4]
      \arrow["{X_{21}}", from=6-5, to=8-6]
      \arrow["{f_1}"{description}, from=8-3, to=8-6]
      \arrow["{k_0}"{description}, from=1-1, to=4-1]
      \arrow["{s_0}"{description}, crossing over, bend left, dashed, from=4-1, to=1-1]
      \arrow["{k_2}"{description}, crossing over, from=3-2, to=6-2]
      \arrow["{s_2}"{description}, crossing over, bend left, dashed, from=6-2, to=3-2]
      \arrow["{k_1}"{description}, crossing over, from=5-3, to=8-3]
      \arrow["{s_1}"{description}, crossing over, bend left, dashed, from=8-3, to=5-3]
      \arrow["{Z_{21}}"{description, pos=0.7}, crossing over, from=3-2, to=5-3]
    \end{tikzcd}
  \end{equation*}
  %
  %
  Examining the component at 0 of the bundle $\Pi_Yf \colon \Pi_YZ \to X$, one
  observes that at each $x_0 \in X_0$, the fibre of
  $(\Pi_Yf)_0 \colon (\Pi_YZ)_0 \to X_0$ consists of those sections $s_0$ of
  $k_0 \colon Y_0 \to Z_0$ restricted to $f_0^{-1}x_0 \hookrightarrow Y_0$.
  Likewise, the fibre at each $x_1 \in X_1$ of the component at 1 of the bundle
  $\Pi_Yf \colon \Pi_YZ \to X$ consists of those sections $s_1$ of
  $k_1 \colon Y_1 \to Z_1$ restricted to $f_1^{-1}x_1 \hookrightarrow Y_1$.
  In other words, $(\Pi_YZ)_0 = \Pi_{Y_0}Z_0$ and $(\Pi_YZ)_1 = \Pi_{Y_1}Z_1$.

  In a similar vein, the fibre at each $x_2 \in X_2$ of the component at 2 of
  the bundle $\Pi_yf \colon \Pi_YZ \to X$ is a section $s_2$ of
  $k_2 \colon Z_2 \to Y_2$ restricted to $f_2^{-1}x_2 \hookrightarrow Y_2$.
  However, by the above descriptions for $(\Pi_YZ)_0$ and $(\Pi_YZ)_1$, the
  functorial actions $(\Pi_YZ)_{20} \colon (\Pi_YZ)_2 \to (\Pi_YZ)_0$ and
  $(\Pi_YZ)_{21} \colon (\Pi_YZ)_2 \to (\Pi_YZ)_1$ must send $s_2$ to certain
  sections of $k_0 \colon Z_0 \to Y_0$ and $k_1 \colon Z_1 \to Y_1$ fibred along
  $f_0 \colon Y_0 \to X_0$ and $f_1 \colon Y_1 \to X_1$ respectively.
  Naturality dictates that $s_0 \coloneqq (\Pi_YZ)_{20}s_2$ must lie in the
  fibre of $(\Pi_Yf)_0 \colon (\Pi_YZ)_0 \to X_0$ over $X_{20}x_2 \in X_0$ and
  likewise $s_1 \coloneqq (\Pi_YZ)_{20}s_2$ must lie in the fibre of
  $(\Pi_Yf)_0 \colon (\Pi_YZ)_0 \to X_0$ over $X_{20}x_2 \in X_0$.
  Naturality further requires that the left faces of the above diagram commute,
  in that
  $Z_{20} \cdot s_2 = s_0 \cdot Y_{20} \colon f_2^{-1}x_2 \rightrightarrows Z_0$
  and
  $Z_{21} \cdot s_2 = s_1 \cdot Y_{21} \colon f_2^{-1}x_2 \rightrightarrows
  Z_1$.

  Equivalently, the fibre over each $x_2 \in X_2$ of the component at 2 of the
  bundle $\Pi_Yf \colon \Pi_YZ \to X$ is a section $s_2$ of the bundle
  \begin{align*}
    \Pi_{Y_2}k_2 \colon \Pi_{Y_2}Z_2 \to X_2
  \end{align*}
  over $x_2$ along with a pair of sections $(s_0,s_1)$ of the bundle
  $\Pi_{Y_1}k_1 \times \Pi_{Y_2}k_2 \colon \Pi_{Y_1}Z_1 \times \Pi_{Y_0}Z_0 \to
  X_1 \times X_0$ over $(X_{20} \times X_{21})x_2$ subject to naturality
  constraints.
  Internalising, one sees that $(s_0,s_1)$ is a fibre over $x_2$ of the bundle
  \begin{align*}
   X_2 \times_{(X_0 \times X_1)} (\Pi_{Y_0}Z_0 \times \Pi_{Y_1}Z_1) \to X_2
  \end{align*}
  The naturality constraint for commutativity of the left face is equivalently
  expressed by requiring that
  $(Z_{20}, Z_{21}) \cdot s_2 = (s_0, s_1) \cdot (Y_{20}, Y_{21}) \colon
  f_2^{-1}x_2 \rightrightarrows Z_1 \times Z_0$.
  These two (equal) compositions are sections of
  $Y_2 \times_{(Y_0 \times Y_1)} (Z_0 \times Z_1) \to Y_2$ fibred along
  $f_2 \colon Y_2 \to X_2$, so they can be internalised as fibres of
  \begin{align*}
    \Pi_{Y_2}(Y_2 \times_{(Y_0 \times Y_1)} (Z_0 \times Z_1)) \to X_2
  \end{align*}
  over $x_2$.
  The naturality condition
  $(Z_{20}, Z_{21}) \cdot s_2 = (s_0, s_1) \cdot (Y_{20}, Y_{21}) \colon
  f_2^{-1}x_2 \rightrightarrows Z_1 \times Z_0$ involves post-composing
  $(Z_{20}, Z_{21})$ with $s_2$ and pre-composing $(Y_{20},Y_{21})$ with
  $(s_0,s_1)$.

  Post-composing with $(Z_{20},Z_{21})$ is internalised as the functorial action
  of $\Pi_{Y_2} \colon \Set/Y_2 \to \Set/X_2$ on the map
  $(k_2, Z_{20} \times Z_{21}) \colon Z_2 \to Y_2 \times_{(Y_0 \times Y_1)}(Z_0
  \times Z_1)$.
  To internalise pre-composition with $(Y_{20},Y_{21})$, first note that the
  fibre $(s_0,s_1)$ over $x_2$ of the bundle
  $X_2 \times_{(X_0 \times X_1)}(\Pi_{Y_0}Z_0 \times \Pi_{Y_1}Z_1) \to X_2$
  canonically gives rise to a fibre also over $x_2$ of the bundle
  $X_2 \times_{(X_0 \times X_1)} \Pi_{(Y_0 \times Y_1)}(Z_0 \times Z_1) \to X_2$
  by pointwise application:
  each $(p_0,p_1) \in (f_0,f_1)^{-1}(X_{20},X_{21})x_2$ is mapped
  to $(s_0p_0,s_1p_1) \in (k_0,k_1)^{-1}(p_0,p_1)$.
  On the other hand, by the adjunction between fibred sections and fibred
  product, there is the evaluation counit
  $\ev\relax \colon (Y_0 \times Y_1) \times_{(X_0 \times X_1)} \Pi_{(Y_0 \times
    Y_1)}(Z_0 \times Z_1) \to Z_0 \times Z_1$ over $Y_0 \times Y_1$.
  Pulling back this evaluation counit $\ev\relax$ along
  $(Y_{20},Y_{21}) \colon Y_2 \to Y_{20} \times Y_{21}$ then gives a map
  $(Y_{20},Y_{21})^*\ev\relax \colon (Y_{20},Y_{21})^*((Y_0 \times Y_1)
  \times_{(X_0 \times X_1)} \Pi_{(Y_0 \times Y_1)}(Z_0 \times Z_1)) \to
  (Y_{20},Y_{21})^*(Z_0 \times Z_1)$ over $Y_2$.
  The fibres of
  $(Y_{20},Y_{21})^*((Y_0 \times Y_1) \times_{(X_0 \times X_1)} \Pi_{(Y_0 \times
    Y_1)}(Z_0 \times Z_1))$ over each $y_2 \in Y_2$ is a section of
  $(k_0,k_1) \colon Z_0 \times Z_1 \to Y_0 \times Y_1$ around the neighbourhood
  $(f_0,f_1)^{-1}(f_0,f_1)(Y_{20},Y_{21})y_2$ and
  $(Y_{20},Y_{21})^*\ev\relax$ evaluates this section at $(Y_{20},Y_{21})y_2$.
  But note that
  \begin{align*}
    (Y_{20},Y_{21})^*((Y_0 \times Y_1) \times_{(X_0 \times X_1)} \Pi_{(Y_0 \times
    Y_1)}(Z_0 \times Z_1))
    &=
    (Y_{20},Y_{21})^*(f_0,f_1)^*\Pi_{(Y_0 \times Y_1)}(Z_0 \times Z_1)
    \\
    &=
    f_2^*(X_{20},X_{21})^*\Pi_{(Y_0 \times Y_1)}(Z_0 \times Z_1)
    \\
    &=
    f_2^*(X_2 \times_{(X_0 \times X_1)} \Pi_{(Y_0 \times Y_1)}(Z_0 \times Z_1))
  \end{align*}
  and
  $(Y_{20}, Y_{21})^*(Z_0 \times Z_1) = Y_2 \times_{(Y_0 \times Y_1)} (Z_0
  \times Z_1)$.
  Thus, under the adjunction $f_2^* \dashv \Pi_{Y_2}$, the map
  $(Y_{20},Y_{21})^*\ev\relax$ transposes to a map
  \begin{align*}
    ((Y_{20},Y_{21})^*\ev\relax)^\ddagger \colon (Y_0 \times Y_1) \times_{(X_0
    \times X_1)} \Pi_{(Y_0 \times Y_1)}(Z_0 \times Z_1) \to \Pi_{Y_2}(Y_2
    \times_{(Y_0 \times Y_1)} (Z_0 \times Z_1))
  \end{align*}
  that internalises pre-composition with $(Y_{20},Y_{21})$.

  Putting everything together, one may therefore deduce that $(\Pi_YZ)_2$ is the
  pullback over $X_2$
  \begin{equation*}
    \begin{tikzcd}
      (\Pi_YZ)_2
      \ar[rr]
      \ar[d]
      \ar[rd, "\lrcorner"{pos=0}, phantom]
      &
      &
      \Pi_{Y_2}Z_2
      \ar[d]
      \\
      X_2 \times_{(X_0 \times X_1)} (\Pi_{Y_0}Z_0 \times \Pi_{Y_1}Z_1)
      \ar[r]
      &
      X_2 \times_{(X_0 \times X_1)} \Pi_{(Y_0 \times Y_1)}(Z_0 \times Z_1)
      \ar[r]
      &
      \Pi_{Y_2}(Y_2 \times_{(Y_0 \times Y_1)} (Z_0 \times Z_1))
    \end{tikzcd}
  \end{equation*}
  and the functorial actions to $(\Pi_YZ)_0 = \Pi_{Y_0}Z_0$ and
  $(\Pi_YZ)_1 = \Pi_{Y_1}Z_1$ are induced by
  $(\Pi_YZ)_2 \to X_2 \times_{(X_0 \times X_1)} (\Pi_{Y_0}Z_0 \times
  \Pi_{Y_1}Z_1)$.
  %
  %
  %
  %
\end{example}

We now proceed to generalise the observation made in \Cref{ex:glu-Pi-ex} to
arbitrary gluing categories $\Gl(F \colon \McC \to \McE)$.
%
%
%
%
Fix an object $(g \colon b \to a \in \McE, k \colon y \to x \in \McC) \in \Gl(F)$
Also, fix a map $(f,h)$ in $\Gl(F)$ as below:
\begin{equation*}
  \begin{tikzcd}
    |[alias=c]| c
    \\
    & |[alias=Fz]| Fz
    \\
    |[alias=b]| b
    & & |[alias=a]| a
    \\
    & |[alias=Fy]| Fy &
    & |[alias=Fx]| Fx
    \ar[from=c, to=b, "f"']
    \ar[from=b, to=a, "g"{pos=0.25}]
    \ar[from=b, to=Fy, "\beta"']
    \ar[from=a, to=Fx, "\alpha"]
    \ar[from=c, to=Fz, "\gamma"]
    \ar[from=Fz, to=Fy, crossing over, "Fh"{pos=0.75}]
    \ar[from=Fy, to=Fx, "Fk"']
  \end{tikzcd}
\end{equation*}
We further assume
\begin{assumption}\label{asm:F-Pi-cond} $ $
  \begin{itemize}
    \item The maps $g, Fk \in \McE$ and $k \in \McC$ are powerful.
    \item $F$ preserves pullbacks.
  \end{itemize}
\end{assumption}

Under the above assumptions, in the following \Cref{constr:glue-Pi}, we
construct the dependent product $\Pi_\beta\gamma$.
Then, in \Cref{constr:flat}, we construct the map
$\Hom(-,\Pi_\beta\gamma) \to \Hom(- \times_\alpha \beta, \gamma)$ and verify the
correctness of its construction in \Cref{lem:flat}.
The map in other direction
$\Hom(- \times_\alpha \beta, \gamma) \to \Hom(-,\Pi_\beta,\gamma)$ is then
constructed in two steps, with the main ingredients prepared in
\Cref{constr:sharp-components} and requisite properties verified in
\Cref{lem:sharp-components} before assembling them into the actual map
$\Hom(- \times_\alpha \beta, \gamma) \to \Hom(-,\Pi_\beta,\gamma)$ in
\Cref{constr:sharp}.
Finally, in \Cref{lem:sharp-flat-inverses} we check that these two constructions
in \Cref{constr:flat,constr:sharp} are mutual inverses, thus showing the
adjointness of the dependent product, which allows us to conclude the
correctness of our constructions in \Cref{thm:glue-Pi}.

\begin{construction}\label{constr:glue-Pi}
  Define the canonical comparison map $F(\Pi_yz) \to \Pi_{Fy}Fz$
  to correspond, under the transpose $(Fk)^* \dashv \Pi_{Fy}$, to be the map
  \begin{equation*}
    \begin{tikzcd}[column sep=large, row sep=large]
      |[alias=pb-F]| F(\Pi_yz) \times_{Fx} Fy
      &
      |[alias=F-pb]| F(\Pi_yz \times_x y)
      &
      |[alias=Fz]| Fz
      &
      |[alias=F-Pi]| F(\Pi_yz)
      &
      |[alias=Pi-F]| \Pi_{Fy}Fz
      \\
      &
      &
      |[alias=Fy]| Fy
      &
      |[alias=Fx]| Fx
      \ar[from=pb-F, to=F-pb, leftrightarrow, "\cong"{description}]
      \ar[from=F-pb, to=Fz, "F(\ev)"{description}]
      \ar[from=pb-F, to=Fz, bend left, "\theta^\dagger"]
      \ar[from=F-Pi, to=Pi-F, "\theta"]
      \ar[from=Fy, to=Fx, "Fk"']
      \ar[from=pb-F, to=Fy, bend right=10]
      \ar[from=F-pb, to=Fy]
      \ar[from=Fz, to=Fy, "Fh"]
      \ar[from=F-Pi, to=Fx, "F(\Pi_yh)"']
      \ar[from=Pi-F, to=Fx, bend left=10, "\Pi_{Fy}Fh"]
    \end{tikzcd}
  \end{equation*}
  where the isomorphism $F(\Pi_yz) \times_{Fx} Fy \cong F(\Pi_yz \times_x y)$
  over $Fy$ is because $F$ preserves pullbacks.
  For ease of understanding, we denote this map as
  $F(\ev)^\ddagger \colon F(\Pi_yz) \to \Pi_{Fy}Fz$.

  Pulling back $F(\ev)^\ddagger$ along $\alpha$ then gives a map
  $a \times_{Fx} F(\ev)^\ddagger \colon a \times_{Fx} F(\Pi_yz) \to a \times_{Fx}
  \Pi_{Fy}Fz$.
  Over $b$, one has a map
  $b \times_{Fy} \ev \colon b \times_a (a \times_{Fx} \Pi_{Fy}Fz) = b
  \times_{Fy}(Fy \times_{Fx} \Pi_{Fy}Fz) \to b \times_{Fy} Fz$ induced by the
  action of $\beta^*$ on the adjoint $(Fk)^* \dashv \Pi_{Fy}$.
  Transposing $b \times_{Fy} \ev$ along $g^* \dashv \Pi_b$ then gives a map
  $(b \times_{Fy} \ev)^\ddagger \colon a \times_{Fx} \Pi_{Fy}Fz \to \Pi_b(b
  \times_{Fy} Fz)$.
  Taking the composition then gives
  \begin{equation*}
    \begin{tikzcd}[column sep=huge]
      a \times_{Fx} F(\Pi_y z) \ar[r, "a \times_{Fx} F(\ev)^\ddagger", color=green0]
      &
      a \times_{Fx} \Pi_{Fy}Fz \ar[r, "(b \times_{Fy} \ev)^\ddagger", color=yellow0]
      &
      \Pi_b(b \times_{Fy} Fz)
    \end{tikzcd}
  \end{equation*}
  over $a$.

  On the other hand, the cospan \begin{tikzcd}[cramped, column sep=small]
    b & c \ar[l, "f"'] \ar[r, "\gamma"] & Fz
  \end{tikzcd}
  induces a map $(f, \gamma) \colon c \to b \times_{Fy} Fz$, so by functoriality
  of $\Pi_b$, one obtains a map
  \begin{equation*}
    \begin{tikzcd}[column sep=huge]
      \Pi_bc \ar[r, "{\Pi_b(f,\gamma)}", cyan0] & \Pi_b(b \times_{Fy} Fz)
    \end{tikzcd}
  \end{equation*}
  over $a$.
  In summary:
  \begin{equation*}
    \begin{tikzcd}[column sep=small]
      &
      &
      &
      &
      &
      |[alias=Pi-b-c]| \Pi_bc
      &
      |[alias=Pi-b-Fz]| \Pi_b(b \times_{Fy} Fz)
      \\
      |[alias=Pb-Pb-Pi-F]| b \times_a (a \times_{Fx} \Pi_{Fy}Fz)
      &
      |[alias=c]| c
      &
      &
      &
      &
      & |[alias=Pb-Pi-F]| a \times_{Fx} \Pi_{Fy}Fz
      \\
      &
      &
      |[alias=b-Fz]| b \times_{Fy} Fz
      &
      &
      & |[alias=Pb-F-Pi]| a \times_{Fx} F(\Pi_yz)
      &
      &
      \\
      &
      &
      & |[alias=Fz]| Fz
      &
      &
      & |[alias=F-Pi]| F(\Pi_yz)
      &
      \\
      &
      &
      |[alias=b]| b
      &
      &
      & |[alias=a]| a
      &
      &
      \\
      &
      |[alias=Pb2-Pi-F]| Fy \times_{Fx} \Pi_{Fy}Fz
      &
      &
      &
      &
      &
      & |[alias=Pi-F]| \Pi_{Fy}Fz
      \\
      &
      &
      & |[alias=Fy]| Fy
      &
      &
      & |[alias=Fx]| Fx
      \ar[from=c, to=b, "f" {description}, bend right=10]
      \ar[from=c, to=Fz, "\gamma"{description}, bend left=40]
      \ar[from=c, to=b-Fz, "{(f, \gamma)}" {description}, color=cyan0]
      \ar[from=Pi-b-c, to=Pi-b-Fz, "{\Pi_b(f,\gamma)}", cyan0]
      \ar[from=b-Fz, to=b]
      \ar[from=b-Fz, to=Fz]
      \ar[from=b-Fz, to=Fy, phantom, "\lrcorner"{pos=0, rotate=-30}]
      \ar[from=b, to=a, "g"{pos=0.75}]
      \ar[from=b, to=Fy, "\beta"' {pos=0.2, description}]
      \ar[from=a, to=Fx, "\alpha"{pos=0.2, description}]
      \ar[from=Fz, to=Fy, crossing over, "Fh"{pos=0.5, description}]
      \ar[from=Fy, to=Fx, "Fk"']
      \ar[from=F-Pi, to=Fx, "F(\Pi_yh)"'{pos=0.3}]
      \ar[from=Pb-F-Pi, to=a]
      \ar[from=Pb-F-Pi, to=Fx, phantom, "\lrcorner"{pos=0, rotate=-30}]
      \ar[from=Pb-F-Pi, to=Pb-Pi-F, "a \times_{Fx} F(\ev)^\ddagger"{description}, color=green0]
      \ar[from=Pi-F, to=Fx, color=red0, "\Pi_{Fy}Fh"]
      \ar[from=Pb-Pi-F, to=a, color=red0]
      \ar[from=Pb-Pi-F, to=Pi-F, color=red0, bend left=10]
      \ar[from=Pb-Pi-F, to=Fx, phantom, "\lrcorner"{pos=0, rotate=-60}]
      \ar[from=Pb-F-Pi, to=F-Pi, crossing over]
      \ar[from=F-Pi, to=Pi-F, "F(\ev)^\ddagger" {description}, color=green0]
      \ar[from=Pb-Pb-Pi-F, to=b,, color=red0]
      \ar[from=Pb-Pb-Pi-F, to=Pb2-Pi-F, color=red0]
      \ar[from=Pb-Pb-Pi-F, to=Pb-Pi-F, color=red0, bend left=10]
      \ar[from=Pb2-Pi-F, to=Pi-F, crossing over, color=red0]
      \ar[from=Pb2-Pi-F, to=Fy, crossing over, color=red0]
      \ar[from=Pb2-Pi-F, to=Fz, "\ev" {pos=0.3}, crossing over, bend left=20, color=yellow0]
      \ar[from=Pb2-Pi-F, to=Fx, phantom, "\lrcorner"{rotate=0,pos=0,fill=white0}]
      \ar[from=Pb-Pb-Pi-F, to=b-Fz, "b \times_{Fy} \ev"{description}, crossing over, color=yellow0]
      \ar[from=Pb-Pi-F, to=Pi-b-Fz, "(b \times_{Fy} \ev)^\ddagger"', color=yellow0]
      \ar[from=Pb-Pb-Pi-F, to=a, out=-10, phantom, "\lrcorner"{pos=0}]
    \end{tikzcd}
  \end{equation*}
  where all the rectangles with pink edges are pullbacks.

  Therefore, we may take the pullback over $a$:
  \begin{equation}\label{eqn:Pi-pullback}\tag{\textsc{$\Pi$-pullback}}
    \begin{tikzcd}[column sep=huge]
      p \coloneqq \Pi_bc \times_{\Pi_b(b \times_{Fy} Fz)} (a \times_{Fx} F(\Pi_yz))
      \ar[d] \ar[rr]
      \ar[rrd, phantom, "\lrcorner"{pos=0}]
      & &
      \Pi_bc \ar[d, "{\Pi_b(f,\gamma)}", color=cyan0]
      \\
      a \times_{Fx} F(\Pi_y z) \ar[r, "a \times_{Fx} F(\ev)^\ddagger"', color=green0]
      &
      a \times_{Fx} \Pi_{Fy}Fz \ar[r, "(b \times_{Fy} \ev)^\ddagger"', color=yellow0]
      &
      \Pi_b(b \times_{Fy} Fz)
    \end{tikzcd}
  \end{equation}
  and put
  $p = \Pi_bc \times_{\Pi_b(b \times_{Fy} Fz)} (a \times_{Fx} F(\Pi_yz)) \to a
  \times_{Fx} F(\Pi_yz) \to F(\Pi_y z)$ as the sections of
  $\gamma \colon c \to Fz$ fibred over $(g,k)$.
\end{construction}

\begin{construction}\label{constr:flat}
  Suppose now that one has another object $\delta \colon d \to Fw$ over
  $\alpha \colon a \to Fx$ in $\Gl(F)$.
  A map $\delta \to \Pi_\beta\gamma$ over $\alpha$ is a pair of dashed maps
  $(u \colon d \to p, v \colon w \to \Pi_yz)$ as below such that the following
  diagram commutes:
  \begin{equation*}
    \begin{tikzcd}
      |[alias=de]| d
      &
      & |[alias=p]| p
      \\
      & |[alias=Fw]| Fw
      &
      & |[alias=F-Pi]| F(\Pi_yz)
      \\
      & & |[alias=a]| a
      \\
      & & & |[alias=Fx]| Fx
      \ar[from=de, to=p, "u", dashed]
      \ar[from=de, to=Fw, "\delta"{description}]
      \ar[from=de, to=a, "i"', bend right]
      \ar[from=p, to=a]
      \ar[from=p, to=F-Pi, "\Pi_{\beta}\gamma"]
      \ar[from=Fw, to=F-Pi, "Fv" {description, pos=0.25}, crossing over, dashed]
      \ar[from=Fw, to=Fx, "Fj"', bend right, crossing over]
      \ar[from=a, to=Fx, "\alpha"{description}]
      \ar[from=F-Pi, to=Fx, "F(\Pi_yh)"]
    \end{tikzcd}
  \end{equation*}
  By \Cref{eqn:Pi-pullback}, this is equivalent to pairs of maps
  $u_1,u_2$ over $a$ such that
  \begin{equation*}\label{eqn:pi-right}\tag{\textsc{$\Pi$-right}}
    \begin{tikzcd}[column sep=normal]
      &
      &
      &
      &
      &
      |[alias=Pi-b-c]| \Pi_bc
      &
      |[alias=Pi-b-Fz]| \Pi_b(b \times_{Fy} Fz)
      \\
      &
      &
      &
      |[alias=de]| d
      &
      |[alias=p]| p
      &
      &
      |[alias=Pb-Pi-F]| a \times_{Fx} \Pi_{Fy}Fz
      \\
      &
      &
      &
      &
      & |[alias=Pb-F-Pi]| a \times_{Fx} F(\Pi_yz)
      &
      &
      \\
      &
      &
      &
      &
      |[alias=Fw]| Fw
      &
      & |[alias=F-Pi]| F(\Pi_yz)
      &
      \\
      &
      &
      |[alias=b]| b
      &
      &
      & |[alias=a]| a
      &
      &
      \\
      &
      &
      &
      &
      &
      &
      & |[alias=Pi-F]| \Pi_{Fy}Fz
      \\
      &
      &
      &
      |[alias=Fy]| Fy
      &
      &
      & |[alias=Fx]| Fx
      %
      %
      \ar[from=Pi-b-c, to=Pi-b-Fz, "{\Pi_b(f,\gamma)}", cyan0]
      %
      %
      \ar[from=b, to=a, "g"{pos=0.5,description}]
      \ar[from=b, to=Fy, "\beta"' {pos=0.2, description}]
      \ar[from=a, to=Fx, "\alpha"{pos=0.2, description}]
      \ar[from=Fy, to=Fx, "Fk"']
      \ar[from=F-Pi, to=Fx, "F(\Pi_yh)"'{pos=0.3}]
      \ar[from=Pb-F-Pi, to=a]
      \ar[from=Pb-F-Pi, to=Fx, phantom, "\lrcorner"{pos=0, rotate=-30}]
      \ar[from=Pb-F-Pi, to=Pb-Pi-F, "a \times_{Fx} F(\ev)^\ddagger"{description}, color=green0]
      \ar[from=Pi-F, to=Fx, color=red0, "\Pi_{Fy}Fh"]
      %
      \ar[from=Pb-F-Pi, to=F-Pi, crossing over, "F(\Pi_yh)^*\alpha = \Pi_\beta\gamma"{description}]
      \ar[from=F-Pi, to=Pi-F, "F(\ev)^\ddagger" {description}, color=green0]
      %
      %
      %
      %
      \ar[from=Pb-Pi-F, to=Pi-b-Fz, "(b \times_{Fy} \ev)^\ddagger"', color=yellow0]
      %
      %
      \ar[from=p, to=Pi-b-c]
      \ar[from=p, to=Pb-F-Pi]
      \ar[from=p, to=Pi-b-Fz, out=-15, in=190, phantom, "\urcorner"{pos=0, rotate=-45}]
      \ar[from=de, to=p, "u" {description}, dashed]
      \ar[from=de, to=Pi-b-c, "u_2", dashed]
      \ar[from=de, to=Pb-F-Pi, "u_1"{description}, dashed]
      \ar[from=de, to=a, "i"{description}, bend left=10]
      \ar[from=de, to=Fw, "\delta"']
      \ar[from=Fw, to=F-Pi, "Fv" {description, pos=0.7}, crossing over, dashed]
      \ar[from=Fw, to=Fx, "Fj"{description}, bend right=10, crossing over]
    \end{tikzcd}
  \end{equation*}

  The second component $u_2 \colon d \to \Pi_bc$ of $u \colon d \to p$ over $a$
  transposes along $g^* \dashv \Pi_b$ to $u_2^\ddagger \colon b \times_a d \to c$.
  Similarly, $v \colon w \to \Pi_yz$ over $x$ in $\McC$ transposes to
  $v^\dagger \colon y \times_x w \to z$ over $y$ in $\McC$ by the adjunction
  $k^* \dashv \Pi_y$.
  Limits in $\Gl(F)$ are defined componentwise, so the pullback of
  $(i,j) \colon \delta \to \alpha$ along $(g,k) \colon \beta \to \alpha$ is the
  map
  $\delta \times_\alpha \beta \colon b \times_a d \to F(y \times_x w) \cong Fy \times_{Fx}
  Fw$.
  Define the transpose of $(u,v)$ as
  \begin{equation*}
    (u,v)^\flat \coloneqq (u_2^\ddagger, v^\dagger)
  \end{equation*}
\end{construction}

\begin{lemma}\label{lem:flat}
  $(u,v)^\flat$ as above is indeed a map $\delta \times_\alpha \beta \to \gamma$
  over $\beta$ in $\Gl(F)$.
\end{lemma}
\begin{proof}
  To show that $(u,v)^\flat$ as above is indeed a map
  $\delta \times_\alpha \beta \to \gamma$ over $\beta$ is to show that if
  \Cref{eqn:pi-right} commutes then \Cref{eqn:flat-left} below commutes.
  \begin{equation*}\label{eqn:flat-left}\tag{\textsc{$\flat$-left}}
    \begin{tikzcd}[column sep=normal]
      &
      &
      &
      &
      &
      &
      \\
      &
      |[alias=c]| c
      &
      |[alias=b-d]| b \times_a d
      &
      |[alias=de]| d
      &
      &
      &
      \\
      |[alias=b-Fz]| b \times_{Fy} Fz
      &
      &
      &
      &
      &
      &
      &
      \\
      &
      |[alias=Fz]| Fz
      &
      &
      |[alias=F-y-w]| F(y \times_x w)
      &
      |[alias=Fw]| Fw
      &
      &
      &
      \\
      &
      &
      |[alias=b]| b
      &
      &
      &
      |[alias=a]| a
      &
      &
      \\
      &
      &
      &
      &
      &
      &
      &
      \\
      &
      &
      &
      |[alias=Fy]| Fy
      &
      &
      &
      |[alias=Fx]| Fx
      \ar[from=c, to=b, "f" {description}, bend right=10]
      \ar[from=c, to=Fz, "\gamma"{description}]
      \ar[from=c, to=b-Fz, "{(f, \gamma)}" {description}, color=cyan0]
      %
      %
      \ar[from=b-Fz, to=b, bend right]
      \ar[from=b-Fz, to=Fz, "(Fh)^*\beta"{description}]
      \ar[from=b-Fz, to=Fy, phantom, "\lrcorner"{pos=0, rotate=-30}, out=-43]
      \ar[from=b, to=a, "g"{pos=0.5,description}]
      \ar[from=b, to=Fy, "\beta"' {pos=0.2, description}]
      \ar[from=a, to=Fx, "\alpha"{pos=0.2, description}]
      \ar[from=Fz, to=Fy, crossing over, "Fh"{pos=0.5, description}, bend right=20]
      \ar[from=Fy, to=Fx, "Fk"']
      %
      %
      %
      %
      %
      %
      %
      %
      %
      %
      %
      \ar[from=de, to=a, "i"{description}, bend left=20]
      \ar[from=de, to=Fw, "\delta"']
      %
      \ar[from=Fw, to=Fx, "Fj"{description}, bend left=30, crossing over]
      \ar[from=b-d, to=c, "u_2^\ddagger"', dashed]
      \ar[from=b-d, to=de]
      \ar[from=b-d, to=b]
      \ar[from=b-d, to=a, phantom, "\lrcorner"{pos=0}]
      \ar[from=F-y-w, to=Fw]
      \ar[from=F-y-w, to=Fy, crossing over]
      \ar[from=F-y-w, to=Fx, phantom, "\lrcorner"{pos=0}]
      \ar[from=F-y-w, to=Fz, "F(v^\dagger)"{description, pos=0.25}, dashed, crossing over]
      \ar[from=b-d, to=F-y-w, dashed, "{\beta \times_\alpha \delta}"{description}]
    \end{tikzcd}
  \end{equation*}

  In particular, commutativity of \Cref{eqn:flat-left} amounts to
  $F(v^\dagger) \cdot \beta \times_\alpha \delta = \gamma \cdot u_2^\ddagger$ while
  commutativity of \Cref{eqn:pi-right} amounts to
  $\textcolor{yellow0}{(b \times_{Fy} \ev)^\ddagger} \cdot
  \textcolor{green0}{a \times_{Fx} F(\ev)^\ddagger} \cdot u_1 = \textcolor{cyan0}{\Pi_b(f,\gamma)}
  \cdot u_2$ and $Fv \cdot \delta = F(\Pi_yh)^*\alpha \cdot u_1$.

  Taking the transpose of
  \begin{equation*}
    \begin{tikzcd}[column sep=huge]
      d \ar[r, "u_1", dashed]
      &
      a \times_{Fx} F(\Pi_y z) \ar[r, "a \times_{Fx} F(\ev)^\ddagger", color=green0]
      &
      a \times_{Fx} \Pi_{Fy}Fz \ar[r, "(b \times_{Fy} \ev)^\ddagger", color=yellow0]
      &
      \Pi_b(b \times_{Fy} Fz)
    \end{tikzcd}
  \end{equation*}
  over $a$ under the adjunction $g^* \dashv \Pi_b$ then further composing with the
  projection $(Fh)^*\beta \colon b \times_{Fy} Fz \to Fz$ yields
  %
  \begin{equation*}
    \begin{tikzcd}[every label/.append style = {font = \tiny}, column sep=huge]
      |[alias=b-d]| b \times_a d
      &
      |[alias=Pb-Pb-F-Pi]| b \times_a (a \times_{Fx} F(\Pi_yz))
      &
      |[alias=Pb-Pb-Pi-F]| b \times_a (a \times_{Fx} \Pi_{Fy}Fz)
      &
      |[alias=Pb-Pi-b-Fz]| b \times_a \Pi_b(b \times_{Fy} Fz)
      \\
      &
      |[alias=2-Pb-Pb-F-Pi]| b \times_{Fy} (Fy \times_{Fx} F(\Pi_yz))
      &
      |[alias=2-Pb-Pb-Pi-F]| b \times_{Fy} (Fy \times_{Fx} \Pi_{Fy}Fz)
      &
      |[alias=b-Fz]| b \times_{Fy} Fz
      \\
      &
      |[alias=2-Pb-F-Pi]| Fy \times_{Fx} F(\Pi_yz)
      &
      |[alias=Pb-F-y-Pi]| b \times_{Fy} F(y \times_x \Pi_yz)
      &
      |[alias=Fz]| Fz
      \\
      & &
      |[alias=F-y-Pi]| F(y \times_x \Pi_yz)
      \ar[from=b-d, to=Pb-Pb-F-Pi, "b \times_a u_1", dashed]
      \ar[from=Pb-Pb-F-Pi, to=Pb-Pb-Pi-F, "b \times_a (a \times_{Fx} F(\ev)^\ddagger)", color=green0]
      \ar[from=Pb-Pb-Pi-F, to=Pb-Pi-b-Fz, "b \times_a (b \times_{Fy} \ev)^\ddagger", color=yellow0]
      \ar[from=Pb-Pi-b-Fz, to=b-Fz, "\ev"]
      \ar[from=Pb-Pb-Pi-F, to=b-Fz, "b \times_{Fy} \ev"{description}]
      \ar[from=Pb-Pb-Pi-F, to=2-Pb-Pb-Pi-F, equal]
      \ar[from=Pb-Pb-F-Pi, to=2-Pb-Pb-F-Pi, equal]
      \ar[from=2-Pb-Pb-F-Pi, to=2-Pb-Pb-Pi-F, "{\scriptsize b \times_{Fy} (Fy \times_{Fx} F(\ev)^\ddagger)}"]
      \ar[from=2-Pb-Pb-F-Pi, to=Pb-F-y-Pi, "b \times_{Fy} \cong" {description}]
      \ar[from=Pb-F-y-Pi, to=b-Fz, "b \times_{Fy} F(\ev)" {description}]
      \ar[from=2-Pb-Pb-F-Pi, to=2-Pb-F-Pi, "(Fy \times_{Fx} F(\Pi_yh))^*\beta"']
      \ar[from=Pb-F-y-Pi, to=F-y-Pi, "F(y \times \Pi_yh)^*\beta" {description}]
      \ar[from=b-Fz, to=Fz, "(Fh)^*\beta"]
      \ar[from=2-Pb-F-Pi, to=F-y-Pi, "\cong"']
      \ar[from=F-y-Pi, to=Fz, "F(\ev)"']
    \end{tikzcd}
  \end{equation*}
  Next, computing the transpose of
  \begin{equation*}
    \begin{tikzcd}[column sep=huge]
      d \ar[r, "u_2", dashed]
      & \Pi_bc \ar[r, "{\Pi_b(f,\gamma)}", color=cyan0]
      & \Pi_b(b \times_{Fy} Fz)
    \end{tikzcd}
  \end{equation*}
  over $a$ under the adjunction $g^* \dashv \Pi_b$ then further composing with the
  projection $(Fh)^*\beta \colon b \times_{Fy} Fz \to Fz$ gives
  \begin{equation*}
    \begin{tikzcd}[column sep=huge]
      b \times_a d
      \ar[r, "b \times_a u_2"]
      \ar[rd, "u_2^\ddagger"', dashed]
      &
      b \times_a \Pi_bc
      \ar[r, "{b \times \Pi_b(f,\gamma)}"]
      \ar[d, "\ev" {description}]
      &
      b \times_a \Pi_b(b \times_{Fy} Fz)
      \ar[d, "\ev"]
      \\
      &
      c
      \ar[r, "{(f,\gamma)}"{description}, color=cyan0]
      \ar[rr, "\gamma"', bend right=20]
      &
      b \times_{Fy} Fz
      \ar[r, "(Fh)^*\beta"]
      &
      Fz
    \end{tikzcd}
  \end{equation*}

  %
  On the other hand, note that the pullback of
  \begin{tikzcd}[cramped]
    d \ar[r, "u_1", dashed]
    & a \times_{Fx} F(\Pi_yz) \ar[r, "F(\Pi_yh)^*\alpha"]
    &[2em] F(\Pi_yz)
  \end{tikzcd}
  along $Fk \colon Fy \to Fx$ is
  \begin{equation*}
    \begin{tikzcd}[column sep=huge]
      b \times_a d \ar[r, "b \times_a u_1"]
      & b \times_a (a \times_{Fx} F(\Pi_yz)) \ar[r, "(Fy \times_{Fx} F(\Pi_yh))^*\beta"]
      &[2em] Fy \times_{Fx} F(\Pi_yz)
    \end{tikzcd}
  \end{equation*}
  as observed:
  %
  %
  \begin{equation*}
    \begin{tikzcd}[column sep=small]
      &
      &
      |[alias=b-d]| b \times_a d
      &
      &
      |[alias=de]| d
      \\
      \\
      |[alias=Pb-Pb-cong]| b \times_{Fy} F(y \times_x \Pi_yz)
      &
      &
      |[alias=Pb-Pb-F-Pi]| b \times_a (a \times_{Fx} F(\Pi_yz))
      &
      &
      |[alias=Pb-F-Pi-a]| a \times_{Fx} F(\Pi_yz)
      \\
      &
      |[alias=Pb-cong]| F(y \times_x \Pi_y z)
      &
      &
      |[alias=Pb-F-Pi-Fy]| Fy \times_{Fx} F(\Pi_yz)
      &
      &
      |[alias=F-Pi]| F(\Pi_yz)
      \\
      |[alias=b1]| b
      &
      &
      |[alias=b2]| b
      &
      &
      |[alias=a]| a
      \\
      &
      |[alias=Fy-1]| Fy
      &
      &
      |[alias=Fy-2]| Fy
      &
      &
      |[alias=Fx]| Fx
      \ar[from=b-d, to=de]
      \ar[from=b-d, to=Pb-Pb-F-Pi, "b \times_a u_1"']
      \ar[from=b-d, to=Pb-F-Pi, phantom, "\lrcorner"{pos=0}]
      \ar[from=de, to=Pb-F-Pi-a, "u_1"]
      \ar[from=Pb-Pb-cong, to=b1]
      \ar[from=Pb-Pb-cong, to=Pb-cong, "F(y \times_x \Pi_yh)^*\beta" {description}]
      \ar[from=Pb-Pb-F-Pi, to=Pb-F-Pi-a]
      \ar[from=Pb-Pb-F-Pi, to=b2]
      \ar[from=Pb-Pb-F-Pi, to=Pb-Pb-cong, leftrightarrow, "b \times_{Fy} \cong"']
      \ar[from=Pb-Pb-F-Pi, to=Pb-F-Pi-Fy, "(Fy \times_{Fx} F(\Pi_yh))^*\beta"{description}]
      \ar[from=Pb-Pb-F-Pi, to=a, phantom, "\lrcorner"{pos=0}]
      \ar[from=Pb-F-Pi-a, to=a]
      \ar[from=Pb-F-Pi-a, to=F-Pi, "F(\Pi_yh)^*\alpha"]
      \ar[from=Pb-F-Pi-a, to=Fx, phantom, "\lrcorner"{pos=0, rotate=-30}]
      \ar[from=b2, to=a, "g"{description, pos=0.75}]
      \ar[from=b2, to=Fy-2, "\beta"{description}]
      \ar[from=Pb-F-Pi-Fy, to=Pb-cong, leftrightarrow, crossing over, "\cong"{description, pos=0.6}]
      \ar[from=Pb-F-Pi-Fy, to=Fy-2, "Fy \times_{Fx} F(\Pi_yh)" {description, pos=0.25}, crossing over]
      \ar[from=Pb-F-Pi-Fy, to=F-Pi, crossing over]
      \ar[from=Pb-F-Pi-Fy, to=Fx, phantom, "\lrcorner"{pos=0}]
      \ar[from=b1, to=Fy-1, "\beta"']
      \ar[from=b1, to=b2, equal]
      \ar[from=a, to=Fx, "\alpha"{description}]
      \ar[from=F-Pi, to=Fx, "F(\Pi_yh)"]
      \ar[from=Fy-1, to=Fy-2, equal]
      \ar[from=Fy-2, to=Fx, "Fk"']
      \ar[from=Pb-cong, to=Fy-1, "F(y \times_x \Pi_yh)" {description, pos=0.25}, crossing over]
    \end{tikzcd}
  \end{equation*}
  A similar computation shows that the pullback under $Fk$ of
  \begin{tikzcd}[cramped]
    d \ar[r, "\delta"] & Fw \ar[r, "Fv", dashed] & F(\Pi_yz)
  \end{tikzcd}
  along $Fk \colon Fy \to Fx$ is given by
  \begin{equation*}
    \begin{tikzcd}[]
      b \times_a d \ar[r, "{\beta \times_\alpha \delta}", dashed]
      & Fy \times_{Fx} Fw \ar[r, "\cong"]
      & F(y \times_x w) \ar[r, "F(y \times_x v)"]
      &[2em] F(y \times_x \Pi_yz) \ar[r, "\cong"]
      & Fy \times_{Fx} F(\Pi_yz)
    \end{tikzcd}
  \end{equation*}
  as observed:
  \begin{equation*}
    \begin{tikzcd}[column sep=normal]
      &
      &
      |[alias=b-d]| b \times_a d
      &
      &
      |[alias=de]| d
      \\
      &
      |[alias=Pb-Fw-cong]| F(y \times_x w)
      &
      &
      |[alias=Pb-Fw]| Fy \times_{Fx} Fw
      &
      &
      |[alias=Fw]| Fw
      \\
      &
      &
      &
      &
      \\
      &
      |[alias=Pb-cong]| F(y \times_x \Pi_y z)
      &
      &
      |[alias=Pb-F-Pi-Fy]| Fy \times_{Fx} F(\Pi_yz)
      &
      &
      |[alias=F-Pi]| F(\Pi_yz)
      \\
      |[alias=b1]| b
      &
      &
      |[alias=b2]| b
      &
      &
      |[alias=a]| a
      \\
      &
      |[alias=Fy-1]| Fy
      &
      &
      |[alias=Fy-2]| Fy
      &
      &
      |[alias=Fx]| Fx
      %
      \ar[from=b-d, to=de]
      \ar[from=b-d, to=Pb-Fw, "{\beta \times_\alpha \delta}"{description}]
      \ar[from=b-d, to=b2]
      \ar[from=b-d, to=a, phantom, "\lrcorner"{pos=0}]
      \ar[from=de, to=Fw, "\delta"]
      \ar[from=de, to=a, "i"]
      %
      \ar[from=Fw, to=F-Pi, "Fv"]
      \ar[from=Pb-Fw, to=Pb-F-Pi-Fy, "Fy \times_{Fx} Fv"{description}]
      \ar[from=Pb-Fw, to=Fw, crossing over]
      \ar[from=Pb-Fw, to=F-Pi, phantom, "\lrcorner"{pos=0}]
      \ar[from=Pb-Fw, to=Pb-Fw-cong, leftrightarrow, crossing over, "\cong"{description, pos=0.6}]
      \ar[from=Pb-Fw-cong, to=Pb-cong, "F(y \times_x v)"{description}]
      %
      %
      %
      %
      \ar[from=b2, to=a, "g"{description, pos=0.75}]
      \ar[from=b2, to=Fy-2, "\beta"{description}]
      \ar[from=Pb-F-Pi-Fy, to=Pb-cong, leftrightarrow, crossing over, "\cong"{description, pos=0.6}]
      \ar[from=Pb-F-Pi-Fy, to=Fy-2, "Fy \times_{Fx} F(\Pi_yh)" {description, pos=0.25}, crossing over]
      \ar[from=Pb-F-Pi-Fy, to=F-Pi, crossing over]
      \ar[from=Pb-F-Pi-Fy, to=Fx, phantom, "\lrcorner"{pos=0}]
      \ar[from=b1, to=Fy-1, "\beta"']
      \ar[from=b1, to=b2, equal]
      \ar[from=a, to=Fx, "\alpha"{description}]
      \ar[from=F-Pi, to=Fx, "F(\Pi_yh)"]
      \ar[from=Fy-1, to=Fy-2, equal]
      \ar[from=Fy-2, to=Fx, "Fk"']
      \ar[from=Pb-cong, to=Fy-1, "F(y \times_x \Pi_yh)" {description, pos=0.25}, crossing over]
    \end{tikzcd}
  \end{equation*}

  We have assumed \Cref{eqn:pi-right} commutes, so
  \begin{tikzcd}[cramped]
    d \ar[r, "u_1", dashed] &[-1em] a \times_{Fx} F(\Pi_yz) \ar[r, "{F(\Pi_yh)^*\alpha}"] & F(\Pi_yz)
  \end{tikzcd}
  agrees with
  \begin{tikzcd}[cramped, column sep=small]
    d \ar[r, "\delta"] & Fw \ar[r, "Fv", dashed] & F(\Pi_yz)
  \end{tikzcd}.
  Hence, pulling back both maps along $Fk \colon Fy \to Fx$ and then composing
  with $Fy \times_{Fx} F(\Pi_yz) \cong F(y \times_x \Pi_yz) \xrightarrow{F(\ev)} Fz$ gives
  \begin{equation*}
    \begin{tikzcd}
      &
      b \times_a (a \times_{Fx} F(\Pi_yz))
      \ar[r, equal]
      &
      b \times_{Fy} (Fy \times_{Fx} F(\Pi_yz))
      \ar[rrd, "(Fy \times_{Fx} F(\Pi_yh))^*\beta"]
      \\
      b \times_a d
      \ar[r, "{\beta \times_\alpha \delta}"', dashed]
      \ar[ur, dashed, "b \times_a u_1"]
      &
      Fy \times_{Fx} Fw
      \ar[r, "\cong"']
      &
      F(y \times_x w)
      \ar[r, "{F(y \times_x v)}"]
      \ar[rrdd, "F(v^\dagger)"', dashed]
      &
      F(y \times_x \Pi_yz)
      \ar[r, "\cong"{description}]
      \ar[rd, equal]
      &
      Fy \times_{Fx} F(\Pi_yz)
      \ar[d, "\cong"]
      \\
      &
      &
      &
      &
      F(y \times_x \Pi_yz)
      \ar[d, "F(\ev)"]
      \\
      &
      &
      &
      &
      Fz
    \end{tikzcd}
  \end{equation*}
  Commutativity of \Cref{eqn:pi-right} also means
  $\textcolor{cyan0}{\Pi_b(f,\gamma)} \cdot u_2 = \textcolor{yellow0}{(b
    \times_{Fy} \ev)^\ddagger} \cdot \textcolor{green0}{a \times_{Fx} F(\ev)^\ddagger}
  \cdot u_1$.
  So, transposing under $g^* \dashv \Pi_b$ and composing with
  $(Fh)^*\beta \colon b \times_{Fy} Fz \to Fz$ gives the same map.
  But performing this operation to
  $\textcolor{yellow0}{(b \times_{Fy} \ev)^\ddagger} \cdot \textcolor{green0}{a
    \times_{Fx} F(\ev)^\ddagger} \cdot u_1$ is exactly the boundary of the above
  diagram, while performing this operation to
  $\textcolor{cyan0}{\Pi_b(f,\gamma)} \cdot u_2$ gives
  \begin{tikzcd}
    b \times_a d \ar[r, "u_2^\ddagger", dashed] & c \ar[r, "\gamma"] & Fz
  \end{tikzcd}.
  Hence, we have proved that
  \begin{equation*}
    \begin{tikzcd}
      c \ar[d, "\gamma"'] & b \times_a d \ar[l, "u_2^\ddagger"', dashed] \ar[d, "{\beta \times_\alpha \delta}"]
      \\
      Fz & F(y \times_x w) \ar[l, "F(v^\dagger)", dashed]
    \end{tikzcd}
  \end{equation*}
  and so \Cref{eqn:flat-left} commutes.
\end{proof}

\begin{construction}\label{constr:sharp-components}
  Next, assume there is some $(i,j) \colon \delta \to \alpha$ and a map
  $(\widetilde{u},\widetilde{v}) \colon \beta \times_\alpha \delta \to \gamma$ over $\beta$ in $\Gl(F)$,
  as follows.
  Because $\widetilde{v} \colon y \times_x w \to z$ is over $y$, it transposes to
  $v \coloneqq \widetilde{v}^\dagger \colon w \to \Pi_yz$ over $x$.
  Hence, it induces a unique map $\widetilde{u}_1 \colon d \to a \times_{Fx} F(\Pi_yz)$ that
  factors $(i, F(\widetilde{v}^\dagger) \cdot \delta)$ via $F(\Pi_yh)^*\alpha$.
  Further, the map
  $u$ over $b$ transposes along $g^* \dashv \Pi_b$ to a map
  $\widetilde{u}_2 \coloneqq \widetilde{u}^\ddagger \colon d \to \Pi_bc$ over $a$.
  \begin{equation*}\label{eqn:sharp-left}\tag{\textsc{$\sharp$-left}}
    \begin{tikzcd}[column sep=small]
      &
      &
      &
      &
      &
      &
      \\
      &
      |[alias=c]| c
      &
      |[alias=b-d]| b \times_a d
      &
      |[alias=de]| d
      &
      &
      &
      \\
      |[alias=b-Fz]| b \times_{Fy} Fz
      &
      &
      &
      &
      &
      |[alias=Pb-F-Pi]| a \times_{Fx} F(\Pi_yz)
      &
      &
      \\
      &
      |[alias=Fz]| Fz
      &
      &
      |[alias=F-y-w]| F(y \times_x w)
      &
      |[alias=Fw]| Fw
      &
      &
      |[alias=F-Pi]| F(\Pi_yz)
      &
      \\
      &
      &
      |[alias=b]| b
      &
      &
      &
      |[alias=a]| a
      &
      &
      \\
      &
      &
      &
      &
      &
      &
      &
      \\
      &
      &
      &
      |[alias=Fy]| Fy
      &
      &
      &
      |[alias=Fx]| Fx
      \ar[from=c, to=b, "f" {description}, bend right=10]
      \ar[from=c, to=Fz, "\gamma"{description}]
      \ar[from=c, to=b-Fz, "{(f, \gamma)}" {description}, color=cyan0]
      %
      %
      \ar[from=b-Fz, to=b, bend right]
      \ar[from=b-Fz, to=Fz, "(Fh)^*\beta"{description}]
      \ar[from=b-Fz, to=Fy, phantom, "\lrcorner"{pos=0, rotate=-30}, out=-43]
      \ar[from=b, to=a, "g"{pos=0.5,description}]
      \ar[from=b, to=Fy, "\beta"' {pos=0.2, description}]
      \ar[from=a, to=Fx, "\alpha"{pos=0.2, description}]
      \ar[from=Fz, to=Fy, crossing over, "Fh"{pos=0.5, description}, bend right=20]
      \ar[from=Fy, to=Fx, "Fk"']
      \ar[from=F-Pi, to=Fx, "F(\Pi_yh)"{description,pos=0.3}]
      \ar[from=Pb-F-Pi, to=a]
      \ar[from=Pb-F-Pi, to=Fx, phantom, "\lrcorner"{pos=0, rotate=-30}]
      %
      %
      \ar[from=Pb-F-Pi, to=F-Pi, crossing over, "F(\Pi_yh)^*\alpha"]
      %
      %
      %
      %
      %
      %
      %
      %
      \ar[from=de, to=a, "i"{description}, bend left=20]
      \ar[from=de, to=Fw, "\delta"']
      \ar[from=de, to=Pb-F-Pi, "\widetilde{u}_1", dashed]
      %
      \ar[from=Fw, to=Fx, "Fj"{description}, bend left=30, crossing over]
      \ar[from=Fw, to=F-Pi, "F(\widetilde{v}^\dagger)" {pos=0.7, description}, dashed, crossing over]
      \ar[from=b-d, to=c, "\widetilde{u}"', dashed]
      \ar[from=b-d, to=de]
      \ar[from=b-d, to=b]
      \ar[from=b-d, to=a, phantom, "\lrcorner"{pos=0}]
      \ar[from=F-y-w, to=Fw]
      \ar[from=F-y-w, to=Fy, crossing over]
      \ar[from=F-y-w, to=Fx, phantom, "\lrcorner"{pos=0}]
      \ar[from=F-y-w, to=Fz, "F\widetilde{v}"{description, pos=0.25}, dashed, crossing over]
      \ar[from=b-d, to=F-y-w, dashed, "{\beta \times_\alpha \delta}"{description}]
    \end{tikzcd}
  \end{equation*}
\end{construction}

\begin{lemma}\label{lem:sharp-components}
  The maps $\widetilde{u}_1,\widetilde{u}_2$ constructed above is such that
  \begin{equation*}
    \begin{tikzcd}
      & \Pi_bc \ar[r, "{\Pi_b(f,\gamma)}", color=cyan0] & \Pi_b(b \times_{Fy} Fz) \\
      d \ar[ur, "\widetilde{u}_2", dashed]  \ar[rd, "\widetilde{u}_1"', dashed]
      & & a \times_{Fx} \Pi_{Fy}Fz \ar[u, "(b \times_{Fy} \ev)^\ddagger"', color=yellow0] \\
      & a \times_{Fx} F(\Pi_yz) \ar[ur, "a \times_{Fx} F(\ev)^\ddagger"', color=green0]
    \end{tikzcd}
  \end{equation*}
\end{lemma}
\begin{proof}
  %
  %
  %
  %
  As in the proof of \Cref{lem:flat}, pulling back the commutativity of
  $F(\Pi_yh)^*\alpha \cdot \widetilde{u}_1 = F(\widetilde{v}^\dagger) \cdot \delta$ along
  $Fk$ then composing with
  $Fy \times_{Fx} F(\Pi_yz) \cong F(y \times_x \Pi_yz) \xrightarrow{F(\ev)} Fz$
  and then using the commutativity of
  $\gamma \cdot \widetilde{u} = F\widetilde{v} \cdot (\beta \times_\alpha \delta)$ gives
  \begin{equation*}
    \begin{tikzcd}
      &
      b \times_a (a \times_{Fx} F(\Pi_yz))
      \ar[r, equal]
      &
      b \times_{Fy} (Fy \times_{Fx} F(\Pi_yz))
      \ar[rrd, "(Fy \times_{Fx} F(\Pi_yh))^*\beta"]
      \\
      b \times_a d
      \ar[r, "{\beta \times_\alpha \delta}"{description}, dashed]
      \ar[ur, "b \times_a \widetilde{u}_1"]
      \ar[rdd, "\widetilde{u}_2^\ddagger = \widetilde{u}"', dashed]
      &
      Fy \times_{Fx} Fw
      \ar[r, "\cong"{description}]
      &
      F(y \times_x w)
      \ar[r, "{F(y \times_x v^\dagger)}"]
      \ar[rrdd, "F(\widetilde{v}^{\dagger\dagger}) = F\widetilde{v}"{description}, dashed]
      &
      F(y \times_x \Pi_yz)
      \ar[r, "\cong"{description}]
      \ar[rd, equal]
      &
      Fy \times_{Fx} F(\Pi_yz)
      \ar[d, "\cong"]
      \\
      &
      &
      &
      &
      F(y \times_x \Pi_yz)
      \ar[d, "F(\ev)"]
      \\
      &
      c \ar[rrr, "\gamma"']
      &
      &
      &
      Fz
    \end{tikzcd}
  \end{equation*}

  Reusing the calculations from \Cref{lem:flat}, the transpose of
  $\textcolor{yellow0}{(b \times_{Fy} \ev^\ddagger)} \cdot \textcolor{green0}{a
    \times_{Fx} F(\ev)^\ddagger} \cdot \widetilde{u}_1$ and
  $\textcolor{cyan0}{\Pi_b(f,\gamma)} \cdot \widetilde{u}_2$ respectively composed with
  $b \times_{Fy} Fz \to Fz$ are exactly the top and bottom boundaries above respectively.
  To show that two maps $d \rightrightarrows \Pi_b(b \times_{Fy} Fz)$ over $a$
  are identical is to show that their transposes over $g^* \dashv \Pi_b$
  composed with the projection $(Fh)^*\beta \colon b \times_{Fy} Fz \to Fz$ are
  identical, the result follows.
\end{proof}

\begin{construction}\label{constr:sharp}
  By \Cref{lem:sharp-components}, the maps $\widetilde{u}_1,\widetilde{u}_2$ from
  \Cref{constr:sharp-components} gives rise to a unique map $u \colon d \to p$
  factoring $(\widetilde{u}_1,\widetilde{u}_2)$ through
  $(\textcolor{yellow0}{(b \times_{Fy} \ev)^\ddagger} \cdot \textcolor{green0}{(a
  \times_{Fx} F(\ev)^\ddagger)}, \textcolor{cyan0}{\Pi_b(f,\gamma)})$
  as in \Cref{eqn:pi-right}.
  In particular, this gives a map
  \begin{equation*}
    (\widetilde{u},\widetilde{v})^\sharp \coloneqq (u,v) \colon \delta \to
    \Pi_\beta\gamma
  \end{equation*}
  in $\Gl(F)$ over $\alpha$.
\end{construction}

\begin{lemma}\label{lem:sharp-flat-inverses}
  The maps
  $(-)^\flat \colon \sfrac{\Gl(F)}{\alpha}(\delta, \Pi_\beta\gamma)
  \leftrightarrows \sfrac{\Gl(F)}{\beta}(\beta \times_\alpha \delta, \gamma)
  \cocolon (-)^\sharp$ from \Cref{constr:sharp,constr:flat} are mutual inverses.
\end{lemma}
\begin{proof}
  Given a pair of maps $(u,v)$ as in \Cref{constr:flat}, we obtain maps
  $(u_1,u_2)$ as in \Cref{eqn:pi-right} and have
  $(u,v)^\flat = (u_2^\ddagger, v^\dagger)$.
  Then, by \Cref{constr:sharp-components}, with $\widetilde{u} = u_2^\ddagger$
  and $\widetilde{v} = v^\dagger$, the second component of
  $(u,v)^{\flat\sharp} = (\widetilde{u}, \widetilde{v})^\sharp$ is
  $\widetilde{v}^\dagger = v^{\dagger\dagger} = v$.
  Its first component is constructed by factoring
  $(\widetilde{u}_1,\widetilde{u}_2)$ through the pullback in
  \Cref{eqn:Pi-pullback}, where
  $\widetilde{u}_2 = \widetilde{u}^\ddagger = u_2^{\ddagger\ddagger} = u_2$ and
  $\widetilde{u}_1$ is the unique map factoring
  $(\alpha \cdot i, F(\widetilde{v}^\dagger) \cdot \delta) =
  (\alpha \cdot i, Fv \cdot \delta)$ through
  $F(\Pi_yh)^*\alpha$ as in \Cref{eqn:sharp-left}.
  By \Cref{eqn:pi-right}, it follows that $\widetilde{u}_1 = u_1$.
  Hence, the first component of $(\widetilde{u}, \widetilde{v})^\sharp$ is $u$,
  thus showing $(u,v)^{\flat\sharp} = (u,v)$.

  Conversely, assume there is a pair of maps $(\widetilde{u},\widetilde{v})$ as
  in \Cref{constr:sharp-components}.
  Then, as before, the second component of
  $(\widetilde{u},\widetilde{v})^{\sharp\flat}$ is easily seen to be
  $\widetilde{v}$.
  Set $\widetilde{u}_1$ as in \Cref{eqn:sharp-left} and
  $u_2 \coloneqq \widetilde{u}^\ddagger$ like in \Cref{constr:sharp-components}
  so that by \Cref{constr:sharp}, there is a unique map
  $u \coloneqq (\widetilde{u}_1, \widetilde{u}_2) \colon d \to p$ arising from
  the pullback \Cref{eqn:Pi-pullback}.
  In particular,
  $d \xrightarrow{u} p \to \Pi_bc = \widetilde{u}_2 = \widetilde{u}^\ddagger$ so
  following \Cref{constr:flat}, it follows that
  $(\widetilde{u},\widetilde{v})^{\sharp\flat} =
  (u,\widetilde{v}^\dagger)^\sharp = (\widetilde{u}_2^\ddagger,\widetilde{v}^{\dagger\dagger})
  = (\widetilde{u}^{\ddagger\ddagger},\widetilde{v}^{\dagger\dagger})
  = (\widetilde{u},\widetilde{v})$.
\end{proof}

Therefore, to summarise, we have proved:

\begin{theorem}\label{thm:glue-Pi}
  If $F \colon \McC \to \McE$ preserves pullbacks and
  $(g \colon b \to a, k \colon y \to x) \colon \beta \to \alpha$ is a map in
  $\Gl(F)$ between $\beta \colon b \to Fy$ and $\alpha \colon a \to Fx$ such
  that $g,k$ and $Fk$ is powerful then so is $(g,k)$, with the dependent product
  along $(g,k)$ constructed in \Cref{constr:glue-Pi}.
  Further, the transposes are constructed in \Cref{constr:flat,constr:sharp} and
  the projection $\Gl(F) \to \McE$ preserves the counit.
\end{theorem}


%% file: colim-cats.tex
\section{Logical Structure in Limits of Categories}\label{sec:limit-logic}
We now shift our attention to the problem of constructing the logical structure in diagram categories from the corresponding logical structure in subdiagram categories.
That is, suppose an indexing category $\McI$ is built up as $\McI = \colim_n \McI_n$.
Then, for any category $\McE$, one has $\McE^\McI \simeq \colim_n\McE^{\McI_n}$,
and one would like to assemble the logical structure in each $\McE^{\McI_n}$ to
form corresponding logical structures in $\McE^\McI$.
Abstracting this goal, we investigate in this section, for a diagram of
categories $\bD \colon \McJ \to \Cat$, how compatible logical structures
(specifically the subobject classifier and dependent products) in each $\bD_j$
assemble to form logical structure in the limit of categories
$\overline\bD = \lim_{j \in \McJ}\bD_j$.

\subsection{Subobject Classifier}
In this part, we show in \Cref{lem:limit-Omega} that if each $\bD_j$ is equipped
with subobject classifiers $\Omega_j$ along with truth maps
$\true_j \colon 1 \to \Omega_j$ and the functorial action of $\bD$ preserves
these subobject classifiers and truth maps then these subobject classifiers and
truth maps assemble to form subobject classifiers and truth maps in
$\overline\bD$.

\begin{lemma}\label{lem:limit-jointly-reflects}
  The limiting legs
  $((-)|_j \colon \overline\bD \to \bD_j)_{j \in \McJ}$ of the limit
  $\overline{\bD} \coloneqq \lim_{j \in \McJ}\bD_j$ of a diagram
  $\bD \colon \McJ \to \Cat$ jointly reflects limits.
  This means that if there is a diagram $F \colon \bC \to \bD$ and a cone
  $\lambda = (d \xrightarrow{\lambda_c} Fc \in \overline\bD ~|~ c \in \bC)$ such
  that each $\lambda|_j = (d|_j \xrightarrow{\lambda_c|_j} (Fc)|_j \in \bD_j)_c$
  is a limiting cone for
  $\bC \xrightarrow{F} \overline\bD \xrightarrow{(-)|_j} \bD_j$ then
  $\lambda$ is a limiting cone for $F$.
\end{lemma}
\begin{proof}
  Suppose there is another cone
  $\lambda' = (\lambda'_c \colon d' \to Fc \in \overline\bD)_c$.
  Then, for each $j$, one has a cone
  $\lambda'|_j = (\lambda'_c|_j \colon d'|_j \to (Fc)|_j \in \bD_j)$ for
  $\bC \xrightarrow{F} \overline\bD \xrightarrow{(-)|_j} \bD_j$.
  Because $\lambda|_j$ is limiting for
  $\bC \xrightarrow{F} \overline\bD \xrightarrow{(-)|_j} \bD_j$, there is a
  unique map $f_j \colon d'|_j \to d|_j$ that is a map of cones from
  $\lambda'|_j$ to $\lambda|_j$.
  Clearly the family of maps $(f_j \colon d'|_j \to d|_j \in \bD_j)_j$ is
  compatible so they induce a unique map $f \colon d' \to d \in \overline\bD$
  such that $f|_j = f_j$.
  So, $f$ is a map of cones $\lambda' \to \lambda$.
  It is clearly the unique map of cones because if there is
  $f' \colon \lambda' \to \lambda$ then by the fact that each $\lambda|_j$ is
  limiting, each $f'|_j = f_j = f|_j$.
  And this means $f' = f$.
\end{proof}

\begin{lemma}\label{lem:limit-Omega}
  Fix a diagram $\bD \colon \McJ \to \Cat$ and put
  $\overline\bD \coloneqq \lim_{j \in \McJ}\bD_j$ with limiting legs
  $((-)|_j \colon \overline\bD \to \bD_j)_{j \in \McJ}$.
  Suppose:
  \begin{itemize}
    \item Each limiting leg $(-)|_j \colon \overline\bD \to \bD_j$ preserves
    pullbacks.
    \item Each $\bD_j$ is equipped with a subobject classifier $\Omega_j$ and
    terminal object $1_j$ and truth map $\true_j \colon 1_j \to \Omega_j$.
    \item For each $\alpha \colon j \to j' \in \McJ$, the functorial action
    $\alpha^* \coloneqq D_\alpha \colon \bD_j \to \bD_{j'}$ preserves pullbacks and moreover
    $\alpha^*\true_j = \true_{j'}$ (in particular $\alpha^*$ preserves the terminal
    object and subobject classifier).
  \end{itemize}
  Then, $\overline\bD$ admits:
  \begin{itemize}
    \item A terminal object $1$ such that $1|_j = 1_j$.
    \item An unique object $\Omega \in \overline\bD$ such that each
    $\Omega|_j = \Omega_j$.
    \item An unique map $\true \colon 1 \to \Omega$ such that $\true|_j = \true_j$ which
    serves as the truth map making $\Omega$ into a subobject classifier.
  \end{itemize}
\end{lemma}
\begin{proof}
  It is easy to note the existence of a terminal object $1 \in \overline\bD$
  such that each $1|_j = 1_j$, a unique object $\Omega \in \overline\bD$ such
  that each $\Omega|_j = \Omega_j$ and a unique $\true \colon 1 \to \Omega$ such
  that each $\true|_j = \true_j$.
  It remains to show that $\true \colon 1 \to \Omega$ is indeed the truth map.

  Fix a mono $i \colon c \hookrightarrow d \in \overline\bD$.
  Because each limiting leg $(-)|_j \colon \overline\bD \to \bD_j$ preserves
  pullbacks, $i|_j \colon c|_j \hookrightarrow d|_j$ is a mono.
  So, there is a unique map $\chi_j \colon d|_j \to \Omega_j \in \bD_j$ such
  that
  \begin{equation}\label{eqn:chi-Dj}\tag{\textsc{$\chi$-$\bD_j$}}
    \begin{tikzcd}
      c|_j \ar[r, "!"] \ar[d, "i|_j"', hook]
      \ar[rd, "\lrcorner"{pos=0}, phantom]
      &
      1_j \ar[d, "\true_j"]
      \\
      d|_j \ar[r, "\chi_j"']
      &
      \Omega_j
    \end{tikzcd} \in \bD_j
  \end{equation}
  is a pullback.
  Moreover, for each $\alpha \colon j \to j' \in \bD$, one has
  $\alpha^*\chi_j = \chi_{j'}$.
  This is because $\alpha^* \colon \bD_j \to \bD_{j'}$ preserves pullbacks and
  truth values, so
  \begin{equation*}
    \begin{tikzcd}
      c|_{j'} = \alpha^*c|_j
      \ar[d, hook, "{i|_{j'} = \alpha^*i|_j}"']
      \ar[r, "!"]
      \ar[rd, "\lrcorner"{pos=0},, phantom]
      &
      \alpha^*1_j = 1_{j'}
      \ar[d, "\alpha^*\true_j = \true_{j'}"]
      \\
      d|_{j'} = \alpha^*d|_j
      \ar[r, "\alpha^*\chi_j"']
      &
      \alpha^*\Omega_j = \Omega_{j'}
    \end{tikzcd} \in \bD_{j'}
  \end{equation*}
  Because $\true_{j'} \colon 1_{j'} \to \Omega_{j'}$ is the truth map, which
  exists uniquely, it follows that $\chi_{j'} = (\bD f)\chi_j$.

  Hence, the family of maps $(\chi_j \colon d|_j \to \Omega_j)_j$ assemble to
  form a unique map $\chi \colon d \to \Omega$ such that each
  $\chi|_j = \chi_j$.
  This means that
  \begin{equation}\label{eqn:chi-D}\tag{\textsc{$\chi$-$\bD$}}
    \begin{tikzcd}
      c \ar[r, "!"] \ar[d, "i"', hook]
      &
      1 \ar[d, "\true"]
      \\
      d \ar[r, "\chi"']
      &
      \Omega
    \end{tikzcd} \in \overline\bD
  \end{equation}
  And because the image of \Cref{eqn:chi-D} under each
  $(-)|_j \colon \overline\bD \to \bD_j$ is the pullback \Cref{eqn:chi-Dj},
  \Cref{lem:limit-jointly-reflects} shows that \Cref{eqn:chi-D} is also a
  pullback.

  Moreover, $\chi \colon d \to \Omega$ is the unique map making \Cref{eqn:chi-D}
  into a pullback, because if there is another such map
  $\chi' \colon d \to \Omega$ then $\chi'|_j \colon d|_j \to \Omega$ makes
  \Cref{eqn:chi-Dj} into a pullback, so $\chi'|_j = \chi_j = \chi|_j$ for each
  $j$.
\end{proof}

\subsection{Dependent Products}
In this part, we show in \Cref{cor:limit-Pi} that a map
$g \colon b \to a \in \overline\bD$ is powerful when each
$g|_j \colon b|_j \to a|_j \in \bD_J$ is powerful and the functorial action of
$\bD$ preserves these dependent products.
In this case, the dependent product along each $g|_j \colon b|_j \to a|_j$ in
$\bD_j$ assemble to form a dependent product along $g \colon b \to a$ in $\bD$.
The argument proceeds by showing in \Cref{lem:limit-adjunction} that limits of
adjunctions assemble to form adjunction between limits of categories and in
\Cref{lem:limit-slices} that limit of slices are slices in limit of categories.

\begin{lemma}\label{lem:limit-adjunction}
  Fix diagrams $\bD,\bE \colon \McJ \rightrightarrows \Cat$ and put
  \begin{align*}
    \overline\bD \coloneqq \lim_{j \in \McJ} \bD_j &&
    \overline\bE \coloneqq \lim_{j \in \McJ} \bE_j
  \end{align*}
  Suppose there are functors
  $F \colon \overline\bD \leftrightarrows \overline\bE \cocolon U$ and
  $(F_j \colon \bD_j \leftrightarrows \bE_j \cocolon U_j ~|~ j \in \McJ)$ such that
  $F_j \dashv U_j$ and for all $\alpha \colon j \to j' \in \McJ$,
  \begin{equation*}
    \begin{tikzcd}[row sep=large, column sep=large]
      & |[alias=limD]| \overline\bD & & |[alias=limE]| \overline\bE \\
      & & |[alias=Djp]| \bD_{j'} & & |[alias=Ejp]| \bE_{j'} \\
      |[alias=Dj]| \bD_j & & |[alias=Ej]| \bE_j
      \ar[from=limD, to=limE, yshift=3, "F"]
      \ar[from=limE, to=limD, yshift=-3, "U"]
      \ar[from=Djp, to=Ejp, yshift=3, "F_{j'}"]
      \ar[from=Ejp, to=Djp, yshift=-3, "U_{j'}"]
      \ar[from=Djp, to=Ejp, phantom, "{\vdash}"{decoration,rotate=90}]
      \ar[from=Dj, to=Ej, yshift=3, "F_{j}"]
      \ar[from=Ej, to=Dj, yshift=-3, "U_{j}"]
      \ar[from=Dj, to=Ej, phantom, "{\vdash}"{decoration,rotate=90}]
      \ar[from=Dj, to=Djp, "\alpha^*"]
      \ar[from=Ej, to=Ejp, "\alpha^*"']
      \ar[from=limD, to=Dj, "(-)|_j"']
      \ar[from=limE, to=Ej, crossing over, "(-)|_j"{description,pos=0.25}]
      \ar[from=limD, to=Djp, "(-)|_{j'}"{description}]
      \ar[from=limE, to=Ejp, crossing over, "(-)|_{j'}"]
    \end{tikzcd}
  \end{equation*}
  where the maps $(-)|_j$ are the legs of the respective limiting cones and
  $\alpha^*$ are the respective actions on maps that give rise to a map of
  adjunctions from $F_j \dashv U_j$ to $F_{j'} \dashv U_{j'}$.
  Then, $F \dashv U$ and the slanted faces above are maps of adjunctions.
\end{lemma}
\begin{proof}
  By the universal property of the limit,
  $UF \colon \overline\bD \to \overline\bD$ is the unique functor such that
  $(UFd)|_j = U_jF_jd|_j$.
  Because each
  $(\alpha^* \colon \bD_j \to \bD_{j'}, \alpha^* \colon \bE_j \to \bE_{j'})$ is
  a map of adjunctions from $F_j \dashv U_j$ to $F_{j'} \dashv U_{j'}$, for each
  $d \in \overline\bD$, let $\eta_d \colon \mbbo{2} \to \overline\bD$ be the
  unique functor from the walking arrow category $\mbbo{2} = {\to}$ to
  $\overline\bD$ such that $(\eta_d)|_j$ picks out the unit
  $(\eta_j)_{d|_j} \colon d|_j \to U_jF_jd|_j \in \bD_j$ of $F_j \dashv U_j$ so
  that $\eta_d \colon d \to UFd$.
  Fixing $f \colon d \to d' \in \overline\bD$, define
  $\phi \colon \mbbo{3} = \set{\to\to} \to \overline\bD$ to be the unique map
  such that $\phi|_j$ selects
  $d|_j \xrightarrow{(\eta_j)_{d|_j}} U_jF_jd|_j \xrightarrow{U_jF_jf|_j}
  U_jF_jd'|_j$ and define $\psi \colon \mbbo{3} = \set{\to\to} \to \overline\bD$
  to be the unique map such that
  $\psi|_j = \set{d|_j \xrightarrow{f|_j} d'|_j \xrightarrow{(\eta_j)_{d'|_j}}
    U_jF_jd'|_j}$.
  Naturality of each $\eta_j$ means that $\phi|_j = \psi|_j$, so $\phi = \psi$,
  verifying naturality.
  This constructs the unit $\eta \colon \id_{\overline\bD} \to UF$ as the unique
  natural transformation with components $\eta_d$ such that
  $(\eta_d)|_j = (\eta_j)_{d|_j} \colon d|_j \to U_jF_jd|_j$.
  A similar construction gives the counit
  $\epsilon \colon FU \to \id_{\overline\bE}$ as the unique natural
  transformation with components $\epsilon_e$ such that
  $(\epsilon_e)|_j = (\epsilon_j)_{e|_j} \colon F_jU_je|_j \to e|_j$.

  The triangle identity says for each $F_j \dashv U_j$,
  \begin{align*}
    (\epsilon_{Fd} \cdot F\eta_d)|_j = F_jd|_j \xrightarrow{(F\eta_d)|_j g= F|_j (\eta_d)|_j = F|_j(\eta_j)|_{d_j}}
    F_jU_jF_jd|_j \xrightarrow{(\epsilon_{Fd})|_j = (\epsilon_j)_{F_jd|_j}} F_jd|_j = \id_{(Fd)|_j}
  \end{align*}
  Hence, by the universal property of the limit, $\epsilon_{Fd} \cdot F\eta_d = \id_{Fd}$.
  A similar verification shows $U\epsilon_e \cdot \eta_{Ue} = \id_{Ue}$.
  Hence, $F \dashv U$.
\end{proof}

\begin{lemma}\label{lem:limit-slices}
  Fix a diagram $\bD \colon \McJ \to \Cat$ and put
  $\overline\bD \coloneqq \lim_{j \in \McJ}\bD_j$ with limiting cone
  $((-)|_j \colon \overline\bD \to \bD_j)_{j \in \McJ}$.
  Let $d \in \overline\bD$ be an object of $\overline\bD$.
  Define $\bD_{(d)} \colon \McJ \to \Cat$ whose action on objects is
  $\bD_{(d)}(j) \coloneqq \sfrac{\bD_j}{d|_j}$ and whose action on a map
  $\alpha \colon j \to j'$ takes $d' \to d|_j \in \sfrac{\bD_j}{d|_j}$ to
  $\alpha^*d' \to \alpha^*d|_j = d|_{j'}$.
  Then,
  \begin{align*}
    \lim_{j \in \McJ} \bD_{(d)}(j) = \lim_{j \in \McJ} \sfrac{\bD_j}{d|_j} =
    \bD/d
  \end{align*}
\end{lemma}
\begin{proof}
  The fact that $((-)|_j \colon \bD \to \bD_j)_{j \in \McJ}$ is a cone ensures
  the above definition of $\bD_{(d)}(\alpha)$ does send a map over $d|_j$ to a
  map over $d|_{j'}$.
  It is also clear from functoriality of $\bD \colon \McJ \to \Cat$ that
  $\bD_{(d)}$ does define a functor $\McJ \to \Cat$.

  Note that we have a cone $(\pi_j \colon \bD/d \to \bD_{(d)}(j))_{j\in\McJ}$
  whose each leg at $j \in \McJ$ takes $d' \to d \in \bD$ to
  $d'|_j \to d|_j \in \bD_{(d)}(j)$ by functorial action of the limiting legs of
  $((-)|_j \colon \overline\bD \to \bD_j)_{j \in\McJ}$.
  Assume now that there is another cone
  $(F_j \colon \bE \to \bD_{(d)}(j))_{j \in \McJ}$.
  Fix an object $e \in \bE$.
  The cone condition of $(F_j)_j$ says that one has a cone
  $(\mbbo{2} \to \bD_j)_j$ with legs at $j \in \McJ$ picking out the arrow
  $F_je \colon \dom{F_je} \to d|_j \in \bD_j$.
  This cone lifts uniquely to a map in $\bD$ with codomain $d$ and domain whose
  restrictions at $j$ are $\dom{F_je}$.
  Let this unique lift be $Fe$ so that $(Fe)|_j = F_je$.
  It is clear that this defines a map $F \colon \bE \to \bD/d$ such that
  $(F_j)_j = (\pi_j)_j \cdot F$.
  Uniqueness follows from the uniqueness of the lifts arising from the universal
  property of $\overline\bD = \lim_j\bD_j$ used to define $F$.
\end{proof}

\begin{corollary}\label{cor:limit-Pi}
  Fix a diagram $\bD \colon \McJ \to \Cat$ and put
  $\overline\bD \coloneqq \lim_{j \in \McJ}\bD_j$ with limiting legs
  $((-)|_j \colon \overline\bD \to \bD_j)_{j \in \McJ}$.
  If $g \colon b \to a \in \overline\bD$ is such that for any $j \in \McJ$, the
  restriction $g|_j \colon b|_j \to a|_j \in \bD_j$ is powerful and such that
  for any $\alpha \colon j \to j' \in \McJ$, the solid maps commute with the
  bottom face being a map of adjunctions:
  \begin{equation*}
    \begin{tikzcd}[row sep=large, column sep=large]
      & |[alias=D-a]| \overline\bD/a & & |[alias=D-b]| \overline\bD/b \\
      & & |[alias=Djp-a]| \sfrac{\bD_{j'}}{a|_{j'}} & & |[alias=Djp-b]| \sfrac{\bD_{j'}}{b|_{j'}} \\
      |[alias=Dj-a]| \sfrac{\bD_j}{a|_j} & & |[alias=Dj-b]| \sfrac{\bD_j}{b|_j}
      \ar[from=D-a, to=D-b, yshift=3, "g^*"]
      \ar[from=D-b, to=D-a, yshift=-3, "\Pi_g", dashed]
      \ar[from=Djp-a, to=Djp-b, yshift=3, "(g|_{j'})^*"]
      \ar[from=Djp-b, to=Djp-a, yshift=-3, "\Pi_{g|_{j'}}"]
      \ar[from=Djp-a, to=Djp-b, phantom, "{\vdash}"{decoration,rotate=90}]
      \ar[from=Dj-a, to=Dj-b, yshift=3, "(g|_{j})^*"]
      \ar[from=Dj-b, to=Dj-a, yshift=-3, "\Pi_{g|_{j'}}"]
      \ar[from=Dj-a, to=Dj-b, phantom, "{\vdash}"{decoration,rotate=90}]
      \ar[from=Dj-a, to=Djp-a, "\alpha^*"]
      \ar[from=Dj-b, to=Djp-b, "\alpha^*"']
      \ar[from=D-a, to=Dj-a]
      \ar[from=D-b, to=Dj-b, crossing over]
      \ar[from=D-a, to=Djp-a]
      \ar[from=D-b, to=Djp-b, crossing over]
    \end{tikzcd}
  \end{equation*}
  where $\alpha^*$ arise from the functorial action of the diagram $\bD$.
  Then, $g$ is powerful with the right adjoint of $g^*$ given by the unique map
  such that
  $(\Pi_g(c \xrightarrow{f} b))|_j = \Pi_{g|_j}(c|_j \xrightarrow{f|_j} b|_j)$.
  Furthermore, each of the slanted faces above is a map of adjunctions.
\end{corollary}
\begin{proof}
  By \Cref{lem:limit-slices}, we have
  $\overline\bD/a = \lim_j\sfrac{\bD_j}{a|_j}$ and
  $\overline\bD/b = \lim_j\sfrac{\bD_j}{b|_j}$, so by
  \Cref{lem:limit-adjunction}, it follows that $g^* \dashv \Pi_g$.
\end{proof}


%% file: glue-diagrams.tex
\section{Iterated Gluing Diagrams}\label{sec:iter-glue}
In this section, we aim to develop a framework in which one can combine the results from the previous \Cref{sec:gpd-logic,sec:gluing-logic,sec:limit-logic}.
%
%
Namely, our goal is to define categories $\McI$ that may be constructed as a
colimit $\McI = \colim_n \McI_{\leq n}$ where each ``$n$-th stage''
$\McI_{\leq n}$ is obtained from some notion of ``previous stages'' $\McI_{<n}$
by attaching an ``$n$-th boundary'' $\partial\McI_n$, so that each
$\McI_{\leq n}$ arises as an Artin gluing category along some functor from
$\McI_{<n}$ into $\partial\McI_n$.
An already well-established class of categories with this property are the
\emph{inverse categories}, which are special cases of Reedy categories where all
maps lower degrees.
We will therefore take inverse categories as inspiration for formulating our
framework of iterated Artin gluing.

Much like how the simplex category $\bDelta$ is a canonical example of a Reedy
category, a canonical example of an inverse category is the semi-simplex
category $\bDelta_{\mathsf{inj}}$ of finite linear orders and order-preserving
surjections.
Semi-simplicial sets can be constructed dimension-wise in an iterative manner by
specifying its $n$-simplicies and the faces of each of its $n$-simplicies in
suitably compatible manners.
With inverse categories and semi-simplicial sets as concrete examples to guide
our intuition in mind, we now proceed to motivate and formulate our framework of
iterated gluing categories.

\begin{definition}\label{def:gen-inv}
  A \emph{generalised inverse structure} on a category $\McI$ is a function
  $\deg\relax \colon \ob\McI \to \bN$ such that if $f \colon i \to j \in \McI$
  is not an isomorphism, $\deg i > \deg j$; and if $f \colon i \cong j \in \McI$
  is an isomorphism then $\deg{i} = \deg{j}$.

  When equipped with such a structure, for each $i \in \McI$, put $\McI^-(i)$ as
  the full subcategory of $i \downarrow \McI$ spanned by the strictly
  degree-decreasing maps.
  And for each $n \in \bN$, put:
  \begin{itemize}
    \item $\bG_n(\McI)$ to be the full subgroupoid of $\McI$ spanned by the
    isomorphisms whose source and target are all degree $n$.
    \item $\McI_{\leq n}$ to be the full subcategory of $\McI$ spanned by
    objects not exceeding degree $n$.
    Often, we also write $\McI_{<n}$ for $\McI_{\leq n}$.
  \end{itemize}

  A \emph{strict inverse structure} on $\McI$ is a generalised inverse structure
  on $\McI$ where each $\bG_n(\McI)$ consists only of identity maps.
\end{definition}
Indeed, the opposite category of the semi-simplex category
$\bDelta_{\mathsf{inj}}$ of the simplex category $\bDelta$ spanned by the face
maps is an example of a strict inverse category.

If $X$ is a semi-simplicial set then each of its $n$-simplex $x_n$ is uniquely
determined by its $n$ (compatible) faces $(\delta_i^nx_n)_{i=0,...,n-1}$, where
each $\delta_i^nx_n$ is a simplex of dimension strictly less than $n$.
Inverse categories generalise this in that diagrams indexed by a strict inverse
category $\McI$ valued in a category $\McE$ can be constructed by induction on
the degree (i.e. dimension) provided that $\McE$ has enough limits \cite[Lemma
3.10]{rv14}.
To see this, first note that to construct a diagram
$X_{\leq n} \in \McE^{\McI_{\leq n}}$ is precisely to specify:
\begin{itemize}
  \item A diagram $X_{<n} \in \McE^{\McI_{<n}}$; and
  \item For each each $i \in \bG_n(\McI)$ of exactly degree $n$, an object
  $X_{=n}i \in \McE$; and
  \item A compatible family of maps
  $(X_{\leq n}f^- \colon X_{=n}i \to X_{<n}\underline{i})_{f^- \colon i
    \to \underline{i} \in \McI^-(i)}$, which suffices because all non-identity
  maps strict lower degrees
\end{itemize}
The compatible family of maps
$(X_{\leq n}f^- \colon X_{=n}i \to X_{\leq n-1}\underline{i})_{f^- \colon i \to
  \underline{i} \in \McI^-(i)}$ is, by the universal property of the weighted
limit, exactly a map $X_{=n}i \to \set{\McI^-(i), X_{\leq n-1}}$, where
$\McI^-(i)$ is the slice under $i$ spanned by those maps excluding the identity.
In the case of semi-simplicial sets, the weighted limit
$\set{\bDelta_{\mathsf{inj}}^-([n]), X}$ is precisely the usual $n$-th
coskeleton of $X$.

Assembling each these maps
$(X_{=n}i \to \set{\McI^-(i), X_{\leq n-1}})_{i \in \bG_n(\McI)}$ into weighted
limits functorially, we arrive at the definition of absolute matching objects
and matching maps.

\begin{definition}\label{def:inv-mat}
  Let $\McI$ be a category equipped with an inverse structure
  $\deg\relax \colon \ob\McI \to \bN$.
  Fix a category $\McE$ admitting limits indexed by each $\McI^-(i)$, the
  underslice of $i$ spanned by those maps other than the initial object.

  Denote by $\res_{< n} \colon \McE^{\McI_{\leq n}} \to \McE^{\McI_{<n}}$
  precomposition with $\McI_{<n} \hookrightarrow \McI_{\leq n}$ and
  $\cosk_n \colon \McE^{\McI_{<n}} \to \McE^{\McI_{\leq n}}$ its right Kan
  extension (which exists as $\McE$ is sufficiently complete as described above).
  Further put $t_n \colon \bG_n(\McI) \hookrightarrow \McI_{\leq n}$ as the
  inclusion.
  \begin{equation*}
    \begin{tikzcd}
      \McE^{\bG_n(\McI)}
      &
      \McE^{\McI_{\leq n}}
      \ar[r, "\res_{<n}", yshift=3]
      \ar[r, "\vdash"{rotate=90}, phantom]
      \ar[l, "(t_n)^*"']
      &
      \McE^{\McI_{<n}}
      \ar[l, "\cosk_n", yshift=-3]
    \end{tikzcd}
  \end{equation*}
  Denote by
  $\M_n \coloneqq (t_n)^* \cdot \cosk_n \colon \McE^{\McI_{<n}} \to
  \McE^{\McI_{\leq n}} \to \McE^{\bG_n(\McI)}$ the \emph{$n$-th matching object
    functor} and
  $\m_n \colon (t_n)^* \to \M_n \cdot \res_{<n} = (t_n)^* \cdot \cosk_n \cdot
  \res_{<n}$ the unit of the adjunction $\res_{<n} \dashv \cosk_n$ composed with
  $(t_n)^*$ the \emph{$n$-th matching map}.
\end{definition}

To understand the behaviour of the matching object and matching map as well as
how they relate to construction of inverse diagrams, consider the 3-horn
$\Lambda^3_2 \in \Set^{\bDelta^\op}$ given by subobject of the standard
3-simplex $\Delta^3$ spanned by all of the faces containing the vertex 2, or
equivalently the standard 3-simplex with the face opposite to the vertex 2
removed, restricted to $\Set^{\bDelta_{\mathsf{inj}}^\op}$.
Because $\bDelta_{\mathsf{inj}}^\op$ is spanned by the face maps, the
restriction $\Set^{\bDelta^\op} \to \Set^{\bDelta_{\mathsf{inj}}^\op}$ forgets
the degeneracies of a simplicial set.
Hence, one may picture $\Lambda^3_2$ as a semi-simplicial set in the picture on
the left as below, with the purple faces representing the faces present, the
blue face representing the absent face and the black lines representing the
edges.
The restriction of $\Lambda^3_2$ to $(\bDelta_{\mathsf{inj}}^\op)^-([2])$ is
then $\res_{<2}\Lambda^3_2$, which is the semi-simplicial set with the same
1-simplices as $\Lambda^3_2$ but no 2- or 3-simplices.
In other words, it is just all the edges of $\Lambda^3_2$, as given by the
right diagram below, where the blue faces represent empty faces.
\begin{center}
  \begin{tikzpicture}[decoration={markings,mark=at position 0.5 with {\arrow{To}}}]
    \coordinate [label=above:3] (3) at (0,{sqrt(2)},0);
    \coordinate [label=left:0] (0) at ({-.5*sqrt(3)},0,-.5);
    \coordinate [label=below:1] (1) at (0,0,1);
    \coordinate [label=right:2] (2) at ({.5*sqrt(3)},0,-.5);
    \draw[fill=magenta0,fill opacity=.5] (1)--(2)--(3)--cycle;
    \draw[fill=magenta0,fill opacity=.5] (0)--(1)--(2)--cycle;
    \draw[fill=cyan0,fill opacity=.25] (0)--(1)--(3)--cycle;
    \draw[fill=magenta0,fill opacity=.5] (0)--(2)--(3)--cycle;
    \draw[postaction={decorate}] (0)--(1);
    \draw[postaction={decorate}] (1)--(2);
    \draw[postaction={decorate}] (2)--(3);
    \draw[postaction={decorate}] (0)--(2);
    \draw[postaction={decorate}] (0)--(3);
    \draw[postaction={decorate}] (1)--(3);
  \end{tikzpicture}
  \hfil
  \begin{tikzpicture}[decoration={markings,mark=at position 0.5 with {\arrow{To}}}]
    \coordinate [label=above:3] (3) at (0,{sqrt(2)},0);
    \coordinate [label=left:0] (0) at ({-.5*sqrt(3)},0,-.5);
    \coordinate [label=below:1] (1) at (0,0,1);
    \coordinate [label=right:2] (2) at ({.5*sqrt(3)},0,-.5);
    \draw[fill=cyan0,fill opacity=.25] (1)--(2)--(3)--cycle;
    \draw[fill=cyan0,fill opacity=.25] (0)--(2)--(3)--cycle;
    \draw[fill=cyan0,fill opacity=.25] (0)--(1)--(2)--cycle;
    \draw[fill=cyan0,fill opacity=.25] (0)--(1)--(3)--cycle;
    \draw[postaction={decorate}] (0)--(1);
    \draw[postaction={decorate}] (1)--(2);
    \draw[postaction={decorate}] (2)--(3);
    \draw[postaction={decorate}] (0)--(2);
    \draw[postaction={decorate}] (0)--(3);
    \draw[postaction={decorate}] (1)--(3);
  \end{tikzpicture}
\end{center}

Next, consider the coskeleton $\cosk_{2}\res_{<2}\Lambda^3_2$.
Each 2-simplex of the coskeleton
$\cosk_{2}\res_{<2}\Lambda^3_2$ is a choice of an element in $(\cosk_{2}\res_{<2}\Lambda^3_2)_2$.
By the universal property of the weighted limit, this is the same as choosing a
compatible tuple
$(x_d \in (\Lambda^3_2)_1)_{d \colon [1] \to [2] \text{ face map}}$.
Examples of such tuples are $(1 \to 2, 2 \to 3, 1 \to 3)$ and
$(0 \to 1, 1 \to 2, 0 \to 2)$.
Observing that the tuple $(1 \to 2, 2 \to 3, 1 \to 3)$ may be encoded as a
formal filling of the area in the diagram on the right above spanned by the maps
$1 \to 2$ and $2 \to 3$ and $2 \to 3$, one concludes that the 2-simplices of the
coskeleton $\cosk_{2}\res_{<2}\Lambda^3_2$ is obtained by formally filling in
the 2-dimensional gaps along the 1-dimensional edges.
Likewise, the 3-simplices of $\cosk_{2}\res_{<2}\Lambda^3_2$ is obtained by
formally filling in the 3-dimensional gaps along the 1-dimensional edges.
Therefore, one may picture the coskeleton $\cosk_{2}\res_{<2}\Lambda^3_2$ below
as the pink simplex, where the pink faces represent the formal fillers.
In particular, the matching object at $[2]$ of $\Lambda^3_2$ is exactly
$\M_2\res_{<2}\Lambda^3_2 = (\cosk_{2}\res_{<2}\Lambda^3_2)_2$, the set of
2-simplices of the coskeleton, which we have concluded to be all of the formal
face fillers along the edges of the skeleton.
We also note that in terms of physical intuition, the coskeleton
$\cosk_{2}\res_{<2}\Lambda^3_2$ also admits an pink 3$d$-filler representing
the only 3-simplex it has which corresponds to the identity map on $[3]$.
In other words, the skeleton $\cosk_2\res_{<2}\Lambda_2^3 = \Delta^3$ is the
standard 3-simplex and the matching object $\M_2\Lambda^3_2$ is the set of its
2-simplices.

\begin{center}
  \begin{tikzpicture}[decoration={markings,mark=at position 0.5 with {\arrow{To}}}]
    \coordinate [label=left:0] (0) at ({-.5*sqrt(3)},0,-.5);
    \coordinate [label=below:1] (1) at (0,0,1);
    \coordinate [label=right:2] (2) at ({.5*sqrt(3)},0,-.5);
    \coordinate [label=above:3] (3) at (0,{sqrt(2)},0);
    \coordinate [label=left:0] (0b) at ({-.5*sqrt(3)},-1.5,-.5);
    \coordinate [label=below:1] (1b) at (0,-1.5,1);
    \coordinate [label=right:2] (2b) at ({.5*sqrt(3)},-1.5,-.5);
    \coordinate [label=above:3] (3f) at (2,{sqrt(2)},0);
    \coordinate [label=below:1] (1f) at (2,0,1);
    \coordinate [label=right:2] (2f) at ({.5*sqrt(3)+2},0,-.5);
    \coordinate [label=left:0] (0u) at ({-.5*sqrt(3)-1},2,-.5);
    \coordinate [label=right:2] (2u) at ({.5*sqrt(3)-1},2,-.5);
    \coordinate [label=above:3] (3u) at (-1,{sqrt(2)+2},0);
    \draw[fill=red0,fill opacity=.5] (1)--(2)--(3)--cycle;
    \draw[fill=red0,fill opacity=.5] (0)--(2)--(3)--cycle;
    \draw[fill=red0,fill opacity=.5] (0)--(1)--(2)--cycle;
    \draw[fill=red0,fill opacity=.5] (0)--(1)--(3)--cycle;
    \draw[postaction={decorate}] (0)--(1);
    \draw[postaction={decorate}] (1)--(2);
    \draw[postaction={decorate}] (0)--(2);
    \draw[postaction={decorate}] (2)--(3);
    \draw[postaction={decorate}] (0)--(3);
    \draw[postaction={decorate}] (1)--(3);
    \draw[postaction={decorate}] (0b)--(1b);
    \draw[postaction={decorate}] (1b)--(2b);
    \draw[postaction={decorate}] (2b)--(0b);
    \draw[fill=magenta0,fill opacity=.5] (0b)--(1b)--(2b)--cycle;
    \draw[dashed, |->] (0,-1,0) to[bend right] (0,0,0);
    \draw[postaction={decorate}] (1f)--(2f);
    \draw[postaction={decorate}] (2f)--(3f);
    \draw[postaction={decorate}] (3f)--(1f);
    \draw[fill=magenta0,fill opacity=.5] (1f)--(2f)--(3f)--cycle;
    \draw[dashed, |->] (1.5,{0.3*sqrt(2)},-.25) to[bend right] (0.25,{0.3*sqrt(2)},-.25);
    \draw[postaction={decorate}] (0u)--(2u);
    \draw[postaction={decorate}] (2u)--(3u);
    \draw[postaction={decorate}] (3u)--(0u);
    \draw[fill=magenta0,fill opacity=.5] (0u)--(2u)--(3u)--cycle;
    \draw[dashed, |->] (-1,1.75,-.5) to[bend right=10] (-.15,{.25*sqrt(2)},-.5);
  \end{tikzpicture}
\end{center}

Therefore, the matching map is the expected inclusion
$(\Lambda^3_2)_2 \to \Delta^3_2$, as determined by the dashed arrows above.
Putting these all together, the fact that $\Lambda^3_2$ is uniquely determined
by its restriction $\res_{<2}\Lambda^3_2$ along with the matching map
$(\Lambda^3_2)_2 \to \M_2\Lambda^3_2 = \Delta_2^3$ simply says that the horn
$\Lambda^3_2$ is constructed from the blue skeleton of edges above by first
freely filling in the faces (and the 3$d$-interior) and then selecting which of
the freely filled-in faces (and 3$d$-interior) to keep as its simplices of
dimension 2 (and 3).

Generalising this example of the 3-horn to arbitrary semi-simplicial sets, we
observe that for each inverse category $\McI$ and category $\McE$ along with
$n \in \bN$, diagrams $X_{\leq n} \in \McE^{\McI_{\leq n}}$ are in unique
correspondence with diagrams $X_{<n} \in \McE^{\McI_{<n}}$, a functor
$X_{=n} \in \McE^{\bG_n(\McI)}$ and a map $X_{=n} \to \M_nX_{<n}$.
Packaging everything together, we note that a triple of data
$(X_{<n} \in \McE^{\McI_{<n}}, X_{=n} \in \McE^{\bG_n(\McI)}, X_{=n} \to
\M_nX_{<n})$ is exactly an element of the comma category
$\McE^{\bG_n(\McI)} \downarrow \M_n$, which is precisely the Artin gluing
category $\Gl(\M_n)$.

In this vocabulary, as observed by \textcite{shu15}, $\McE^{\McI_{\leq n}}$ is
the Artin gluing category along the $n$-th matching object functor and diagrams
$\McE^\McI$ are constructed by iterated Artin gluing.
In the case where $\McR$ is a Reedy category such as in the case of the simplex
category, diagrams $\McE^\McR$ can be likewise constructed by induction on the
degree, but $\McE^{\McR_{\leq n}}$ is instead a \emph{bigluing category} in the
sense of \cite[Definition 3.1]{shu15}.
The notion of \emph{c-Reedy categories} of \cite[Definition 8.5]{shu15}
generalises categories where one may construct diagrams by way of iterated
bigluing.
Following \textcite{shu15}, we work in the framework of \emph{iterated gluing
  categories} defined below.
Iterated gluing categories are adaptations of \emph{c-Reedy categories} of
\textcite{shu15} to the case of inverse categories.

\begin{definition}\label{def:prof-collage}
  A profunctor $H \colon \McC \slashedrightarrow \McD$ between categories $\McC$
  and $\McD$ is a functor $H \colon \McC^\op \times \McD \to \Set$.
  Its collage $[H]$ consists of the same objects of the disjoint union
  $\McC \sqcup \McD$ and all maps of both $\McC$ and $\McD$ included in this
  disjoint union along with the $\Hom\relax$-sets
  \begin{align*}
    [H](c,d) = H(c,d) \text{ for } (c,d) \in (\ob\McC) \times (\ob\McD)
  \end{align*}
  Identities and composition are given by those in $\McC$ and $\McD$, as well as
  the functoriality of $H$.
\end{definition}

\begin{definition}\label{def:iter-glue}
  The data for an \emph{iterated gluing diagram} is given by
  $(\McN,\McI,\partial\McI,\McI^\circ)$, where:
  \begin{itemize}
    \item $\McN$ is a strict inverse category.
    \item $\McI \colon \McN^\op \to \Cat$ is a diagram of categories and
    $\partial\McI \colon \ob\McN \to \Cat$ is a family of categories indexed
    by the objects of $\McN$.
    Here, $\partial\McI_n$ is called the \emph{$n$-th strata} of $\McI$.
    \item
    $\McI^\circ = (\McI^\circ_n \colon \McI_{<n} \slashedrightarrow
    \partial\McI_n)_{n \in \ob\McN}$ is a family profunctors where each
    $\McI^\circ_n$, also called the \emph{$n$-th attaching map}, is a profunctor
    from the weighted colimit
    $\McI_{<n} = \McN(n,-)^\circ \otimes_{\McN^\op} \McI$, where
    $\McN(n,-)^\circ$ is $\McN(n,-)$ with the identity removed, to the $n$-th
    boundary $\partial\McI_n$.
    $\McI_{<n}$ is also called the \emph{$n$-th interior}.
  \end{itemize}
  subject to the condition that for each $n \in \McN$, the category $\McI_n$ is
  the collage of $\McI_n^\circ$
  \begin{align*}
    \McI_n = [\McI_n^\circ]
  \end{align*}
  The \emph{iterated gluing category induced by
    $(\McN,\McI,\partial\McI,\McI^\circ)$} is then given by
  $\McI_\infty \coloneqq \colim_{n \in \McN} \McI_n$
\end{definition}

The above definition says that each $\McI_n \in \Cat$ is constructed inductively
by attaching a boundary $\partial\McI_n \in \Cat$ onto the interior
$\McI_{< n} \in \Cat$ as specified by
$\McI_n^\circ \colon \McI_{< n} \slashedrightarrow \partial\McI_n$.
Objects in $\McI_{< n}$ are viewed as objects of strictly smaller than degree
$n$, while objects in $\partial\McI_n$ are viewed as objects of exactly degree
$n$.
By having $\McI_n = [\McI_n^\circ]$, the objects of $\McI_n$ consist of the
disjoint union $\ob{\McI_{< n}} \sqcup \ob{\partial\McI_n}$.
For $j_1,j_2 \in \ob{\McI_{<n}}$, one has the Hom-set
$\McI_n(j_1,j_2) = \McI_{< n}(j_1,j_2)$ and likewise for
$i_1,i_2 \in \ob{\partial\McI_n}$, one has the Hom-set
$\McI_n(i_1,i_2) = \partial\McI_n(i_1,i_2)$.
Furthermore, for $j \in \ob{\McI_{<n}}$ and $i \in \ob{\partial\McI_n}$, one has
\begin{align*}
  \McI_n(i,j) = \McI_n^\circ(i,j)
\end{align*}
while $\McI_n(j,i) = \emptyset$.
In other words each set $\McI_n(i,j) = \McI_n^\circ(i,j)$ is the set of maps
strictly lowering degree from $i$ to $j$, while $\McI_n(j,i) = \emptyset$ means
that there are no strictly degree-raising maps.
However, given $i,i' \in \partial\McI_n$, the set $\partial\McI_n(i,i')$ are the
degree-preserving maps.

Pictorially, we may view each $n$-th stage $\McI_n$ as obtained from
$\McI_{<n}$, which is itself the amalgamation of all of the smaller stages
$\McI_{n'}$ with $n'<n$, by attaching the $n$-th boundary $\partial\McI_n$ along
the $n$-th interior specified formally by $\McI^\circ$.
Given an element $f^- \in \McI^\circ(i,j)$, the composition $hf^-g$ for
$g \colon i' \to i \in \partial\McI_n$ and $h \colon j \to j' \in \McI_{<n}$ is
given by the functorial action $\McI^\circ(g,h)f^- \in \McI^\circ(i',j')$.
\begin{center}
  \begin{tikzpicture}
    \draw[fill=gray, fill opacity=.25] (0,3) ellipse (3 and 1);
    \node (ipn) at (4, 3) {$\partial\McI_n$};
    \draw[fill=gray, fill opacity=.25] (0,0) ellipse (2 and .75);
    \node (iltn) at (4, 0) {$\McI_{<n}$};
    \draw[->] (ipn) to node[]{--} node[pos=0.5, right]{$\McI_n^\circ$} (iltn);
    \draw [fill=black] (-2,3) node (i') {} node [anchor=east] {$i'$} ellipse (.025 and .025);
    \draw [fill=black] (-1,3) node (i) {} node [anchor=west] {$i$} ellipse (.025 and .025);
    \draw [fill=black] (-1,0) node (j) {} node [anchor=east] {$j$} ellipse (.025 and .025);
    \draw [fill=black] (-.25,0) node (j') {} node [anchor=west] {$j'$} ellipse (.025 and .025);
    \draw[->] (i') -- (i) node[pos=0.5, above] {$g$};
    \draw[->] (j) -- (j') node[pos=0.5, below] {$h$};
    \draw[->] (i) -- (j) node[pos=0.5, right] {$f^- \in \McI^\circ(i,j)$};
    \draw[->] (i') to[bend right=10] node[pos=0.5, left]{$\McI^\circ(i',j')\ni \McI^\circ(g,h)f^- = hf^-g$} (j');
  \end{tikzpicture}
\end{center}

For example, by taking $\McN$ to be an ordinal and each $\partial\McI_n$ to be a
discrete set $\McI_\infty$ is an inverse category and by taking each
$\partial\McI_n$ to be a groupoid, $\McI_\infty$ is a generalised inverse
category.

\begin{proposition}\label{prop:inv-iter-glue}
  Let $(\McI,\deg)$ be the data for a generalised inverse category.
  Then, $(\bN, \McI_{\leq -}, \bG_-(\McI), \McI_{\leq -}(-,-))$ forms the data
  of an iterated gluing diagram.
  Furthermore, $\McI = \McI_{\leq\infty}$ is the iterated gluing category
  induced by the data for the iterated gluing diagram.
\end{proposition}
\begin{proof}
  By \cite[Theorem 4.11]{shu15}, the above data makes $\McI$ into a stratified
  category of height $\omega$ as in \cite[Definition 4.10]{shu15}, which is a
  special case of \Cref{def:iter-glue} where $\McN$ is taken to be an ordinal.
\end{proof}

The primary reason we are interested in iterated gluing categories is because
they generalise the construction of inverse diagrams by way of induction on the
degree.
In particular, we may adapt the matching objects of \Cref{def:inv-mat} from the
case of inverse categories to our present setting like so.
\begin{definition}\label{def:iter-glue-mat}
  Let $(\McN,\McI,\partial\McI,\McI^\circ)$ be the data for an iterated gluing
  diagram $\McI$ stratified by $\McN$.
  Fix a category $\McE$ such that for each $n \in \McN$ and $i \in \partial\McI_n$, all
  limits indexed by $\oint_{j \in \McI_{<n}} \McI_n^\circ(i,j)$ exists in
  $\McE$.

  Denote by $\res_{<n} \colon \McE^{\McI_n} \to \McE^{\McI_{<n}}$ precomposition
  with the inclusion $\McI_{<n} \hookrightarrow \McI_n$ and
  $\cosk_n \colon \McE^{\McI_{<n}} \to \McE^{\McI_n}$ its right Kan extension
  (which exists as $\McE$ is sufficiently complete as described above).
  Further put $t_n \colon \partial\McI_n \hookrightarrow \McI_n$ as the
  inclusion.
  %
  \begin{equation*}
    \begin{tikzcd}
      \McE^{\partial\McI_n}
      &
      \McE^{\McI_n}
      \ar[r, "\res_{<n}", yshift=3]
      \ar[r, "\vdash"{rotate=90}, phantom]
      \ar[l, "(t_n)^*"']
      &
      \McE^{\McI_{<n}}
      \ar[l, "\cosk_n", yshift=-3]
    \end{tikzcd}
  \end{equation*}
  Denote by
  $\M_n \coloneqq (t_n)^* \cdot \cosk_n \colon \McE^{\McI_{<n}} \to
  \McE^{\McI_n} \to \McE^{\partial\McI_n}$ the \emph{$n$-th matching object
    functor} and
  $\m_n \colon (t_n)^* \to \M_n \cdot \res_{<n} = (t_n)^* \cdot \cosk_n \cdot
  \res_{<n}$ the unit of the adjunction $\res_{<n} \dashv \cosk_n$ composed with
  $(t_n)^*$ the \emph{$n$-th matching map}.
  For $i \in \McI_n$ and $X_{<n} \in \McE^{\McI_{<n}}$, we often write
  $\M_iX_{<n}$ to mean $(\M_nX_{<n})i$.
  Also, for $X_n \in \McE^{\McI_n}$ we often write $\M_nX_n$ to mean
  $\M_n(\res_{<n}X)$ and $\m_iX_n$ to mean $(\m_nX_n) \colon X_ni \to \M_iX_n$.
\end{definition}
%

In particular, by the formula for the right Kan extension, we see
that in \Cref{def:iter-glue-mat}, if $X_{<n} \in \McE^{\McI_{<n}}$ and
$i \in \partial\McI_n \hookrightarrow \McI_n$ then
\begin{align*}
  (\cosk_nX_{<n})i &=  \lim((\McI_{<n} \hookrightarrow \McI_n) \downarrow i \to \McI_{<n} \xrightarrow{X_{<n}} \McE) \\
                   &\cong \lim(\oint_{j \in \McI_{<n}} \McI_n^\circ(i,j) \to\McI_{<n} \xrightarrow{X_{<n}} \McE) \\
  (\cosk_nX_{<n})i &\cong \set{\McI_n^\circ(i,-), X_{<n}}
\end{align*}
And so \cite[Theorem 4.5]{shu15} implies the following.
\begin{proposition}\label{prop:iter-glue}
  Let $(\McN,\McI,\partial\McI,\McI^\circ)$ be the data for an iterated gluing
  diagram $\McI$ stratified by $\McN$.
  Fix a category $\McE$.
  If, for $n \in \McN$, all limits indexed by
  $\oint_{j \in \McI_{<n}} \McI^\circ_n(i,j)$ exists in $\McE$ for all
  $i \in \partial\McI_n$ then
  \begin{align*}
    \McE^{\McI_n} \simeq \Gl(\McE^{\McI_{<n}} \xrightarrow{\cosk_n} \McE^{\McI_n} \xrightarrow{(t_n)^*} \McE^{\partial\McI_n}) = \Gl(\M_n)
  \end{align*}
  This equivalence of categories sends each $X_n \in \McE^{\McI_n}$ to
  $\m_nX_n \colon (t_n)^*X_n \to \M_n(\res_{<n}X_n) \in \Gl(\M_n)$ and each
  $X_{=n} \to \M_nX_{<n} \in \McE^{\partial_n\McI}$ for
  $X_{<n} \in \McE^{\McI_{<n}}$ to the unique $X_n \in \McE^{\McI_n}$ that
  extend both $X_{=n}$ and $X_{<n}$.
  \def\endingmark{\qedsymbol}
\end{proposition}
%


%% file: diagrams-logic.tex
\section{Logical Structure in Inverse Diagrams}\label{sec:diagrams-logic}
With all the necessary results established in the previous sections, we are now
ready to tackle the main problem of this paper: construct the subobject
classifier and dependent products in diagram categories indexed by iterated
gluing diagrams, and therefore, by extension, generalised inverse categories.
Additionally, we also investigate conditions under which the dependent product
functor is homotopical.

Throughout this section and the next, we fix
$(\McN,<,\McI,\partial\McI,\McI^\circ)$ data for an iterated gluing diagram and
$\McE$ a category admitting enough limits in the following sense so that the
matching object functors exist.
\begin{assumption}\label{asm:E-enough-limits}
  Limits indexed by each $\oint_{j \in \McI_{<n}} \McI^\circ(i,j)$ for
  $n \in \McN$ and $i \in \partial\McI_n$ exists in $\McE$.
\end{assumption}

We further adopt the following notational conventions throughout this section.
\begin{itemize}
  \item For $n \in \McN$, write $(-)|_n \colon \McE^{\McI_\infty} \to \McE^{\McI_n}$
  for the restriction along $\McI_n \to \McI_\infty$.
  Further, write $(-)|_{<n} \colon \McE^{\McI_n} \to \McE^{\McI_{<n}}$
  for restriction along $\McI_{<n} \hookrightarrow \McI_n$
  and $(-)|_{=n} \colon \McE^{\McI_n} \to \McE^{\partial\McI_n}$
  for restriction along $\partial\McI_n \hookrightarrow \McI_n$.
  When it is obvious from context, we abuse notation by writing $(-)|_{<n}$ and
  $(-)|_{=n}$ for $((-)|_n)|_{<n}$ and $((-)|_n)|_{=n}$ respectively.
  \item For each $\alpha \colon n \to n' \in \McN$, write
  $(-)|_\alpha \colon \McE^{\McI_{<n}} \to \McE^{\McI_{n'}}$ for
  restriction along the colimiting leg
  $\McI_{n'} \to \McN(n,-)^\circ \otimes_{\McN^\op} \McI = \McI_{<n}$.
\end{itemize}

\subsection{Subobject Classifiers}\label{subsec:inv-diag-Omega}
We now construct the subobject classifier and truth map in
$\McE^{\McI_\infty} \simeq \lim_{n \in \McN}\McE^{\McI_n}$.
To do so, we aim to use \Cref{lem:limit-Omega} by constructing subobject
classifiers and truth maps in each $\McE^{\McI_n}$ is a suitably compatible way
so that they assemble into the corresponding logical structure in
$\McE^{\McI_\infty}$.

Specifically, we proceed by induction on $n \in \McN$, so that the task is to
construct the subobject classifier and truth map in each $\McE^{\McI_n}$ with
the assumption that there is already a compatible family of subobject classifier
and truth map constructed for each $\McE^{\McI_{n'}}$ with $\deg{n'} < \deg{n}$.
In order to do so, \Cref{prop:iter-glue} states that
$\McE^{\McI_n} \simeq \Gl(\McE^{\McI_{<n}} \xrightarrow{\M_n}
\McE^{\partial\McI_n})$.
The construction of subobject classifiers and truth maps in gluing categories is
provided in \Cref{thm:glue-Omega}.
In order to apply \Cref{thm:glue-Omega} to the $n$-th absolute matching object
functor $\McE^{\McI_{<n}} \xrightarrow{\M_n} \McE^{\partial\McI_n}$, the
categories $\McE^{\McI_{<n}}$ and $\McE^{\partial\McI_n}$ must admit subobject
classifiers and truth maps.
Because $\McI_{<n}$ is a (weighted) colimit of all those $\McI_{n'}$ with
$\deg{n'} < \deg{n}$, one may express $\McE^{\McI_{<n}}$ as a (weighted) limit
consisting of those $\McE^{\McI_{n'}}$ with $\deg{n'} < \deg{n}$.
By the induction hypothesis, the subobject classifiers and truth maps for each
of these $\McE^{\McI_{n'}}$ are already constructed in a suitably compatible
manner, so \Cref{lem:limit-Omega} assembles them into subobject classifiers and
truth maps in $\McE^{\McI_{<n}}$.
The subobject classifier and truth map of $\McE^{\partial\McI_n}$, on the other
hand, cannot be constructed with general procedures like these because the
$n$-th boundary $\partial\McI_n$ may be any category.
Therefore, in the fully general case of iterated gluing diagrams, we work under
the assumption of the existence of subobject classifier and truth map in
$\McE^{\partial\McI_n}$ to complete the induction step using
\Cref{thm:glue-Omega}.
However, in the generalised inverse case, each $\McE^{\partial\McI_n}$ is a
groupoid, so \Cref{prop:gpd-Omega-ptwise} provides a construction (and therefor existence) of the
subobject classifier and truth map in $\McE^{\partial\McI_n}$.

\begin{theorem}\label{thm:E-iter-glue-Omega}
  Suppose \Cref{asm:E-enough-limits} holds.
  Further assume that $\McE$ has all finite limits and each
  $\McE^{\partial\McI_n}$ is equipped with a subobject classifier
  $\overline{\Omega}_n$ along with a truth value
  $\overline{\true}_n \colon 1 \to \overline{\Omega}_n$.
  Then, $\McE^{\McI_\infty}$ has a subobject classifier $\Omega_\infty$ and a
  truth value $\true_\infty \colon 1 \to \Omega_\infty$.

  For each $i \in \partial\McI_n$, one has the equaliser
  \begin{equation*}
    \begin{tikzcd}[column sep=large]
      {\Omega_\infty [i]_n}
      &
      {\overline{\Omega}_ni \times \M_i\Omega_\infty}
      &[2em]
      {\overline{\Omega}_n i \times \overline{\Omega}_n i}
      &
      {\overline{\Omega}_n i}
      \arrow["\wedge_i", from=1-3, to=1-4]
      \arrow["{\id \times (\chi_{\M_n(\true_\infty)})_{[i]_n}}", from=1-2, to=1-3]
      \arrow["{\pi_1}"', bend right=10, from=1-2, to=1-4]
      \arrow[hook, from=1-1, to=1-2]
    \end{tikzcd} \in \McE
  \end{equation*}
  where $\chi_{\M_n(\true_\infty)}$ is the characteristic map of
  $\M_n(\true_\infty) \colon \M_n1 \cong 1 \hookrightarrow \M_n\Omega_\infty
  \in \McE^{\partial\McI_n}$ and
  \begin{equation*}
    \begin{tikzcd}
      & 1 \ar[d, "(\true_\infty)_{[i]_n}"{description}]
      \ar[ddl, "(\overline{\true}_n)_i"', bend right] \ar[ddr, "\M_i(\true_\infty)", bend left]
      \\
      & \Omega_\infty [i]_n \ar[d, hook]
      \\
      \overline{\Omega}_ni
      &
      \overline{\Omega}_ni \times \M_i\Omega_\infty \ar[r, "\pi_2"'] \ar[l, "\pi_1"]
      &
      \M_i\Omega_\infty
    \end{tikzcd} \in \McE
  \end{equation*}
\end{theorem}
\begin{proof}
  By definition, $\McI_\infty = \colim_{n \in \McN} \McI_n$ and so
  $\McE^{\McI_\infty} = \lim_{n \in \McN}\McE^{\McI_n}$.
  We aim to use \Cref{lem:limit-Omega} by constructing subobject classifiers
  $\Omega_n \in \McE^{\McI_n}$ and truth values
  $\true_n \colon 1 \to \Omega_n \in \McE^{\McI_n}$ for each $n \in \McN$
  in such a way that if $\alpha \colon n \to n'$ then
  $\McE^{\McI_\alpha} \colon \McE^{\McI_n} \to \McE^{\McI_{n'}}$ is such that
  $\McE^{\McI_\alpha}(\true_n) = \true_{n'}$.

  Because $\McN$ is inverse, we proceed by induction on the degree of objects.
  Assume that, for $n \in \McN$ fixed, the required subobject classifiers
  $\Omega_{n'}$ and truth maps
  $\true_{n'} \colon 1 \to \Omega_{n'} \in \McE^{\McI_{n'}}$ have all been
  constructed for each $n' \in \McN$ with $\deg{n'} < \deg{n}$ such that if
  $\alpha \colon n_2' \to n_1' \in \McN$ for $n_1',n_2' < n$ then
  $\McE^{\McI_\alpha}(\true_{n_2'}) = \true_{n_1'}$.
  Then, by \Cref{lem:limit-Omega},
  \begin{align*}
    \McE^{\McI_{<n}} = \McE^{\colim((n/\McN)^\op \to \McN^\op \xrightarrow{\McI}
    \Cat)} \simeq \lim(n/\McN \to \McN \xrightarrow{\McE^{\McI_-}} \Cat)
  \end{align*}
  admits a subobject classifier $\Omega_{<n}$ and truth map
  $\true_{<n} \colon 1 \to \Omega_{<n}$ such that for each
  $\alpha \colon n \to n' \in \McN$, one has $\true_{<n}|_\alpha = \true_{n'}$.
  By assumption, $\McE^{\partial\McI_n}$ has a subobject classifier
  $\overline{\Omega}_n$ and truth map
  $\overline{\true}_n \colon 1 \to \overline{\Omega}_n$.
  Hence, by \Cref{thm:glue-Omega}, the gluing category
  $\Gl(\M_n \colon \McE^{\McI_{<n}} \to \McE^{\partial\McI_n})$ has a subobject
  classifier and truth map.
  The subobject classifier is given by
  $\Omega_{=n} \hookrightarrow \overline{\Omega}_n \times \M_n\Omega_{<n} \xrightarrow{\pi_2}
  \M_n\Omega_{<n}$ in which the first map is the equalising map
  \begin{equation*}
    \begin{tikzcd}[column sep=huge]
      {\Omega_{=n}} & {\overline{\Omega}_n \times \M_n\Omega_{<n}} & {\overline{\Omega}_n \times \overline{\Omega}_n} & {\overline{\Omega}_n}
      \arrow["\wedge", from=1-3, to=1-4]
      \arrow["{\id \times \chi_{\M_n(\true_{<n})}}", from=1-2, to=1-3]
      \arrow["{\pi_1}"', bend right=10, from=1-2, to=1-4]
      \arrow[hook, from=1-1, to=1-2]
    \end{tikzcd}
  \end{equation*}
  where $\chi_{\M_n(\true_{<n})}$ is the classifying map of
  $\M_n(\true_{<n}) \colon \cong 1 \M_n1 \hookrightarrow \M_n\Omega_{<n}$.
  And the truth map is given by $\true_{=n} \colon 1 \to \Omega_{=n}$ induced by
  $(1 \xrightarrow{\overline{\true}_n} \overline{\Omega}_n, 1 \cong \M_n1
  \xrightarrow{\M_n(\true_{<n})} \M_n(\Omega_{<n}))$.
  Therefore, as in \Cref{eqn:gl-Omega}, we have
  \begin{equation*}
    \begin{tikzcd}[column sep=large]
      1 \ar[d, "\cong"] \ar[r, dashed, "\true_{=n}"] & \Omega_{=n} \ar[d] \\
      \oM_n1 \ar[r, dashed, "\M_n(\true_{<n})"'] & \M_n\Omega_{<n}
    \end{tikzcd} \in \Gl(\M_n)
  \end{equation*}
  By \Cref{prop:iter-glue}, the subobject classifier and truth map above gives
  rise to a subobject classifier and truth map
  $\true_n \colon 1 \to \Omega_n \in \McE^{\McI_n}$ that extends
  $\true_{<n} \colon 1 \to \Omega_{<n} \in \McE^{\McI_{<n}}$.
  But for a map $\alpha \colon n \to n' \in \McN$ one has
  $\McI_\alpha = \McI_{n'} \to \McI_{<n} \to \McI_n$ where the first map is the
  colimiting leg, because $\McI_{<n}$ is the weighted colimit
  $\McI_n = \colim((n/\McN)^\op \to \McN^\op \xrightarrow{\McI} \Cat)$.
  Hence,
  $\McE^{\McI_\alpha} = \McE^{\McI_n} \to \McE^{\McI_{<n}} \to
  \McE^{\McI_{n'}}$.
  But by construction, $\true_n \colon 1 \to \Omega_n \in \McE^{\McI_n}$ extends
  $\true_{<n} \colon 1 \to \Omega_{<n} \in \McE^{\McI_n}$, and
  $(\true_{<n})|_\alpha = \true_{n'} \colon 1 \to \Omega_{n'}$.
  Hence, this completes the inductive argument.
\end{proof}

\begin{corollary}\label{cor:E-inv-Omega}
  Suppose
  $(\McN,\McI,\partial\McI,\McI^\circ) = (\bN,\McI_{\leq -}, \bG_-(\McI),
  \McI_{\leq-}(-,-))$ is the iterated gluing data for a generalised inverse
  category $\McI$.
  Further assume that each groupoid $\bG_n(\McI)$ is connected or $\McE$ has an
  initial object.
  If $\McE$ has all finite limits, a subobject classifier $\Omega_\McE$ and a
  truth value $\true_\McE \colon 1_\McE \to \Omega_\McE$ then $\McE^\McI$ has a
  subobject classifier $\Omega$ and a truth value $\true \colon 1 \to \Omega$.

  Moreover, at each $i \in \McI$, one has the equaliser
  \begin{equation*}
    \begin{tikzcd}[column sep=large]
      {\Omega i} & {\Omega_\McE \times \M_i\Omega} & {\Omega_\McE \times \Omega_\McE} & {\Omega_\McE}
      \arrow["\wedge", from=1-3, to=1-4]
      \arrow["{\id \times \chi_{\M_i(\true)}}", from=1-2, to=1-3]
      \arrow["{\pi_1}"', bend right=10, from=1-2, to=1-4]
      \arrow[hook, from=1-1, to=1-2]
    \end{tikzcd}
  \end{equation*}
  where $\chi_{\M_i(\true)}$ is the characteristic map of
  $\M_i(\true) \colon \M_i1 \cong 1 \hookrightarrow \M_i\Omega_\infty
  \in \McE$ and
  \begin{equation*}
    \begin{tikzcd}
      & 1 \ar[d, "\true_i"{description}]
      \ar[ddl, "\true_\McE"', bend right] \ar[ddr, "\M_i(\true)", bend left]
      \\
      & \Omega i \ar[d, hook]
      \\
      \Omega_i
      &
      \Omega_\McE \times \M_i\Omega \ar[r, "\pi_2"'] \ar[l, "\pi_1"]
      &
      \M_i\Omega
    \end{tikzcd}
  \end{equation*}
\end{corollary}
\begin{proof}
  For each $n \in \bN$, because either $\bG_n(\McI)$ is connected or $\McE$ has
  an initial object, and $\McE$ has a subobject classifier $\Omega_\McE$ along
  with a truth map $\true_\McE$, it follows by \Cref{prop:gpd-Omega-ptwise} that
  $\McE^{\bG_n(\McI)}$ has a subobject classifier and truth map given
  respectively by the constant diagram at $\Omega_\McE$ and constant natural
  transformation at $\true_\McE$.
  Thus, \Cref{thm:E-iter-glue-Omega} applies to conclude the result.
\end{proof}

\subsection{Dependent Products}
We now construct the dependent products in
$\McE^{\McI_\infty} \simeq \lim_{n \in \McN}\McE^{\McI_n}$.
The approach is much similar to the one for the subobject classifier and truth
map in \Cref{subsec:inv-diag-Omega} by
constructing dependent products in each $\McE^{\McI_n}$ is a suitably compatible way
so that they assemble into dependent products in
$\McE^{\McI_\infty}$ using \Cref{cor:limit-Pi}.

Like before, we proceed by induction on $n \in \McN$, so that the task is to
construct dependent products in each $\McE^{\McI_n}$ with the assumption that
there is already a compatible family of dependent products constructed for each
$\McE^{\McI_{n'}}$ with $\deg{n'} < \deg{n}$.
In order to do so, \Cref{prop:iter-glue} states that
$\McE^{\McI_n} \simeq \Gl(\McE^{\McI_{<n}} \xrightarrow{\M_n}
\McE^{\partial\McI_n})$.
The construction of dependent products in gluing categories is
provided in \Cref{thm:glue-Pi}.
In order to apply \Cref{thm:glue-Pi} to the $n$-th absolute matching object
functor $\McE^{\McI_{<n}} \xrightarrow{\M_n} \McE^{\partial\McI_n}$, the
categories $\McE^{\McI_{<n}}$ and $\McE^{\partial\McI_n}$ must dependent products.
Because $\McI_{<n}$ is a (weighted) colimit of all those $\McI_{n'}$ with
$\deg{n'} < \deg{n}$, one may express $\McE^{\McI_{<n}}$ as a (weighted) limit
consisting of those $\McE^{\McI_{n'}}$ with $\deg{n'} < \deg{n}$.
By the induction hypothesis, the dependent products for each of these
$\McE^{\McI_{n'}}$ are already constructed in a suitably compatible manner, so
\Cref{cor:limit-Pi} assembles them into dependent products in
$\McE^{\McI_{<n}}$.
The dependent product of $\McE^{\partial\McI_n}$, on the other
hand, cannot be constructed with general procedures like these because the
$n$-th boundary $\partial\McI_n$ may be any category.
Therefore, in the fully general case of iterated gluing diagrams, we work under
the assumption of the existence of dependent products in
$\McE^{\partial\McI_n}$ to complete the induction step using
\Cref{thm:glue-Pi}.
However, in the generalised inverse case, each $\partial\McI_n$ is a groupoid,
so \Cref{thm:all-invert-Pi-htpy} provides a construction (and therefore
existence) of the dependent product in $\McE^{\partial\McI_n}$.

\begin{theorem}\label{thm:E-iter-glue-Pi}
  Fix $f \colon B \to A \in \McE^{\McI_\infty}$.
  Suppose \Cref{asm:E-enough-limits} holds and further assume that for each
  $n \in \McN$, the maps
  $f|_{=n} \colon B|_{=n} \to A|_{=n} \in \McE^{\partial\McI_n}$ and
  $\M_nf \colon \M_nB \to \M_nA \in \McE^{\partial\McI_n}$ are powerful.
  Then, the dependent product $\Pi_B$ exists in $\McE^{\McI_\infty}$.

  For $g \colon C \to B \in \McE^{\McI_\infty}/B$ and $i \in \partial\McI_n$
  value of $\Pi_Bg \colon \Pi_BC \to A \in \sfrac{\McE^{\McI_\infty}}{A}$ at
  $[i]_n \in \McI_\infty$ is given by the pullback
  \begin{equation*}
    \begin{tikzcd}[column sep=huge]
      {(\Pi_BC)[i]_n}
      \ar[d] \ar[rr]
      \ar[rd, phantom, "\lrcorner"{pos=0}]
      & &
      (\Pi_{B|_{=n}}{C|_{=n}})i
      \ar[d, "{(\Pi_{B|_n}(g|_n,\m_nC))_i}"]
      \\
      {A[i]_n \times_{\M_i A} \M_i(\Pi_B C)}
      \ar[r, "{A[i]_n \times_{\M_i A} \M_i(\ev)^\ddagger}"']
      &
      {A[i]_n \times_{\M_i A} (\Pi_{\M_n B}{\M_nC})i}
      \ar[r, "{(f_{[i]_n} \times_{\M_i B} \ev_i)^\ddagger}"']
      &
      (\Pi_{B|_{=n}}(B|_{=n} \times_{\M_n B} \M_nC))i
    \end{tikzcd}
    \in \McE
  \end{equation*}
\end{theorem}
\begin{proof}
  We follow an argument similar to \Cref{thm:E-iter-glue-Omega}.
  Because $\McE^{\McI_\infty} = \lim_{n \in \McN} \McE^{\McI_n}$, we aim to use
  \Cref{cor:limit-Pi} by constructing a compatible family of functors
  $(\Pi_{B|_n} \colon \sfrac{\McE^{\McI_n}}{B|_n} \to
  \sfrac{\McE^{\McI_n}}{A|_n})_{n \in \McN}$.
  Specifically, we construct such a family $(\Pi_{B|_n})_{n \in \McN}$ of
  functors each of which is right adjoint to the pullback
  $((f|_n)^* \colon \sfrac{\McE^{\McI_n}}{A|_n} \to
  \sfrac{\McE^{\McI_n}}{B|_n})_{n \in \McN}$ such that for each
  $\alpha \colon n \to n'$, one has $\McE^{\McI_\alpha}(\Pi_{B|_n}) = \Pi_{B|_{n'}}$ and
  $\McE^{\McI_\alpha}(\ev_n) = \ev_{n'} \colon (f|_{n'})^*\Pi_{B|_{n'}} \to \id \in
  \sfrac{\McE^{\McI_{n'}}}{B|_{n'}}$, where $\ev_n$ is the counit.

  Since $\McN$ is inverse, we proceed by induction on the degree of objects.
  Assume that, for $n \in \McN$ fixed, the required dependent product functors
  $\Pi_{B|_{n'}}$ and counits $\ev_{n'}$ have all been constructed for each
  $n' < n$ such that if $\alpha \colon n_2' \to n_1' \in \McN$ for $n_1',n_2' < n$
  then $\McE^{\McI_\alpha}(\Pi_{B|_{n_2'}}) = \Pi_{B|_{n_1'}}$ and
  $\McE^{\McI_\alpha}(\ev_{n_2'}) = \ev_{n_1'}$.
  Then, by \Cref{cor:limit-Pi},
  \begin{align*}
    \McE^{\McI_{<n}} = \McE^{\colim((n/\McN)^\op \to \McN^\op \xrightarrow{\McI}
    \Cat)} \simeq \lim(n/\McN \to \McN \xrightarrow{\McE^{\McI_-}} \Cat)
  \end{align*}
  admits dependent products along
  $f|_{<n} \colon B|_{<n} \to A|_{<n}$ given by the dependent product functor
  $\Pi_{B|_{<n}} \colon \sfrac{\McE^{\McI_{<n}}}{A|_{<n}} \to
  \sfrac{\McE^{\McI_{<n}}}{B|_{<n}}$.
  %
  Moreover, one has that $(\Pi_{B|_{<n}})|_\alpha = \Pi_{B|_{n'}}$ for each
  $\alpha \colon n \to n'$ and the counit
  $\txtev|_{<n} \colon (f|_{<n})^*\Pi_{B|_{<n}} \to \id$ is such that
  $(\txtev|_{<n})|_\alpha = \ev_{n'}$ for each $\alpha \colon n \to n'$.
  %
  And by assumption,
  $f|_{=n} \colon B|_{=n} \to A|_{=n} \in \McE^{\partial\McI_n}$ is powerful.
  Hence, by \Cref{thm:glue-Pi}, the gluing category
  $\Gl(\M_n \colon \McE^{\McI_{<n}} \to \McE^{\partial\McI_n})$ admits
  dependent products along
  \begin{equation*}
    \begin{tikzcd}
      B|_{=n} \ar[r, "{f|_{=n}}",dashed] \ar[d, "\m_nB"'] & A|_{=n} \ar[d, "\m_nA"] \\
      \M_nB \ar[r, "\M_nf"',dashed] & \M_nA
    \end{tikzcd} \in \Gl(\M_n \colon \McE^{\McI_{<n}} \to \McE^{\partial\McI_n})
  \end{equation*}
  such that the projection $\Gl(\M_n) \to \McE^{\McI_{<n}}$ preserves the
  counit.
  In particular, under $\McE^{\McI_n} \simeq \Gl(\M_n)$ as from \Cref{prop:iter-glue},
  $f|_n \colon B|_n \to A|_n \in \McE^{\McI_n}$ is powerful.
  For each $g \colon C \to B \in \McE^{\McI_n}$ , one has, by
  \Cref{eqn:Pi-pullback}, the pullback
  \begin{equation*}
    \begin{tikzcd}[column sep=huge]
      (\Pi_{B|_n}C)|_{=n}
      \ar[d] \ar[rr]
      \ar[rd, phantom, "\lrcorner"{pos=0}]
      & &
      \Pi_{B|_{=n}}{C|_{=n}}
      \ar[d, "{\Pi_{B|_{=n}}(g|_{=n},\m_nC)}"]
      \\
      A|_{=n} \times_{\M_n A} \M_n(\Pi_{B|_{<n}}C|_{<n})
      \ar[r, "A|_{=n} \times_{\M_n A} \M_n(\ev)^\ddagger"']
      &
      A|_{=n} \times_{\M_n A} \Pi_{\M_n B}{\M_nC}
      \ar[r, "(f|_{=n} \times_{\M_n B} \ev)^\ddagger"']
      &
      \Pi_{B|_{=n}}(B|_{=n} \times_{\M_n B} \M_nC)
    \end{tikzcd}
    \in \McE^{\partial\McI_n}
  \end{equation*}
  with the matching map for $\Pi_{B|_n}C \in \McE^{\McI_n}$ being
  $(\Pi_{B|_n}C)|_{=n} \to A|_{=n} \times_{\M_n A} \M_n(\Pi_{B|_{<n}}C|_{<n})
  \to \M_n(\Pi_{B|_{<n}}C|_{<n})$.
  In particular, $\Pi_{B|_n}C \in \McE^{\McI_n}$ extends
  $\Pi_{B|_{<n}}C|_{<n}\McE^{\McI_{<n}}$.
  Because for $\alpha \colon n \to n'$ the map
  $\McI_\alpha = \McI_{n'} \to \McI_{<n} \to \McI_n$ where the first map is the
  colimiting leg, it follows that
  $\McE^{\McI_\alpha} = \McE^{\McI_n} \to \McE^{\McI_{<n}} \to \McE^{\McI_{n'}}$,
  where the second map is the limiting leg formed by precomposing with the
  colimiting leg $\McI_{n'} \to \McI_{<n}$.
  But by construction, $\Pi_{B|_n}C \in \McE^{\McI_n}$ extends
  $\Pi_{B|_{<n}}C|_{<n} \in \McE^{\McI_{<n}}$, which means
  $\McE^{\McI_\alpha}(\Pi_{B|_n}C) = ((\Pi_{B|_n}C)|_{<n})|_\alpha =
  (\Pi_{B|_{<n}}C|_{<n})|_\alpha = \Pi_{B|_{n'}}C|_{n'}$.
  Furthermore, by the fact that the projection $\Gl(\M_n) \to \McE^{\McI_{<n}}$
  preserves the counit, the counit $\ev_n \colon (f|_n)^*\Pi_{B|_n} \to \id$ in
  $\McE^{\McI_n}$ extends the counit
  $\ev_{<n} \colon (f|_{<n})^*\Pi_{B|_{<n}} \to \id$.
  Therefore, $\McE^{\McI_\alpha}(\ev_n) = (\ev|_{<n})|_\alpha = \ev_{n'}$,
  completing the inductive argument.
\end{proof}

\begin{corollary}\label{cor:I-Pi}
  Suppose
  $(\McN,\McI,\partial\McI,\McI^\circ) = (\bN, \McI_{\leq
    -},\bG_-(\McI),\McI_{\leq -}(-,-))$ is the iterated gluing data for a
  generalised inverse category $\McI$.
  Fix $f \colon B \to A \in \McE^\McI$.
  Further assume that
  \begin{enumerate}
    \item Each component $f_i \colon B_i \to A_i \in \McE$ for each $i \in \McI$ is
    powerful.
    \item The map
    $\M_if \colon \M_iB|_n \to \M_iA|_n \in \McE$ is
    is powerful for each $n \in \bN$ and $i \in \bG_n(\McI)$.
    \item Each $\bG_n(\McI)$ is connected or $\McE$ has an initial object.
  \end{enumerate}
  Then, pulling back along $f$ admits a right adjoint $\Pi_B$.
  For each $g \colon C \to B \in \McE^{\McI}/B$, the value of
  $\Pi_Bg \colon \Pi_BC \to A \in \McE^{\McI}/A$ at $i \in \McI$ with degree $n$
  is given by the pullback
  \begin{equation}\label{eqn:I-Pi-pb}\tag{\textsc{$\McI$-$\Pi$-pb}}
    \begin{tikzcd}[column sep=huge]
      (\Pi_BC)i
      \ar[d] \ar[rr]
      \ar[rd, phantom, "\lrcorner"{pos=0}]
      & &
      \Pi_{Bi}{Ci}
      \ar[d, "{\Pi_{Bi}(g_i,\m_iC)}"]
      \\
      A_i \times_{\M_i A} \M_i(\Pi_B C)
      \ar[r, "A_i \times_{\M_i A} \M_i(\ev)^\ddagger"']
      &
      A_i \times_{\M_i A} \Pi_{\M_i B}{\M_iC}
      \ar[r, "(f_i \times_{\M_i B} \ev)^\ddagger"']
      &
      \Pi_{Bi}(B_i \times_{\M_i B} \M_iC)
    \end{tikzcd}
    \in \McE
  \end{equation}
  %
  %
\end{corollary}
\begin{proof}
  Because each $\M_if \colon \M_iB \to \M_iA \in \McE$ for $i \in \bG_n(\McI)$
  is powerful and either $\bG_n(\McI)$ is a connected groupoid or $\McE$ has an
  initial object, each $\M_nf \colon \M_nB \to \M_nA \in \McE^{\bG_n(\McI)}$ is
  powerful by \Cref{thm:all-invert-Pi-htpy}.
  Therefore, the result follows by \Cref{thm:E-iter-glue-Pi}.
\end{proof}

\section{Homotopical Dependent Products in Inverse Diagrams}\label{sec:homotopical-Pi}
In the final part of the paper, we equip $\McI_\infty$ with a wide subcategory
of weak equivalences $\McW$ (satisfying the 2-of-3 property) and compare the
behaviour of dependent products in $\McE^{\Ho\McI_\infty}$ and
$\McE^{\McI_\infty}$, where $\Ho\McI_\infty$ is the homotopical category
$\McW^{-1}\McI_\infty$.
Namely, given a map of homotopical diagrams
$f \colon B \to A \in \McE^{\Ho\McI_\infty}$ along with a homotopical diagram
$g \colon C \to B \in \McE^{\Ho\McI_\infty}/B$, we aim to answer when one has an
isomorphism
\begin{align*}
  \gamma^*(\Pi_BC) \cong \Pi_{\gamma^*B}{\gamma^*C} \in \McE^{\Ho\McI}/A
\end{align*}
As it turns out in \Cref{thm:pi-comp-iso}, when $\McE$ is sufficiently complete,
the key is to have the homotopical localisation restrict to an initial functor
$\gamma|_i \colon i/\McI_{<n} \to (i/\Ho\McI_\infty)^\circ$ for each
$i \in \partial\McI_n$, where $(i/\Ho\McI_\infty)^\circ$ is the underslice with
the identity removed.

This is because for any category $\bC$ and $f \colon B \to A \in \McE^\bC$ with
$g \colon C \to B \in \McE^\bC/B$, the dependent product $\Pi_BC$ at each
$c \in \bC$ internalises those compatible families
$(\alpha_f \in \Pi_{Bd}{Cd} ~|~ f \colon c \to d \in c/\bC)$.
Applying this observation to $\McE^{\Ho\McI_\infty}$ and $\McE^{\McI_\infty}$,
one sees that at each
$i \in \partial\McI_n \hookrightarrow \McI_n \to \McI_\infty \to
\Ho\McI_\infty$, one may think of $\Pi_BC \in \McE^{\Ho\McI_\infty}$ and
$\Pi_{\gamma^*B}{\gamma^*C} \in \McE \in \McE^{\McI_\infty}$ as roughly
\begin{align*}
  (\Pi_BC)i
  &\approx \set{\text{compatible families } (\alpha_f \in \Pi_{Bj}{Cj} ~|~ f \colon i \to j \in i/\Ho\McI_\infty)} \\
  (\Pi_{\gamma^*B}{\gamma^*C})i
  &\approx \set{\text{compatible families } (\alpha_f \in \Pi_{Bj}{Cj} ~|~ f \colon i \to j \in i/\McI_\infty)}
\end{align*}
Because $\Ho\McI_\infty$ is $\McI_\infty$ but with the maps $\McW$ formally
inverted, $i/\Ho\McI_\infty$ may contain more maps than $i/\McI_\infty$.
In order to have an isomorphism $(\Pi_BC)i \cong (\Pi_{\gamma^*B}{\gamma^*C})i$,
each compatible family
$(\alpha_f \in \Pi_{Bj}{Cj} ~|~ f \colon i \to j \in i/\McI_\infty)$ must
uniquely determine a compatible family
$(\alpha_f \in \Pi_{Bj}{Cj} ~|~ f \colon i \to j \in i/\Ho\McI_\infty)$.

To see when this is the case, consider the concrete example when
$\McI = \set{0 \leftarrow 2 \rightarrow 1}$ is the inverse $\spncat$ category,
$\McE = \Set$ and $f \colon B \to A$ is the terminal map $! \colon B \to !$.
Then, in $\Set^\spncat$, because there are no non-identity maps with domain 1 in
$\spncat$, one sees that
$(\Pi_{\gamma^*B}{\gamma^*C})1 = (\gamma^*B)^{\gamma^*C}1 = \set{\text{all maps
  } \alpha_{\id\relax} \colon C1 \to B1}$.
On the other hand, because $2 \to 1$ is marked as a weak equivalence, in
$(21)^{-1}\spncat$, there are non-identity maps with domain 1, namely
\begin{tikzcd}[cramped] 1 \ar[r, dashed, "\cong"] & 2 \end{tikzcd} and
\begin{tikzcd}[cramped] 1 \ar[r, dashed, "\cong"] & 2 \ar[r] & 0 \end{tikzcd},
where \begin{tikzcd}[cramped] 1 \ar[r, dashed, "\cong"] & 2 \end{tikzcd} is the
formal inverse to $2 \to 1$.
Therefore, $(\Pi_BC)1 = (B^C)1$ consists of those tuples of maps
$(\alpha_{\id\relax}, \alpha_{(21)^{-1}}, \alpha_{(20)(21)^{-1}})$ such that
\begin{equation*}
  \begin{tikzcd}[column sep=large]
    C1 \ar[d, "C_{12}"', dashed, "\cong"] \ar[r, "\alpha_{\id\relax}"] & B1 \ar[d, "B_{12}", dashed, "\cong"'] \\
    C2 \ar[d, "C_{20}"'] \ar[r, "\alpha_{(21)^{-1}}"{description}] & B2 \ar[d, "B_{20}"] \\
    C0 \ar[r, "\alpha_{(20)(21)^{-1}}"'] & B0
  \end{tikzcd}
\end{equation*}
Because $B,C$ are homotopic, the maps $C_{12}, B_{12}$ are isomorphisms, so a
choice of $\alpha_{\id} \colon C1 \to B1$ fixes
$\alpha_{(21)^{-1}} \colon C2 \to B2$.
However, the issue arises in determining
$\alpha_{(20)(21)-{-1}} \colon C0 \to B0$.
There is no guarantee that $C_{20},B_{20}$ are also isomorphisms, so in general
a choice of $\alpha_{\id\relax}$ does not uniquely fix a compatible family
$(\alpha_{\id\relax}, \alpha_{(21)^{-1}}, \alpha_{(20)(21)^{-1}})$ as above.
The issue arises because 0 is now reachable from 1 after inverting $2 \to 1$
with the new map $(20)(21)^{-1}$ failing to factor through any old map with
domain 1, therefore resulting in $\spncat \downarrow (20)(21)^{-1}$ to be empty.
However, if $2 \to 0$ were also to be marked as an weak equivalence, then each
choice of $\alpha_{\id\relax}$ does uniquely fix a compatible family
$(\alpha_{\id\relax}, \alpha_{(21)^{-1}}, \alpha_{(20)(21)^{-1}})$ (although
doing so means that all maps in $\spncat$ are inverted, and it is already proved
in \Cref{thm:all-invert-Pi-htpy} that in this case, dependent products are
homotopical).

Generalising this example, one sees that exponentials are preserved when each
new map $f \colon i \to j$ in the homotopical category factors as some old map
$i \to j'$ followed by an isomorphism (in the homotopical category)
$j' \cong j$.
Formally, this is encoded by the initiality condition for
$\gamma|_i \colon i/\McI_{<n} \to (i/\Ho\McI_\infty)^\circ$.

For the rest of this section, we structure our approach as follows.
First, we observe in \Cref{def:local-boundary,lem:local-gluing} that for any
category $\bC$ and any $c \in \bC$, diagrams $\McE^{c/\bC}$ are similarly
constructed by gluing along an analogue of the matching object functor.
We also observe in \Cref{lem:local-Pi} conditions such that dependent products
in each $\McE^{c/\bC}$ assemble to form dependent products in $\McE^\bC$.
Specialising to the case of $\bC \coloneqq \Ho\McI_\infty$, we give a
description of the dependent product in \Cref{lem:ho-I-pi} in terms of the
matching object functors like in \Cref{thm:E-iter-glue-Pi}.
In the diagram below, the pullback in the front face is the construction of the
dependent product in $\McI_\infty$ from \Cref{thm:E-iter-glue-Pi} while the
pullback is the construction of the dependent product to be obtained in
\Cref{lem:ho-I-pi}.
\begin{equation*}
  \begin{tikzcd}[column sep=-4em, nodes={font=\small}]
    (\Pi_BC)[i]_n
    \ar[dd] \ar[rrrr]
    \ar[rrdd, phantom, "\lrcorner"{pos=0}]
    \ar[rd, "\phi_{C,i}"{description}]
    &
    &
    &
    &
    \Pi_{Bi}{Ci}
    \ar[dd]
    \ar[rd, "="]
    \\
    &[-4em]
    (\Pi_{\gamma^*B}{\gamma^*C})i
    \ar[rrrr, crossing over]
    \ar[rrdd, phantom, "\lrcorner"{pos=0}]
    &
    &
    &
    &
    \Pi_{B[i]_n}{C[i]_n}
    \ar[dd]
    \\
    A[i]_n \times_{\oM_{[i]_n} A} \oM_{[i]_n}(\Pi_B C)
    \ar[rr]
    \ar[rd, "{\psi_{C,i}}"{description}]
    &
    &
    A[i]_n \times_{\oM_{[i]_n} A} \Pi_{\oM_{[i]_n} B}{\oM_{[i]_n}C}
    \ar[rr]
    \ar[rd, "\trho_{C,i}", "\cong"']
    &
    &
    \Pi_{B[i]_n}(B[i]_n \times_{\oM_{[i]_n} B} \oM_{[i]_n}C)
    \ar[rd, "\sigma_{C,i}", "\cong"']
    \\
    &
    Ai \times_{\M_i{\gamma^*A}} \M_i(\Pi_{\gamma^*B}\gamma^*C)
    \ar[uu, crossing over, leftarrow]
    \ar[rr]
    &
    &
    A \times_{\M_i{(\gamma^*A)}} \Pi_{\M_i(\gamma^*B)}\M_i(\gamma^*C)
    \ar[rr]
    &
    &
    \Pi_{Bi}(Bi \times_{\M_i(\gamma^*B)} \M_i(\gamma^*C))
  \end{tikzcd}
\end{equation*}
To homotopicality of the dependent product, we construct a canonical natural
comparison map $\varphi$ in \Cref{constr:pi-comp} and show that it is an
isomorphism.
To do so, we construct maps $\psi$, $\trho$ and $\sigma$ in
\Cref{constr:psiCi,constr:pi-bd-iso,constr:sigmaCi} and show that they give rise
to natural isomorphisms between the pullbacks.
This is achieved in \Cref{thm:pi-comp-iso} by using
\begin{itemize}
  \item \Cref{lem:left-face} for the commutativity of the left face
  \item \Cref{lem:bot-face-left} for the commutativity of the bottom left face
  \item \Cref{lem:bot-face-right} for the commutativity of the bottom right face
  \item \Cref{lem:right-face} for the commutativity of the right face
\end{itemize}
Construction of $\trho$ and $\sigma$ as well as verification of the
commutativity of the bottom right face relies on \Cref{asm:bd-initial}, which
states that the homotopical localisation restricts to an initial map in each
underslice.
This assumption allows us to obtain an isomorphism in \Cref{lem:mat-comp-iso}
between matching objects in $\Ho\McI_\infty$ and in $\McI_\infty$, which is a
crucial ingredient in our argument.

%
%

\begin{definition}\label{def:local-boundary}
  For a category $\bC$ and $c \in \bC$, put $(c/\bC)^\circ$ as the full
  subcategory of $c/\bC$ spanned by everything except the initial object
  $\id_c$.
  Then, for any category $\McE$, set the coskeleton functor $\ocosk_c$ as the
  right adjoint to $\ores_c \colon \McE^{c/\bC} \to \McE^{(c/\bC)^\circ}$
  defined as precomposition of the inclusion
  $(c/\bC)^\circ \hookrightarrow c/\bC$, provided that the necessary limits in
  $\McE$ exists.
  For $X \in \McE^{c/\bC}$, write $\oM_c \colon \McE^{c/\bC} \to \McE$ for the composition
  \begin{equation*}
    \oM_c \coloneqq
    \begin{tikzcd}
      \McE^{c/\bC} \ar[r, "\ores_c"]
      &
      \McE^{(c/\bC)^\circ} \ar[r, "\ocosk_c"]
      &
      \McE^{c/\bC} \ar[r, "\overline{\id_c}^*"]
      &
      \McE
    \end{tikzcd}
  \end{equation*}
  where the last map is precomposition with
  $\overline{\id_c} \colon \mbbo{1} \to c/\bC$.
  And put $\om_c \colon \oM_c \to \overline{\id_c}^*$ as the composition with
  the counit of $\ores_c \dashv \ocosk_c$ and $\overline{\id_c}^*$.

  Often, when $X \in \McE^\bC$, we omit $\ores_c$ and write $\oM_cX$ and
  $\om_cX$ for $\oM_c(\ores_c X)$ and $\om_c(\ores_c X)$.
\end{definition}

\begin{lemma}\label{lem:local-gluing}
  Let $\bC$ be any category and $c \in \bC$.
  Then, for any category $\McE$ such that all limits indexed by
  $(c/\bC)^\circ$ exists,
  \begin{align*}
    \McE^{c/\bC} \simeq \Gl(\McE^{(c/\bC)^\circ} \xrightarrow{\ocosk_c} \McE^{c/\bC} \xrightarrow{\overline{\id_c}^*} \McE)
  \end{align*}
  given by mapping $X \in \McE^{c/\bC}$ to
  $\om_cX \colon Xc \to \oM_cX \in \Gl(\overline{\id_c}^* \cdot \ocosk_c)$ is an
  equivalence of categories when the right adjoint $\ocosk_c$ exists.
\end{lemma}
\begin{proof}
  An object of $\Gl(\oM_c)$ is an object $X_{\id} \in \McE$, a diagram
  $X|_c \colon (c/\bC)^\circ \to \McE$ and a map
  $\om_X \colon X_{\id} \to \oM_cX|_c \in \McE$, where $\oM_cX|_c \in \McE$ is given
  by the end
  \begin{align*}
    \oM_cX|_c = \int_{f \colon c \to c' \neq \id} X|_cf
  \end{align*}
  Hence, the map $\om_X \colon X_{\id} \to \oM_cX|_c$ composed with each of the
  limiting legs $\oM_cX|_c \to X|_cf$ for $f \colon c \to c' \neq \id$ gives a
  map compatible family of maps
  $(X_{\id} \to X|_cf)_{f \colon c \to c' \neq \id\relax \in c/\bC}$.
  Compatibility of this family of maps gives rise to a diagram $X \colon c/\bC \to \McE$
  with $X\id_c = X_{\id} \in \McE$ and $Xf = X|_cf$ for each
  $f \colon c \to c' \neq \id$.
  Conversely, each $X \colon c/\bC \to \McE$ gives rise to a cone
  $(X\id_c \to Xf ~|~ f \colon c \to c' \neq \id)$ and hence to a map
  $\om_cX \colon X_{\id} \to Xf$.

  It is also clear that the operation mapping diagrams $X$ in $\McE^{c/\bC}$ to
  their matching maps $\om_cX \in \Gl(\oM_cX)$ is an equivalence of categories.
\end{proof}

\begin{lemma}\label{lem:local-Pi}
  Let $\bC$ be a category and $e \colon c \twoheadrightarrow c' \in \bC$ be an epi.
  For a category $\McE$, denote by
  $\res_e \colon \McE^{c/\bC} \to \McE^{c'/\bC}$ to be precomposition with $m$.
  Suppose $F \colon B \to A \in \McE^{c/\bC}$ and
  $\res_eF \colon \res_eB \to \res_eA \in \McE^{c'/\bC}$ are both powerful.
  Further assume that $\McE$ has an initial object 0.
  Then, one has
  \begin{equation*}
    \begin{tikzcd}[column sep=large]
      \sfrac{\McE^{c/\bC}}{B}
      \ar[r, "F^*", yshift=3]
      \ar[r, phantom, "\vdash"{rotate=90}]
      \ar[d, "\res_e"']
      & \sfrac{\McE^{c/\bC}}{A}
      \ar[l, "\Pi_B", yshift=-3]
      \ar[d, "\res_e"]
      \\
      \sfrac{\McE^{c'/\bC}}{\res_eB}
      \ar[r, "(\res_eF)^*", yshift=3]
      \ar[r, phantom, "\vdash"{rotate=90}]
      & \sfrac{\McE^{c'/\bC}}{\res_eA}
      \ar[l, "\Pi_{\res_eB}", yshift=-3]
    \end{tikzcd}
  \end{equation*}
  %
  %
\end{lemma}
\begin{proof}
  It is clear that pulling back commutes with the restriction because limits are
  computed pointwise.
  It remains to check that the right adjoints commute and that the above diagram
  is a map of adjunctions.

  To do so, first note that $\res_e \colon \McE^{c/\bC} \to \McE^{c'/\bC}$
  admits a left adjoint $L \colon \McE^{c'/\bC} \to \McE^{c/\bC}$.
  By the formula for the left Kan extension, for $D \in \McE^{c'/\bC}$ the
  functor $LD \in \McE^{c/\bC}$ must send $g \colon c \to d$ to the colimit
  \begin{align*}
    (LD)g = \colim(e^* \downarrow g \to c'/\bC \xrightarrow{D} \McE)
  \end{align*}
  where $e^* \colon c'/\bC \to c/\bC$ is precomposition with
  $e \colon c \twoheadrightarrow c'$.
  Objects in the comma category $e^* \downarrow g$ are pairs
  $(c' \xrightarrow{f} x, x \xrightarrow{k} d)$ such that $kfe = g$, like in the
  back face of the diagram below.
  And a map
  $(f \colon c' \to x, k \colon x \to d) \to (f' \colon c' \to x', k' \colon x'
  \to d)$ is a map $h \colon x \to x'$ such that $f' = hf$ and $k = k'h$, like
  the dotted map $h$ below.
  \begin{equation*}
    \begin{tikzcd}
      & & c \ar[ld, "e"', two heads] \ar[dd, "g"] \\
      & c' \ar[ld, "f"', dashed] \\
      x \ar[rr, "k"{description, pos=0.7}, dashed] \ar[rd, "h"', dotted] & & d \\
      & x' \ar[ru, "k'"', dashed]
      \ar[from=2-2, to=4-2, crossing over, "f'"{description, pos=0.3}, dashed]
    \end{tikzcd}
  \end{equation*}
  If $g$ factors through $e$ as $g = g'e$ for some $g' \colon c' \to d$ then
  because $e$ is an epi, $g'$ is unique.
  Thus, $(g' \colon c' \to d, \id \colon d \to d) \in e^* \downarrow g$ is
  terminal.
  This is observed by noting that given any other
  $(c' \xrightarrow{f} x, x \xrightarrow{k} d) \in e^* \downarrow g$, one has
  $kfe = g = g'e$ so $kf = g'$ and clearly $k$ is the only map such that
  $\id k = k$.
  Hence, $e^* \downarrow g$ contains a terminal object when it is not empty and
  because $\McE$ has an initial object 0, one has that
  \begin{align*}
    (LD)g = \begin{cases}
      Dg' & \text{ when $g$ factors via $e$ as $g = g'e$ for some unique map $g'$} \\
      0 & \text{ otherwise }
    \end{cases}
  \end{align*}

  Now, fix $G \colon C \to B \in \McE^{c/\bC}/B$ and the goal is to show that
  \begin{align*}
    \res_e(\Pi_BG \colon \Pi_BC \to A) \cong \Pi_{\res_eB}\res_eG \colon
    \Pi_{\res_eB}\res_eC \to \res_eA
  \end{align*}
  We first note that
  \begin{align*}
    \sfrac{\McE^{c/\bC}}{B}(LD \times_A B, C) \cong
    \sfrac{\McE^{c'/\bC}}{\res_eB}(D \times_{\res_eA} \res_eB, \res_eC)
  \end{align*}
  This is because a natural transformation
  $K \colon LD \times_A B \to C \in \sfrac{\McE^{c/\bC}}{B}$ is a compatible
  family
  $(Kg \colon (LD)g \times_{Ag} Bg \to Cg \in \McE/Bg ~|~ g \colon c \to d \in
  c/\bC)$.
  If $g \colon c \to d \in c/\bC$ factors through
  $e \colon c \twoheadrightarrow c'$ as $g = g'e$ (uniquely) then
  $Kg \colon (LD)g \times_{Ag} Bg \to Cg$ is
  $Kg \colon Dg' \times_{A(g'm)} B(g'm) = Dg' \times_{(\res_eA)g'} (\res_eB)g'
  \to (\res_eC)g'= C(g'm)$.
  And if $g \colon c \to d$ does not factor through
  $m \colon c \twoheadrightarrow c'$ then $(LD)g = 0$ and so
  $Kg \colon (LD)g \times_{Ag} Bg \to Cg$ is just the unique map
  $! \colon 0 \to Cg$.

  Hence, there is the following chain of isomorphisms, natural in $D \in \McE^{c/\bC}/B$:
  \begin{align*}
    \sfrac{\McE^{c'/\bC}}{\res_eA}(D, \res_e(\Pi_BC))
    &\cong
    \sfrac{\McE^{c/\bC}}{A}(LD, \Pi_BC) \\
    &\cong
    \sfrac{\McE^{c/\bC}}{B}(LD \times_A B, C) \\
    &\cong
    \sfrac{\McE^{c'/\bC}}{\res_eB}(D \times_{\res_eA} \res_eB, \res_eC) \\
    \sfrac{\McE^{c'/\bC}}{\res_eA}(D, \res_e(\Pi_BC))
    &\cong
    \sfrac{\McE^{c'/\bC}}{\res_eA}(D, \Pi_{\res_eB}\res_eC)
  \end{align*}
  Which shows that $\res_e(\Pi_BC) \cong \Pi_{\res_eB}{\res_eC}$ as claimed.

  Tracing through the isomorphism
  $\sfrac{\McE^{c'/\bC}}{\res_eA}(D, \res_e(\Pi_BC)) \cong
  \sfrac{\McE^{c'/\bC}}{\res_eA}(D, \Pi_{\res_eB}\res_eC)$ computed above, one
  observes that the identity map at $\res_e(\Pi_BC)$ is first sent to the map
  $L \cdot \res_e(\Pi_BC) \to \Pi_BC \in \sfrac{\McE^{c/\bC}}{A}$ whose
  components at $g$ which factor through $e$ is the identity and whose component
  at $g$ that does not factor through $e$ is the unique map $0 \to \Pi_BC$.
  Because pullbacks are computed pointwise, pulling this map
  $L \cdot \res_e(\Pi_BC) \to \Pi_BC \in \sfrac{\McE^{c/\bC}}{A}$ back along
  $F \colon B \to A \in \McE^{c/\bC}$ and then composing with the counit
  $\epsilon \colon B \times_A \Pi_BC$ gives a map
  $L \cdot \res_e(\Pi_BC) \times_A B \to \Pi_BC \times_A B \to C \in
  \sfrac{\McE^{c/\bC}}{B}$ whose component at $g$ that factor uniquely via $e$
  as $g = g'e$ is the component of the counit
  $\epsilon_{g'} \colon (\Pi_BC)g' \times_{Ag'} Bg' \to Cg'$.
  Under the isomorphism
  $\sfrac{\McE^{c/\bC}}{B}(L \cdot \res_e(\Pi_BC) \times_A B, C) \cong
  \sfrac{\McE^{c'/\bC}}{\res_eB}(\res_e(\Pi_BC) \times_{\res_eA} \res_eB,
  \res_eC)$, this map
  $L \cdot \res_e(\Pi_BC) \times_A B \to \Pi_BC \times_A B \to C \in
  \sfrac{\McE^{c/\bC}}{B}$, whose component at $g$ that factor uniquely via $e$
  as $g = g'e$ is the component of the counit
  $\epsilon_{g'} \colon (\Pi_BC)g' \times_{Ag'} Bg' \to Cg'$, corresponds to the
  map $\res_eB \times_{\res_eA} \res_e(\Pi_BC) \to \res_eC$ whose component at
  $g' \colon c' \to x$ is the component of the counit
  $\epsilon_{g'} \colon (\Pi_BC)g' \times_{Ag'} Bg' \to Cg'$.
  This computation shows that $\res_e$ preserves the counit of the adjunction,
  so one has a map of adjunction, as claimed.
\end{proof}

Now, fix a wide subcategory of weak equivalences $\McW \subseteq \McI_\infty$
satisfying the 2-of-3 property.
Denote by $\Ho{\McI_\infty}$ the homotopical category $\McW^{-1}\McI_\infty$, and from
now, we further work under the following assumptions:
\begin{assumption}\label{asm:Ho-I-epi-E}
  In addition to the limits indexed by each
  $\oint_{j \in \McI_{<n}} \McI^\circ(i,j)$ for $n \in \McN$ and
  $i \in \partial\McI_n$ existing in $\McE$ as in \Cref{asm:E-enough-limits}, we
  further assume that $\McE$ admits all limits indexed by
  $([i]_n/\Ho\McI)^\circ$.
  Moreover, we assume that all maps in $\Ho\McI_\infty$ are epis.
\end{assumption}

Next, we fix a map $f \colon B \to A \in \McE^{\Ho\McI_\infty}$ which we aim to
take dependent products along and assume it has the following properties:
\begin{assumption}\label{asm:ho-powerful}
  Assume that for each $n \in \McN$ and $i \in \partial\McI_n$,
  \begin{itemize}
    \item The restriction
    $f|_{=n} \colon B|_{=n} \to A|_{=n} \in \McE^{\partial\McI_n}$
    \item The functorial action of the $n$-th absolute matching object functor
    $\M_n(\gamma^*f) \colon \M_n(\gamma^*B) \to \M_n(\gamma^*A) \in
    \McE^{\partial\McI_n}$
    \item The component $f_{[i]_n} \colon B[i]_n \to A[i]_n \in \McE$
    \item The map $\oM_{[i]_n}f \colon \oM_{[i]_n}B \to \oM_{[i]_n}A \in \McE$
    as from \Cref{def:local-boundary}
  \end{itemize}
  are powerful.
\end{assumption}

\begin{lemma}\label{lem:ho-I-pi}
  Under \Cref{asm:ho-powerful}, pulling back along $f$ admits a right adjoint
  $\Pi_B$.
  For each $g \colon C \to B \in \McE^{\Ho\McI_\infty}/B$, the value of
  $\Pi_Bg \colon \Pi_BC \to A \in \McE^{\Ho\McI_\infty}/A$ at each
  $[i]_n \in \McI_\infty$ where $i \in \partial\McI_n$ is given by the pullback
  in $\McE$
  \begin{equation*}
    \begin{tikzcd}[column sep=huge]
      (\Pi_BC)[i]_n
      \ar[d] \ar[rr]
      \ar[rd, phantom, "\lrcorner"{pos=0}]
      & &
      \Pi_{B[i]_n}{C[i]_n} \ar[d, "{\Pi_{B[i]_n}(g_{[i]_n},\om_{[i]_n}C)}"']
      \\
      A[i]_n \times_{\oM_{[i]_n} A} \oM_{[i]_n}(\Pi_B C)
      \ar[r, "{A[i]_n \times_{\oM_{[i]_n} A} \oM_{[i]_n}(\ev)^\ddagger}"']
      &[1em]
      A[i]_n \times_{\oM_{[i]_n} A} \Pi_{\oM_{[i]_n} B}{\oM_{[i]_n}C}
      \ar[r, "(f_{[i]_n} \times_{\oM_{[i]_n} B} \ev)^\ddagger"']
      &
      \Pi_{B[i]_n}(B[i]_n \times_{\oM_{[i]_n} B} \oM_{[i]_n}C)
    \end{tikzcd}
  \end{equation*}
  %
  %
\end{lemma}
\begin{proof}
  Noting that
  $\McE^{\Ho\McI_\infty} = \lim_{n \in \McN}\lim_{i \in
    \partial\McI_n}{\McE^{[i]_n/\Ho\McI_\infty}}$ and by \Cref{lem:local-gluing}
  each $\McE^{[i]_n/\Ho\McI_\infty}$ is equivalent to a the gluing category
  \begin{align*}
    \McE^{[i]_n/\Ho\McI_\infty} \simeq
    \Gl(\McE^{([i]_n/\Ho\McI_\infty)^\circ} \xrightarrow{\ocosk_{[i]_n}} \McE^{[i]_n/\Ho\McI_\infty}
    \xrightarrow{\overline{\id_{[i]_n}}^*} \McE)
  \end{align*}
  a similar argument to \Cref{thm:E-iter-glue-Pi} applies.
  This argument relies on \Cref{cor:limit-Pi}.
  To repeat the argument using \Cref{cor:limit-Pi}, instead of proceeding by
  induction to show that the dependent product constructed in each step agrees
  with the dependent product in the restriction, we use the assumption that all
  maps in $\Ho\McI$ are epis as in \Cref{asm:Ho-I-epi-E} and apply
  \Cref{lem:local-Pi} instead.

  In particular, the assumptions in the statement ensure, by \Cref{thm:glue-Pi},
  that fibred sections along $\om_{[i]_n}f$ in each $\Gl(\oM_{[i]_n})$ for
  $i \in \partial\McI_n$ are constructed as described in the statement.
  By \Cref{lem:local-gluing}, it now follows that each
  $\McE^{[i]_n/\Ho\McI_\infty}$ admits dependent products along
  $f|_{[i]} \colon B|_{[i]} \to A|_{[i]}$.
  By \Cref{lem:local-Pi}, one has that these dependent products and their
  counits commute with restrictions, so by \Cref{cor:limit-Pi}, the result
  follows.
\end{proof}

\begin{construction}\label{constr:mat-comp}
  Fix $i \in \partial\McI_n \hookrightarrow \McI_n$ and $X \in \McE^{\Ho\McI}$.
  Then, using the formula for the right Kan extension for $\ocosk_i$ in
  \Cref{def:local-boundary}, the matching object of $X$ at $[i]_n \in \Ho\McI_\infty$ is
  computed as
  \begin{align*}
    \oM_{[i]_n}X = \lim(([i]_n/\Ho\McI_\infty)^\circ \hookrightarrow [i]_n/\Ho\McI_\infty
    \to \Ho\McI_\infty \xrightarrow{X} \McE)
  \end{align*}
  and the matching object of $(\gamma^*X) \in \McE^{\McI}$ at $i \in \McI$ is computed as
  \begin{align*}
    \M_i(\gamma^*X) = \lim(i/\McI_{<n} \to [i]_n/\McI_\infty \to \McI_\infty
    \xrightarrow{\gamma} \Ho\McI_\infty \xrightarrow{X} \McE)
  \end{align*}
  Denote by
  $\overline\pi = (\oM_{[i]_n}X \xrightarrow{\pi_f} Xj ~|~ f \colon [i]_n \to j
  \in ([i]_n/\Ho\McI)^\circ)$ the limiting cone of $\oM_{[i]_n}X$.
  This cone restricts along
  $i/\McI_{<n} \to ([i]_n/\McI_\infty)^\circ \to ([i]_n/\Ho\McI)^\circ$ to a
  cone $\overline\pi|_{i/\McI_{<n}}$ of shape
  $i/\McI_{<n} \to [i]_n/\McI_\infty \to \McI_\infty \xrightarrow{\gamma}
  \Ho\McI_\infty \xrightarrow{X} \McE$.
  By the universal property, this induces a unique comparison map
  $\kappa_{X,i} \colon \oM_{[i]_n}X \to \M_i(\gamma^*X)$ such that when composed with
  the $(f^-)$-th limiting leg $\pi_{f^-} \colon \M_i(\gamma^*X) \to Xj$ for
  $f^- \colon i \to j \in i/\McI_{<n}$ of $\M_i(\gamma^*X)$ gives
  $\overline\pi_{f^-} \colon \oM_iX \to Xj$.
  \begin{equation*}
    \begin{tikzcd}
      \oM_{[i]_n}X \ar[rr, "\kappa_{X,i}"] \ar[rd, "\overline\pi_{[f^-]_n}"']
      & & \M_i(\gamma^*X) \ar[ld, "\pi_{f^-}"] \\
      & Xj
    \end{tikzcd}
  \end{equation*}

  In particular, the diagram whose limit gives rise to $\oM_{[i]_n}X$ is given
  as the top row below, while the diagram whose limit gives rise to $\M_iX$ is
  given as the bottom-right edge below.
  And the map $\kappa_{X,i} \colon \oM_{[i]_n} \to \M_i(\gamma^*X)$ is due to the
  functoriality of the limit via the factorisation along
  $i/\McI_{<n} \to ([i]_n/\Ho\McI_\infty)^\circ$.
  \begin{equation}\label{eqn:mat-comp-fact}\tag{\textsc{mat-comp}}
    \begin{tikzcd}
      ([i]_n/\Ho\McI_\infty)^\circ
      \ar[r, hookrightarrow]
      &
      {[i]_n/\Ho\McI_\infty}
      \ar[r]
      &
      \Ho\McI_\infty
      \ar[r, "X"]
      &
      \McE
      \\
      i/\McI_{<n}
      \ar[u]
      \ar[r]
      &
      [i]_n/\McI_\infty
      \ar[r]
      &
      \McI_\infty
      \ar[r, "\gamma"']
      &
      \Ho\McI_\infty
      \ar[u, "X"']
    \end{tikzcd}
  \end{equation}
\end{construction}

\begin{lemma}\label{lem:mat-comp}
  The comparison map between the matching objects from \Cref{constr:mat-comp} is
  natural in $X \in \McE^{\Ho\McI}$ and
  $i \in \partial\McI_n \hookrightarrow \McI_n$.
  Moreover, one as
  \begin{equation*}
    \begin{tikzcd}
      & X[i]_n \ar[rd, "\m_i(\gamma^*X)"] \ar[ld, "\om_{[i]_n}X"'] & \\
      \oM_{[i]_n}X \ar[rr, "\kappa_{X,i}"'] & & M_i(\gamma^*X)
    \end{tikzcd}
  \end{equation*}
\end{lemma}
\begin{proof}
  It is clear to observe naturality.
  Recall that the matching map $\m_iX \colon X[i] \to \M_i(\gamma^*X)$ composed
  with the $(f^-)$-th limiting leg $\pi_{f^-}$ of $\M_i(\gamma^*X)$ for
  $f^- \colon i \to j \in i/\McI_{<n}$ gives $X[f^-]_n$, which is exactly
  $\om_{[i]_n}X$ composed with the corresponding limiting leg
  $\overline\pi_{[f^-]_n}$ of $\oM_{[i]_n}X$.
  But then by \Cref{constr:mat-comp},
  $\pi_{f^-} \cdot \kappa_{X,i} = \overline\pi_{[f^-]_n}$, and so
  $\pi_{f^-} \cdot \kappa_{X,i} \cdot \om_{[i]_n}X = \pi_{[f^-]_n} \cdot \om_{[i]_n}X$ for each
  $f^- \in i/\McI_{<n}$.
  \begin{equation*}
    \begin{tikzcd}
      & X[i]_n \ar[rd, "{\m_i(\gamma^*X)}"] \ar[ld, "\om_{[i]_n}X"']  & \\
      \oM_{[i]_n}X \ar[rr, "\kappa_{X,i}"{description, pos=0.75}] \ar[rd, "\overline\pi_{f^-}"']
      & & M_i(\gamma^*X) \ar[ld, "{\pi_{[f^-]_n}}"] \\
      & X[j]_n
      \ar[from=1-2, to=3-2, crossing over, "{X[f^-]_n}"{description, pos=0.25}]
    \end{tikzcd}
  \end{equation*}
  Thus, $\m_i(\gamma^*X) = \kappa_{X,i} \cdot \om_{[i]_n}X$.
\end{proof}

\begin{construction}\label{constr:pi-comp}
  By \Cref{lem:ho-I-pi,thm:E-iter-glue-Pi}, under \Cref{asm:ho-powerful},
  both $f \colon B \to A \in \McE^{\Ho\McI_\infty}$
  and $\gamma^*f \colon \gamma^*B \to \gamma^*A \in \McE^{\McI_\infty}$ are powerful.
  So, one has functors
  \begin{align*}
    \gamma^*(\Pi_B-), \Pi_{\gamma^*B}{\gamma^*-} \colon \McE^{\Ho\McI_\infty}/B
    \rightrightarrows \McE^{\McI_\infty}/\gamma^*A
  \end{align*}
  Set $\varphi \colon \gamma^*(\Pi_B-) \to \Pi_{\gamma^*B}{\gamma^*-}$ to be a
  natural transformation such that the transpose of
  $\varphi_C \colon \gamma^*(\Pi_BC) \to \Pi_{\gamma^*B}{\gamma^*C}$ for each
  $g \colon C \to B \in \McE^{\Ho\McI}/B$ over $\gamma^*A$ under the adjunction
  $(f\gamma)^* \dashv \Pi_{\gamma^*B}$ is given by
  \begin{equation}\label{eqn:phi-def}\tag{\textsc{$\varphi$-def}}
    \begin{tikzcd}
      \gamma^*B \times_{\gamma^*A} \gamma^*(\Pi_BC)
      \ar[r, "\cong"]
      \ar[rrd, bend right=10]
      &
      \gamma^*(B \times_A \Pi_BC)
      \ar[r, "\gamma^*(\ev)"]
      \ar[rd]
      &
      \gamma^*C
      \ar[d, "\gamma^*g"]
      &
      \gamma^*(\Pi_BC)
      \ar[r, "\varphi_C"]
      \ar[d, "{\gamma^*(\Pi_Bg)}"']
      &
      \Pi_{\gamma^*B}{\gamma^*C}
      \ar[ld, bend left, "\Pi_{\gamma^*B}{\gamma^*g}"]
      \\
      &
      &
      \gamma^*B
      \ar[r, "\gamma^*f"']
      &
      \gamma^*A
    \end{tikzcd}
  \end{equation}

  For each $i \in \partial\McI_n$, we also have maps
  \begin{equation*}
    \oM_{[i]_n}(\Pi_B-), \M_i(\Pi_{\gamma^*B}\gamma^*-) \colon \McE^{\Ho\McI}/B \rightrightarrows \McE
  \end{equation*}
  Using $\varphi$, define a map
  $\tphi \colon \oM_i(\Pi_B-) \to \M_i(\Pi_{\gamma^*B}{\gamma^*-})$ whose
  component at $g \colon C \to B \in \McE^{\Ho\McI}/B$ is the unique map such
  that for each $u^- \colon i \to j \in i/\McI_{<n}$,
  \begin{equation}\label{eqn:tphi-def}\tag{\textsc{$\tphi$-def}}
    \begin{tikzcd}
      \oM_{[i]_n}(\Pi_BC)
      \ar[r, "\tphi_{C,i}"]
      \ar[d, "{{\overline\pi}_{[u^-]_n}}"']
      &
      \M_i(\Pi_{\gamma^*B}{\gamma^*C})
      \ar[d, "\pi_{u^-}"]
      \\
      (\Pi_BC)[j]_n
      \ar[r, "\varphi_{C,j}"']
      &
      (\Pi_{\gamma^*B}{\gamma^*C})j
    \end{tikzcd}
  \end{equation}
  where $\pi_{u^-}$ and ${\overline\pi}_{u^-}$ are the respective limiting legs.
\end{construction}

\begin{lemma}\label{lem:pi-comp}
  For $\varphi_{C,i}$ and $\tphi_{C,i}$ from \Cref{constr:pi-comp}, we have
  \begin{equation*}
    \begin{tikzcd}
      (\Pi_BC)[i]_n
      \ar[r, "\varphi_{C,i}"]
      \ar[d, "\om_{[i]_n}(\Pi_BC)"']
      &
      (\Pi_{\gamma^*B}{\gamma^*C})i
      \ar[d, "\m_i(\Pi_{\gamma^*B}{\gamma^*C})"]
      \\
      \oM_{[i]_n}(\Pi_BC)
      \ar[r, "\tphi_{C,i}"]
      \ar[d, "\oM_{[i]_n}(\Pi_Bg)"']
      &
      \M_i(\Pi_{\gamma^*B}{\gamma^*C})
      \ar[d, "\M_{[i]_n}(\Pi_{\gamma^*B}\gamma^*g)"]
      \\
      \oM_{[i]_n}A
      \ar[r, "\kappa_{A,i}"]
      &
      \M_i(\gamma^*A)
    \end{tikzcd} \in \McE
  \end{equation*}
  where $\kappa_{A,i}$ is the comparison map between the matching objects from
  \Cref{constr:mat-comp}.
\end{lemma}
\begin{proof}
  For each $u^- \colon i \to j \in i/\McI_{<n}$, one has the following diagram
  on the left for the top square and the diagram on the right for the bottom
  square
  %
  \begin{equation*}
    \begin{tikzcd}[row sep=large]
      (\Pi_BC)[i]_n
      \ar[r, "\varphi_{C,i}"]
      \ar[d, "\om_{[i]_n}(\Pi_BC)" {description}]
      \ar[dd, "{(\Pi_BC)[u^-]_n}"', bend right=60, marking]
      &
      (\Pi_{\gamma^*B}{\gamma^*C})i
      \ar[d, "\m_i(\Pi_{\gamma^*B}{\gamma^*C})" {description}]
      \ar[dd, "(\Pi_{\gamma^*B}{\gamma^*C})u^-"', bend left=60, marking]
      \\
      \oM_{[i]_n}(\Pi_BC)
      \ar[r, "\tphi_{C,i}"]
      \ar[d, "{\overline\pi}_{[u^-]_n}" {description}]
      &
      \M_i(\Pi_{\gamma^*B}{\gamma^*C})
      \ar[d, "\pi_{u^-}" {description}]
      \\
      (\Pi_BC)[j]_n
      \ar[r, "\varphi_{C,j}"]
      &
      (\Pi_{\gamma^*B}{\gamma^*C})j
    \end{tikzcd}
    %
    \begin{tikzcd}[column sep=small]
      {\oM_{[i]_n}(\Pi_BC)} && {\M_{[i]_n}(\Pi_{\gamma^*B}{\gamma^*C})} \\
      & {\oM_{[i]_n}A} && {\M_i(\gamma^*A)} \\
      {(\Pi_BC)[j]_n} && {(\Pi_{\gamma^*B}{\gamma^*C})j} \\
      & A[j]_n && Aj
      \arrow["{\oM_i(\Pi_Bg)}"{description}, from=1-1, to=2-2]
      \arrow["{{\overline\pi}_{[u^-]_n}}"{description}, from=1-1, to=3-1]
      \arrow["{(\Pi_Bg)_{[j]_n}}"{description}, from=3-1, to=4-2]
      \arrow["{=}"{description}, from=4-2, to=4-4]
      \arrow["{(\Pi_{\gamma^*B}{\gamma^*g})_j}"{description}, from=3-3, to=4-4]
      \arrow["{\pi_{u^-}}"{description}, from=2-4, to=4-4]
      \arrow["{\M_i(\Pi_{\gamma^*B}{\gamma^*g})}"{description}, from=1-3, to=2-4]
      \arrow["{\varphi_{C,j}}"{description, pos=0.7}, from=3-1, to=3-3]
      \arrow["{\tphi_{C,i}}"{description}, from=1-1, to=1-3]
      \arrow["{\pi_{u^-}}"{description, pos=0.3}, from=1-3, to=3-3]
      \arrow["{{\overline\pi}_{[u^-]_n}}"{description, pos=0.3}, from=2-2, to=4-2, crossing over]
      \arrow["{\kappa_{A,i}}"{description, pos=0.2}, from=2-2, to=2-4, crossing over]
    \end{tikzcd}
  \end{equation*}
\end{proof}

\begin{construction}\label{constr:psiCi}
  In view of the bottom square of \Cref{lem:pi-comp} and the fact that
  $\kappa_{A,i}$ is under $A[i]_n = Ai$ by \Cref{lem:mat-comp}, one has a map of
  cospans
  \begin{equation*}
    \begin{tikzcd}[row sep=large, column sep=huge]
      A[i]_n
      \ar[r, "\om_{[i]_n}A"]
      \ar[d, "="']
      &
      \oM_{[i]_n}A
      \ar[d, "\kappa_{A,i}"{description}]
      &
      \oM_i(\Pi_BC)
      \ar[l, "\oM_i(\Pi_Bg)"']
      \ar[d, "\tphi_{C,i}"]
      \\
      Ai
      \ar[r, "\m_i(\gamma^*A)"']
      &
      \M_i(\gamma^*A)
      &
      \M_i(\Pi_{\gamma^*B}{\gamma^*C})
      \ar[l, "\M_i(\Pi_{\gamma^*B}{\gamma^*g})"]
    \end{tikzcd}
  \end{equation*}
  which induces a map
  $\psi_{C,i}
  \colon
  A[i]_n \times_{\oM_{[i]_n}A} \oM_{[i]_n}(\Pi_BC)
  \to
  Ai \times_{\M_i(\gamma^*A)} \M_i(\Pi_{\gamma^*B}{\gamma^*C})$.
\end{construction}

\begin{lemma}\label{lem:left-face}
  For $\psi_{C,i}$ the map from \Cref{constr:psiCi} and $\varphi_{C,i}$ the
  map from \Cref{constr:pi-comp}, one has
  \begin{equation*}
    \begin{tikzcd}
      (\Pi_BC)[i]_n
      \ar[r, "\varphi_{C,i}"]
      \ar[d, "{((\Pi_Bg)_{[i]_n}, \om_{[i]_n}(\Pi_BC))}"']
      &
      (\Pi_{\gamma^*B}{\gamma^*C})i
      \ar[d, "{((\Pi_{\gamma^*B}{\gamma^*g})_i, \m_i(\Pi_{\gamma^*B}{\gamma^*C}))}"]
      \\
      A[i]_n \times_{\oM_{[i]_n}A} \oM_{[i]_n}(\Pi_BC)
      \ar[r, "\psi_{C,i}"']
      &
      Ai \times_{\M_i(\gamma^*A)} \M_i(\Pi_{\gamma^*B}{\gamma^*C})
    \end{tikzcd}
  \end{equation*}
\end{lemma}
\begin{proof}
  Composing with the limiting legs of the pullback
  $Ai \leftarrow Ai \times_{\M_i(\gamma^*A)} \M_i(\Pi_{\gamma^*B}{\gamma^*C})
  \to \M_i(\Pi_{\gamma^*B}{\gamma^*C})$ and using \Cref{lem:pi-comp}, it is
  possible to observe
  %
  \begin{equation*}
    \begin{tikzcd}
      {(\Pi_BC)[i]_n} && {(\Pi_{\gamma^*B}{\gamma^*C})i} \\
      \\
      {A[i]_n \times_{\oM_{[i]_n}A} \oM_{[i]_n}(\Pi_BC)} && {Ai \times_{\M_i(\gamma^*A)} \M_i(\Pi_{\gamma^*B}{\gamma^*C})} \\
      & {\oM_{[i]_n}(\Pi_BC)} && {\M_i(\Pi_{\gamma^*B}{\gamma^*C})} \\
      A[i]_n && Ai \\
      & {\oM_{[i]_n}A} && {\M_i(\gamma^*A)}
      \arrow["{=}"{description, pos=0.6}, from=5-1, to=5-3]
      \arrow["{\om_{[i]_n}A}"{description}, from=5-1, to=6-2]
      \arrow["{\kappa_{A,i}}"{description}, from=6-2, to=6-4]
      \arrow["{\m_i(\gamma^*A)}"{description}, from=5-3, to=6-4]
      \arrow[from=3-3, to=5-3]
      \arrow["{\M_i(\Pi_{\gamma^*B}{\gamma^*g})}"{description}, from=4-4, to=6-4]
      \arrow[from=3-3, to=4-4]
      \arrow["\lrcorner"{pos=0}, phantom, from=3-3, to=6-4]
      \arrow[from=3-1, to=4-2]
      \arrow[from=3-1, to=5-1]
      \arrow["\lrcorner"{pos=0}, phantom, from=3-1, to=6-2]
      \arrow[from=1-1, to=3-1]
      \arrow[from=1-3, to=3-3]
      \arrow["{\m_i(\Pi_{\gamma^*B}{\gamma^*C})}"{description}, bend left, from=1-3, to=4-4]
      \arrow["{\varphi_{C,i}}"{description}, from=1-1, to=1-3]
      \arrow["{(\Pi_Bg)_{[i]_n}}"', bend right=90, from=1-1, to=5-1]
      \arrow["{(\Pi_{\gamma^*B}{\gamma^*g})_i}"{description, pos=0.2}, bend right=90, from=1-3, to=5-3]
      \arrow["{\psi_{C,i}}"{description, pos=0.6}, crossing over, from=3-1, to=3-3]
      \arrow["{\om_i(\Pi_BC)}"{description}, crossing over, bend left, from=1-1, to=4-2]
      \arrow["{\oM_{[i]_n}(\Pi_Bg)}"{description, pos=0.3}, crossing over, from=4-2, to=6-2]
      \arrow["{\tphi_{C,i}}"{description, pos=0.3}, crossing over, from=4-2, to=4-4]
    \end{tikzcd}
  \end{equation*}
\end{proof}

\begin{construction}\label{constr:tphi-pb-proj}
  By naturality of $\kappa$ from \Cref{constr:mat-comp} and by the bottom square
  of \Cref{lem:pi-comp}, one has a map of cospans
  \begin{equation*}
    \begin{tikzcd}[column sep=huge]
      \oM_{[i]_n}B \ar[r, "\oM_{[i]_n}f"] \ar[d, "\kappa_{B,i}"']
      &
      \oM_{[i]_n}A \ar[d, "\kappa_{A,i}"{description}]
      &
      \oM_{[i]_n}(\Pi_BC) \ar[l, "\oM_i(\Pi_Bg)"'] \ar[d, "\tphi_{C,i}"]
      \\
      \M_i(\gamma^*B) \ar[r, "\M_i(\gamma^*f)"']
      &
      \M_i(\gamma^*A)
      &
      \M_i(\Pi_{\gamma^*B}{\gamma^*C}) \ar[l, "\M_i(\Pi_{\gamma^*B}{\gamma^*g})"]
    \end{tikzcd}
  \end{equation*}
  which induces a map $\mu_{C,i} \coloneqq (\kappa_{B,i}, \tphi_{C,i})$ as
  follows:
  \begin{align*}
    \oM_{[i]_n}(B \times_A \Pi_BC)
    \cong
    \oM_{[i]_n}B \times_{\oM_{[i]_n}A} \oM_{[i]_n}(\Pi_BC)
    \xrightarrow{\mu_{C,i}}
    \M_i(\gamma^*B) \times_{\M_i(\gamma^*A)} \M_i(\Pi_{\gamma^*B}{\gamma^*C})
    \cong
    \M_i(\gamma^*B \times_{\gamma^*A} \Pi_{\gamma^*B}{\gamma^*C})
  \end{align*}
\end{construction}

\begin{lemma}\label{lem:tphi-pb-proj}
  The $\mu_{C,i}$ from \Cref{constr:tphi-pb-proj} is such that for any
  $u^- \colon i \to j \in i/\McI_{<n}$, one has
  \begin{equation*}
    \begin{tikzcd}
      \oM_{[i]_n}(B \times_A \Pi_BC) \ar[r, "\mu_{C,i}"] \ar[d, "{\overline\pi}_{[u^-]_n}"']
      &
      \M_i(\gamma^*B \times_{\gamma^*A} \Pi_{\gamma^*A}{\gamma^*C})
      \ar[d, "\pi_{u^-}"]
      \\
      B[j]_n \times_{A[j]_n} (\Pi_BC)[j]_n
      \ar[r, "Bj \times_{Aj} \varphi_{C,j}"']
      &
      Bj \times_{Aj} (\Pi_{\gamma^*B}{\gamma^*C})j
    \end{tikzcd}
  \end{equation*}
  where the $\varphi_{C,j}$ is the comparison map of dependent products from
  \Cref{constr:pi-comp}.
  From this, one concludes that
  \begin{equation*}
    \begin{tikzcd}
      \oM_{[i]_n}(B \times_A \Pi_BC) \ar[r, "\mu_{C,i}"] \ar[d, "{\oM_{[i]_n}(\ev)}"']
      &
      \M_i(\gamma^*B \times_{\gamma^*A} \Pi_{\gamma^*A}{\gamma^*C})
      \ar[d, "{\M_i(\ev)}"]
      \\
      \oM_{[i]_n}C
      \ar[r, "\kappa_{C,i}"']
      &
      \M_iC
    \end{tikzcd}
  \end{equation*}
\end{lemma}
\begin{proof}
  Because $Bj \times_{Aj} (\Pi_{\gamma^*B}{\gamma^*C})j$ is a pullback, to check
  the first diagram
  $\pi_{u^-} \cdot \mu_{C,i} = (Bj \times_{Aj} \varphi_{C,j}) \cdot
  {\overline\pi}_{[u^-]_n}$ is to check that their postcomposition with the
  limiting legs of the pullback commute.
  This is observed by chasing the below diagram, where the first diagram in the
  statement is highlighted in red.
  \begin{equation*}
    \begin{tikzcd}[cramped, column sep=-0.15em]
      & {\oM_{[i]_n}(\Pi_BC)} && {M_i(\Pi_{\gamma^*B}{\gamma^*C})} \\
      \\
      {\oM_{[i]_n}(B \times_A \Pi_BC)} && {\M_i(\gamma^*B \times_{\gamma^*A} \Pi_{\gamma^*B}{\gamma^*C})} \\
      & {\oM_{[i]_n}A} && {\M_i(\gamma^*A)} & {(\Pi_BC)j} && {(\Pi_{\gamma^*B}{\gamma^*C})j} \\
      \\
      {\oM_{[i]_n}B} && {\M_i(\gamma^*B)} & Aj \\
      &&& {Bj \times_{Aj} (\Pi_BC)j} && {Bj \times_{Aj} (\Pi_{\gamma^*B}{\gamma^*C})j} \\
      && Bj
      \arrow["{\kappa_{B,i}}"{description, pos=0.75}, from=6-1, to=6-3]
      \arrow["{\oM_{[i]_n}f}"{description, pos=0.7}, from=6-1, to=4-2]
      \arrow["{\kappa_{A,i}}"{description, pos=0.25}, from=4-2, to=4-4]
      \arrow["{\tphi_{C,i}}"{description}, from=1-2, to=1-4]
      \arrow["{\oM_{[i]_n}(\Pi_Bg)}"{description, pos=0.3}, from=1-2, to=4-2]
      \arrow[from=3-1, to=6-1]
      \arrow[from=3-1, to=1-2]
      \arrow["{\pi_{u^-}}"{description, pos=0.6}, from=6-3, to=8-3]
      \arrow[from=7-6, to=8-3]
      \arrow[from=7-4, to=8-3]
      \arrow["{{\overline\pi}_{[u^-]_n}}"{description}, from=6-1, to=8-3]
      \arrow["{\pi_{u^-}}"{description, pos=0.25}, from=4-4, to=6-4]
      \arrow["{{\overline\pi}_{[u^-]_n}}"{description, pos=0.3}, from=4-2, to=6-4]
      \arrow["{f_j}"{description, pos=0.75}, from=6-4, to=8-3]
      \arrow["{(\Pi_{\gamma^*B}{\gamma^*g})j}"{description}, from=4-7, to=6-4]
      \arrow["{\pi_{u^-}}"{description}, from=1-4, to=4-7]
      \arrow[from=7-6, to=4-7]
      \arrow["{{\overline\pi}_{u^-}}"{description, pos=0.3}, from=1-2, to=4-5]
      \arrow[from=4-5, to=6-4]
      \arrow["{\varphi_{C,j}}"{description}, from=4-5, to=4-7]
      \arrow["{\M_i(\gamma^*f)}"{description}, crossing over, from=6-3, to=4-4]
      \arrow["{\M_i(\Pi_{\gamma^*B}{\gamma^*g})}"{description}, crossing over, from=1-4, to=4-4]
      \arrow[crossing over, from=3-3, to=6-3]
      \arrow[crossing over, from=3-3, to=1-4]
      \arrow[crossing over, from=7-4, to=4-5]
      \arrow["{\mu_{C,i}}"{fill=white0}, color=red0, crossing over, from=3-1, to=3-3]
      \arrow["{\pi_{u-}}"{description, pos=0.8}, color=red0, crossing over, bend right=10, from=3-3, to=7-6]
      \arrow["{{\overline\pi}_{[u^-]_n}}"{description}, out=-90, in=180, color=red0, crossing over, from=3-1, to=7-4]
      \arrow["{Bj \times_{Aj} \varphi_{C,j}}", color=red0, crossing over, from=7-4, to=7-6]
    \end{tikzcd}
  \end{equation*}

  To observe the second diagram, note that for each $u^- \colon i \to j \in i/\McI_{<n}$,
  \begin{equation*}
    \begin{tikzcd}
      & {Bi \times_{Aj} (\Pi_BC)j} \\
      && {Bi \times_{Aj} (\Pi_{\gamma^*B}{\gamma^*C})j} \\
      {\oM_{[i]_n}(B \times_A \Pi_BC)} & {\oM_{[i]_n}C} & C[j]_n \\
      & {\M_i(\gamma^*B \times_{\gamma^*A} \Pi_{\gamma^*B}{\gamma^*C})} & {\M_iC} & Cj
      \arrow["{\kappa_{C,i}}"{description}, from=3-2, to=4-3, color=red0]
      \arrow["{\mu_{C,i}}"{description}, from=3-1, to=4-2, color=red0]
      \arrow["{Bj \times_{Aj} \varphi_{C,j}}"{description}, from=1-2, to=2-3]
      \arrow["{\oM_{[i]_n}(\ev)}"{description}, from=3-1, to=3-2, color=red0]
      \arrow["{\M_i(\ev)}"{description}, from=4-2, to=4-3, color=red0]
      \arrow["{\pi_{u^-}}"{description}, from=4-3, to=4-4]
      \arrow["{{\overline\pi}_{[u^-]_n}}"{description, pos=0.2}, from=3-2, to=3-3]
      \arrow["{=}"{description}, from=3-3, to=4-4]
      \arrow["\ev"{description}, from=1-2, to=3-3]
      \arrow["\ev"{description}, from=2-3, to=4-4]
      \arrow["{\pi_{u^-}}"{description, pos=0.7}, crossing over, from=4-2, to=2-3]
      \arrow["{{\overline\pi}_{[u^-]_n}}"{description}, from=3-1, to=1-2]
    \end{tikzcd}
  \end{equation*}
  where the left slanted face is by the first diagram, the right slanted face is
  by the construction of the comparison map $\varphi$ between dependent products
  as from \Cref{eqn:phi-def} in \Cref{constr:pi-comp}.
  Because $u^-$ is any object from $i/\McI_{<n}$, it follows from the universal
  property of the matching object that the second diagram in the statement
  (i.e. the bottom left square in the diagram above) commutes.
\end{proof}

\begin{assumption}\label{asm:bd-initial}
  Assume that for each $i \in \partial\McI_n$, the localisation restricts to a
  map $\gamma|_i \colon i/\McI_{<n} \to (i/\Ho\McI)^\circ$ that is initial
  (i.e. for each $\id \neq f \colon [i]_n \to j \in \Ho\McI$, the comma category
  $\gamma|_i \downarrow f$ is non-empty and connected).
\end{assumption}
The main reason for the initiality condition in the above \Cref{asm:bd-initial}
is so that we have that matching objects of $\McI_\infty$-shaped and
$(\Ho\McI_\infty)$-shaped diagrams are isomorphic.
This is made precise in the following sense.

\begin{lemma}\label{lem:mat-comp-iso}
  Under \Cref{asm:bd-initial}, the comparison map
  $\kappa_{X,i} \colon \oM_{[i]_n}X \to \M_i(\gamma^*X)$ of
  \Cref{constr:mat-comp} is an isomorphism.
\end{lemma}
\begin{proof}
  Straightforward, by \Cref{eqn:mat-comp-fact} in \Cref{constr:mat-comp}.
\end{proof}

The above isomorphism property of matching objects of homotopical diagrams is
crucial for the construction of the following natural transformation between
fibred section functors.
Roughly, this is because the fibred section functor is contravariant in the
domain and covariant in the codomain.
Thus, in general, maps between fibred sections are dinatural.
The isomorphism property above then allows one to construct a natural
transformation.
This is reminiscent of the situation in \Cref{subsec:gpd-Pi}.
In particular, the following construction is similar in spirit to
\Cref{constr:gpd-Pi}.

\begin{construction}\label{constr:pi-bd-iso}
  For $g \colon C \to B \in \McE^{\Ho\McI_\infty}/B$ and $i \in \partial\McI_n$
  for $n \in \McN$, define a map
  \begin{align*}
    \Pi_{\oM_{[i]_n} B}{\oM_{[i]_n}C} \xrightarrow{\rho_{C,i}}
    \Pi_{\M_i(\gamma^*B)}\M_i(\gamma^*C)
  \end{align*}
  as follows.

  By naturality of $\kappa$ in \Cref{lem:mat-comp} and the fact that each
  $\kappa_{X,i} \colon \oM_iX \to \M_i(\gamma^*X)$ is an isomorphism as by
  \Cref{lem:mat-comp-iso}, the bottom face below is a pullback.
  Thus, $\rho_{C,i}$ is the unique map whose transpose under
  $\M_i(\gamma^*f)^* \dashv \Pi_{\M_i(\gamma^*B)}$ is given by
  \begin{equation*}
    \begin{tikzcd}
      \M_i(\gamma^*f)^*(\Pi_{\oM_{[i]_n} B}{\oM_{[i]_n}C}) \cong
      (\oM_{[i]_n}f)^*(\Pi_{\oM_{[i]_n} B}{\oM_{[i]_n}C})
      \ar[r, "\ev"]
      &
      \oM_{[i]_n}C
      \ar[r, "\kappa_{C,i}", "\cong"']
      &
      \M_i(\gamma^*C)
    \end{tikzcd}
  \end{equation*}
  so that we have
  %
  \begin{equation*}
    \begin{tikzcd}[column sep=-0.1em]
      &&& {\Pi_{\oM_{[i]_n} B}{\oM_{[i]_n}C}} &&& {\Pi_{\M_i(\gamma^*B)}\M_i(\gamma^*C)} \\
      \\
      & {\M_i(\gamma^*f)^*(\Pi_{\oM_i B}{\oM_iC})} &&& {(\M_i(\gamma^*f))^*(\Pi_{\M_i(\gamma^*B)}\M_i(\gamma^*C))} \\
      {\oM_{[i]_n}C} &&&&& {\M_i(\gamma^*C)} \\
      &&& {\oM_{[i]_n}A} &&& {\M_{[i]_n}(\gamma^*A)} \\
      \\
      & {\oM_{[i]_n}B} &&& {\M_i(\gamma^*B)}
      \arrow["{\rho_{C,i}}", from=1-4, to=1-7]
      \arrow["{\Pi_{\M_i(\gamma^*B)}{\M_i(\gamma^*g)}}"{description}, from=1-7, to=5-7]
      \arrow["{\M_i(\gamma^*f)}"{description}, from=7-5, to=5-7]
      \arrow[from=4-1, to=7-2]
      \arrow["{\kappa_{B,i}}"', "\cong", from=7-2, to=7-5]
      \arrow["{\oM_{[i]_n}f}"{description}, from=7-2, to=5-4]
      \arrow["{\Pi_{\oM_{[i]_n}B}{\oM_{[i]_n}g}}"{description, pos=0.6}, from=1-4, to=5-4]
      \arrow["{\kappa_{A,i}}"'{pos=0.2}, "\cong"{pos=0.2}, from=5-4, to=5-7]
      \arrow["\lrcorner"{pos=0, rotate=45}, phantom, from=7-2, to=5-7]
      \arrow[from=3-5, to=1-7]
      \arrow["\lrcorner"{pos=0, rotate=30}, phantom, from=3-5, to=5-7]
      \arrow["\ev"{description}, from=3-5, to=4-6]
      \arrow[from=3-2, to=1-4]
      \arrow[from=3-2, to=7-2]
      \arrow["\lrcorner"{pos=0, rotate=30}, phantom, from=3-2, to=5-4]
      \arrow["\ev"{description}, from=3-2, to=4-1]
      \arrow[crossing over, from=4-6, to=7-5]
      \arrow[crossing over, from=3-5, to=7-5]
      \arrow["{\M_i(\gamma^*f)^*{\rho_{C,i}}}"{description, pos=0.4}, crossing over, from=3-2, to=3-5]
      \arrow["{\kappa_{C,i}}"'{pos=0.3}, "\cong"{pos=0.3}, crossing over, from=4-1, to=4-6]
    \end{tikzcd}
  \end{equation*}
  And because $\kappa_{C,i}$ is an isomorphism, it is easy to see that
  $\rho_{C,i}$ has an inverse
  $\rho_{C,i}^{-1} \colon
  \Pi_{\M_i(\gamma^*B)}\M_i(\gamma^*C)
  \to
  \Pi_{\oM_{[i]_n}B}{\oM_{[i]_n}C}$
  whose transpose under $(\oM_{[i]_n}f)^* \dashv \Pi_{\oM_{[i]_n}B}$ is given by
  \begin{equation*}
    \begin{tikzcd}
      (\oM_{[i]_n}f)^*(\Pi_{\M_i(\gamma^*B)}{\M_i(\gamma^*C)}) \cong
      (\M_i(\gamma^*f))^*(\Pi_{\M_i(\gamma^*B)}{\M_i(\gamma^*C)})
      \ar[r, "\ev"]
      &
      \M_i(\gamma^*C)
      \ar[r, "\kappa_{C,i}^{-1}", "\cong"']
      &
      \oM_{[i]_n}C
    \end{tikzcd}
  \end{equation*}

  Furthermore, by the above construction of $\rho_{C,i}$ over $\kappa_{A,i}$ and
  the fact that $\kappa_{A,i}$ is over $Ai$ as from \Cref{constr:mat-comp}, one
  has a map of cospans
  \begin{equation*}
    \begin{tikzcd}[column sep=huge]
      \Pi_{\oM_{[i]_n}B}\oM_{[i]_n}C \ar[r, "\Pi_{\oM_{[i]_n}B}\oM_{[i]_n}g"] \ar[d, "\rho_{C,i}"', "\cong"]
      &
      \oM_{[i]_n}A \ar[d, "\kappa_{A,i}", "\cong"']
      &
      A[i]_n \ar[d, "="] \ar[l, "\om_iA"']
      \\
      \Pi_{\M_i(\gamma^*B)}\M_i(\gamma^*C)
      \ar[r, "\Pi_{\M_i(\gamma^*B)}\M_i(\gamma^*g)"']
      &
      \M_i(\gamma^*A)
      &
      Ai \ar[l, "\m_i(\gamma^*A)"]
    \end{tikzcd}
  \end{equation*}
  which induces a map
  $\trho_{C,i} \coloneqq (\id\relax, \rho_{C,i}) \colon A[i]_n
  \times_{\oM_{[i]_n}A} \Pi_{\oM_{[i]_n}B}{\oM_{[i]_n}C} \cong Ai
  \times_{\M_i(\gamma^*A)} \Pi_{\M_i(\gamma^*B)}{\M_i(\gamma^*C)}$.
\end{construction}

\begin{lemma}\label{lem:tphi-ev-rho}
  For $\tphi_{C,i}$ the map from \Cref{constr:pi-comp} and $\rho_{C,i}$ from
  \Cref{constr:pi-bd-iso},
  \begin{equation*}
    \begin{tikzcd}
      {\oM_{[i]_n}(\Pi_BC)} && {\Pi_{\oM_{[i]_n}B}{\oM_{[i]_n}C}} \\
      & {\oM_{[i]_n}A} \\
      {\M_i(\Pi_{\gamma^*B}{\gamma^*C})} && {\Pi_{\M_i(\gamma^*B)}{\M_i(\gamma^*C)}} \\
      & {\M_i(\gamma^*A)}
      \arrow["{\Pi_{\M_i(\gamma^*B)}{\M_i(\gamma^*g)}}"{description}, from=3-3, to=4-2]
      \arrow["{\M_i(\Pi_{\gamma^*B}{\gamma^*g})}"{description}, from=3-1, to=4-2]
      \arrow["{\tphi_{C,i}}"{description}, from=1-1, to=3-1]
      \arrow["{\oM_{[i]_n}(\ev)^\ddagger}", from=1-1, to=1-3]
      \arrow["{\rho_{C,i}}"{description}, from=1-3, to=3-3]
      \arrow["{\oM_{[i]_n}(\Pi_Bg)}"{description}, from=1-1, to=2-2]
      \arrow["{\Pi_{\oM_{[i]_n}B}{\oM_{[i]_n}g}}"{description}, from=1-3, to=2-2]
      \arrow["{\M_i(\ev)^\ddagger}"{description, pos=0.3}, from=3-1, to=3-3]
      \arrow["{\kappa_{A,i}}"{description, pos=0.3}, crossing over, from=2-2, to=4-2]
    \end{tikzcd}
  \end{equation*}
\end{lemma}
\begin{proof}
  The left and right slanted faces are respectively by \Cref{lem:pi-comp} and
  the construction of $\rho$ from \Cref{constr:pi-bd-iso}.
  The top and bottom triangles are from the construction of
  $\oM_{[i]_n}(\ev)^\ddagger$ and $\M_i(\ev)^\ddagger$ by adapting the
  construction from \Cref{constr:glue-Pi}.
  It suffices to verify the back face.

  To do so, it suffices to show that the transposes of
  $\M_i(\ev)^\ddagger \cdot \tphi_{C,i}$ and
  $\rho_{C,i} \cdot \oM_{[i]_n}(\ev)^\ddagger$ under the adjunction
  $\M_i(\gamma^*f)^* \dashv \Pi_{\M_i(\gamma^*B)}$ agree.
  By the fact that $\kappa$ is a natural isomorphism as in
  \Cref{lem:mat-comp-iso,lem:mat-comp}, the bottom face of the gigantic cube in
  \Cref{constr:pi-bd-iso} is a pullback.
  %
  Thus, the pullback of
  $\tphi_{C,i} \colon
  \oM_{[i]_n}(\Pi_BC)
  \to
  \M_i(\Pi_{\gamma^*B}{\gamma^*C}) \in
  \sfrac{\McE}{\M_i(\gamma^*A)}$ along
  $\M_i(\gamma^*f) \colon \M_i(\gamma^*B) \to \M_i(\gamma^*A)$
  is the map
  $\oM_{[i]_n}(B \times_A \Pi_BC) \to \M_i(\gamma^*B \times_{\gamma^*A}
  \Pi_{\gamma^*B}{\gamma^*C})$ induced by the map of cospans
  \begin{equation*}
    \begin{tikzcd}[column sep=huge]
      \oM_{[i]_n}B \ar[r, "\oM_{[i]_n}f"] \ar[d, "\kappa_{B,i}"']
      &
      \oM_{[i]_n}A \ar[d, "\kappa_{A,i}"{description}]
      &
      \oM_{[i]_n}(\Pi_BC) \ar[l, "\oM_{[i]_n}(\Pi_Bg)"'] \ar[d, "\tphi_{C,i}"]
      \\
      \M_i(\gamma^*B) \ar[r, "\M_i(\gamma^*f)"']
      &
      \M_i(\gamma^*A)
      &
      \M_i(\Pi_{\gamma^*B}{\gamma^*C}) \ar[l, "\M_i(\Pi_{\gamma^*B}{\gamma^*g})"]
    \end{tikzcd}
  \end{equation*}
  This is exactly $\mu_{C,i} = (\kappa_{B,i}, \tphi_{C,i})$ from
  \Cref{constr:tphi-pb-proj}.
  Also by the same reason, the pullback of $\oM_{[i]_n}(\ev)^\ddagger$ under
  $\M_i(\gamma^*f)$ is its pullback under $\oM_{[i]_n}f$.
  Hence, using the fact that $\rho_{C,j} = (\kappa_{C,i} \cdot \ev)^\ddagger$ as
  from \Cref{constr:pi-bd-iso}, the transposes of
  $\M_i(\ev)^\ddagger \cdot \tphi_{C,i}$ and
  $\rho_{C,i} \cdot \oM_{[i]_n}(\ev)^\ddagger$ are respectively
  $\M_i(\ev) \cdot \mu_{C,i}$ and $\kappa_{C,i} \cdot \oM_{[i]_n}(\ev)$.
  By the second diagram of \Cref{lem:tphi-pb-proj},
  $\kappa_{C,i} \cdot \oM_{[i]_n}(\ev) = \M_i(\ev) \cdot \mu_{C,i}$.
  And so the result follows.
\end{proof}

\begin{lemma}\label{lem:bot-face-left}
  For $n \in \McN$ and $i \in \partial\McI_n$, over $Ai = A[i]_n$,
  %
  \begin{equation*}
    \begin{tikzcd}[column sep=-1em]
      {A[i]_n \times_{\oM_{[i]_n}A} \oM_{[i]_n}(\Pi_BC)} && {A[i]_n \times_{\oM_{[i]_n}A} \Pi_{\oM_{[i]_n}B}{\oM_{[i]_n}C}} \\
      & {Ai \times_{\M_i(\gamma^*A)} \M_i(\Pi_{\gamma^*B}{\gamma^*C})} && {Ai \times_{\M_i(\gamma^*A)} \Pi_{\M_i(\gamma^*B)}{\M_i(\gamma^*C)}} \\
      \\
      & Ai = A[i]_n
      \arrow["{\trho_{C,i}}"{description}, from=1-3, to=2-4]
      \arrow["{\psi_{C,i}}"{description}, from=1-1, to=2-2]
      \arrow["{Ai \times_{\oM_{[i]_n}A} \oM_{[i]_n}(\ev)^\ddagger}"{description}, from=1-1, to=1-3]
      \arrow[from=1-3, to=4-2]
      \arrow[from=1-1, to=4-2]
      \arrow[from=2-2, to=4-2]
      \arrow[from=2-4, to=4-2]
      \arrow["{Ai \times_{\M_i(\gamma^*A)} \M_i(\ev)^\ddagger}"{description}, crossing over, from=2-2, to=2-4]
    \end{tikzcd}
  \end{equation*}
\end{lemma}
\begin{proof}
  Because $\kappa_{A,i} \colon \oM_iA \to \M_i(\gamma^*A)$ is an isomorphism by
  \Cref{lem:mat-comp-iso}, and $\kappa_{A,i}$ is under $Ai$ by \Cref{lem:mat-comp}, we have a pullback square
  \begin{equation*}
    \begin{tikzcd}
      A[i]_n \ar[r, "\om_{[i]_n}A"] \ar[d, "="{description}] \ar[rd, phantom, "\lrcorner"{pos=0}]
      &
      \oM_{[i]_n}A \ar[d, "\kappa_{A,i}", "\cong"']
      \\
      Ai \ar[r, "\m_iA"']
      &
      \M_iA
    \end{tikzcd}
  \end{equation*}
  Thus, pulling back
  $\oM_{[i]_n}(\ev)^\ddagger \colon \oM_{[i]_n}(\Pi_BC) \to \Pi_{\oM_{[i]_n}B}{\oM_{[i]_n}C} \in
  \McE/\M_i(\gamma^*A)$ along $\m_iA \colon Ai \to \M_i(\gamma^*A)$ is the same
  as it back along $\om_{[i]_n}A \colon A[i]_n \to \oM_{[i]_n}A$.
  Furthermore, by \Cref{constr:psiCi,constr:pi-bd-iso}, it is clear that the
  pullbacks of $\tphi_{C,i}$ and $\rho_{C,i}$ are respectively $\psi_{C,i}$ and
  $\trho_{C,i}$.
  Therefore, the result follows by pulling back the diagram in
  \Cref{lem:tphi-ev-rho} along $\m_iA$.
\end{proof}

\begin{construction}\label{constr:sigmaCi}
  For each $n \in \McN$ and $i \in \partial\McI_n$, define a map
  \begin{align*}
    \sigma_{C,i}
    \colon
    \Pi_{B[i]_n}(B[i]_n \times_{\oM_{[i]_n}B} \oM_{[i]_n}C)
    \to
    \Pi_{Bi}(Bi \times_{\M_i(\gamma^*B)} \M_i(\gamma^*C))
  \end{align*}
  over $Ai = A[i]_n$ such that its transpose under $f_i^* \dashv \Pi_{Bi}$ is
  given by $(Bi \times_{\M_i(\gamma^*B)} \kappa_{C,i}) \cdot \ev$:
  %
  \begin{equation*}
    \begin{tikzcd}[column sep=-2em]
      {B[i]_n \times_{A[i]_n} \Pi_{B[i]_n}(B[i]_n \times_{\oM_{[i]_n}B} \oM_{[i]_n}C)}
      && {Bi \times_{Ai} \Pi_{Bi}(Bi \times_{\M_i(\gamma^*B)} \M_i(\gamma^*C))} \\
      \\
      & {B[i]_n \times_{\oM_{[i]_n}B} \oM_{[i]_n}C} && {Bi \times_{\M_i(\gamma^*B)} \M_i(\gamma^*C)} \\
      \\
      & Bi = B[i]_n
      \arrow["{Bi \times_{Ai} \sigma_{C,i}}"{description}, from=1-1, to=1-3]
      \arrow["\ev"{description}, from=1-3, to=3-4]
      \arrow[from=1-1, to=5-2]
      \arrow[from=1-3, to=5-2]
      \arrow[from=3-4, to=5-2]
      \arrow["\ev"{description}, from=1-1, to=3-2]
      \arrow[from=3-2, to=5-2]
      \arrow["{(\id,\kappa_{C,i})}"{description}, crossing over, from=3-2, to=3-4]
    \end{tikzcd}
  \end{equation*}
  where $(\id,\kappa_{C,i})$ is the map between pullbacks induced by the map of
  cospans
  \begin{equation*}
    \begin{tikzcd}[column sep=huge]
      \oM_iC \ar[r, "\oM_ig"] \ar[d, "\kappa_{C,i}", "\cong"']
      \ar[rd, "\lrcorner"{pos=0}, phantom]
      &
      \oM_iB \ar[d, "\kappa_{A,i}", "\cong"']
      &
      Bi \ar[l, "\om_iB"'] \ar[d, "="']
      \\
      \M_i(\gamma^*C) \ar[r, "\M_i(\gamma^*g)"']
      &
      \M_i(\gamma^*B)
      &
      Bi \ar[l, "\m_i(\gamma^*B)"]
    \end{tikzcd}
  \end{equation*}
  by naturality of $\kappa$ and the fact that $\kappa_{B,i}$ is under $Bi$ from
  \Cref{lem:mat-comp}.
  Once again, because the left square is a pullback, it is also the pullback of
  $\kappa_{C,i} \colon \oM_{[i]_n}C \to \M_iC$ under
  $\m_i(\gamma^*B) \colon Bi \to \M_i(\gamma^*B)$.

  Clearly, $(\id, \kappa_{C,i})$ is an isomorphism because
  $\kappa_{C,i}$ is an isomorphism.
  Because $\sigma_{C,i}$ is the image of $(\id, \kappa_{C,i})$ under the
  functorial action of $\Pi_{Bi}$, it is easy to observe that it is an
  isomorphism.
\end{construction}

\begin{lemma}\label{lem:bot-face-right}
  For $\trho_{C,i}$ from \Cref{constr:pi-bd-iso} and $\sigma_{C,i}$ from
  \Cref{constr:sigmaCi}, we have that over $Ai$,
  %
  \begin{equation*}
    \begin{tikzcd}[column sep=-1em]
      {A[i]_n \times_{\oM_{[i]_n}A} \Pi_{\oM_{[i]_n}B}{\oM_{[i]_n}C}}
      && {\Pi_{B[i]_n}(B[i]_n \times_{\oM_{[i]_n}B} \oM_{[i]_n}C)} \\
      & {Ai \times_{\M_i(\gamma^*A)} \Pi_{\M_i(\gamma^*B)}{\M_i(\gamma^*C)}} && {\Pi_{Bi}(Bi \times_{\M_i(\gamma^*B)}{\M_i(\gamma^*C)})} \\
      \\
      & Ai = A[i]_n
      \arrow["{\sigma_{C,i}}"{description}, from=1-3, to=2-4]
      \arrow["{\trho_{C,i}}"{description}, from=1-1, to=2-2]
      \arrow["{(\id\relax, \ev\relax)^\ddagger}"{description}, from=1-1, to=1-3]
      \arrow[from=1-3, to=4-2]
      \arrow[from=2-4, to=4-2]
      \arrow[from=2-2, to=4-2]
      \arrow[from=1-1, to=4-2]
      \arrow["{(\id\relax, \ev\relax)^\ddagger}"{description}, crossing over, from=2-2, to=2-4]
    \end{tikzcd}
  \end{equation*}
\end{lemma}
\begin{proof}
  First, note that because $\kappa$ is a natural isomorphism by
  \Cref{lem:mat-comp,lem:mat-comp-iso}, the left face of the following prism as
  a pullback.
  And by \Cref{constr:pi-bd-iso}, the image of $\rho_{C,i}$ by pulling back along
  $\om_i(\gamma^*A)$ is exactly $\trho_{C,i}$.
  %
  \[\begin{tikzcd}
      & {\oM_{[i]_n}B} \\
      \\
      {\Pi_{\oM_{[i]_n}B}{\oM_{[i]_n}C}} && {\M_i(\gamma^*B)} &&& Bi \\
      & {\oM_{[i]_n}A} \\
      & {\Pi_{\M_i(\gamma^*B)}{\M_i(\gamma^*C)}} \\
      && {\M_i(\gamma^*A)} &&& Ai
      \arrow["{\m_i(\gamma^*A)}", from=6-6, to=6-3]
      \arrow["{\m_i(\gamma^*B)}"{description}, from=3-6, to=3-3]
      \arrow["{f_i}"{description}, from=3-6, to=6-6]
      \arrow["{\kappa_{A,i}}"{description}, from=4-2, to=6-3]
      \arrow["{\om_{[i]_n}A}"{description, pos=0.4}, from=6-6, to=4-2]
      \arrow["{\kappa_{B,i}}"{description}, from=1-2, to=3-3]
      \arrow["{\oM_{[i]_n}f}"{description}, from=1-2, to=4-2]
      \arrow["{\om_{[i]_n}B}"{description}, from=3-6, to=1-2]
      \arrow["{\rho_{C,i}}"{description}, from=3-1, to=5-2]
      \arrow[from=3-1, to=4-2]
      \arrow[from=5-2, to=6-3]
      \arrow["\lrcorner"{pos=0, rotate=-45}, phantom, from=1-2, to=6-3]
      \arrow["{\M_i(\gamma^*f)}"{description, pos=0.4}, crossing over, from=3-3, to=6-3]
    \end{tikzcd}\]
  So further pulling $\trho_{C,i}$ back along $f_i$ is the same as pulling back
  $\rho_{C,i}$ along the composite
  $Bi \xrightarrow{\m_i(\gamma^*B)} \M_i(\gamma^*B) \xrightarrow{\M_i(\gamma^*f)} \M_i(\gamma^*A)$.
  But by \Cref{constr:pi-bd-iso}, pulling $\rho_{C,i}$ back along
  $\M_i(\gamma^*f)$ gives a map
  \begin{align*}
    \M_i(\gamma^*f)^*\rho_{C,i} \colon
    \oM_iB \times_{\oM_iA} \Pi_{\oM_iB}{\oM_iC}
    \to
    \M_i(\gamma^*B) \times_{\M_i(\gamma^*A)} \Pi_{\M_i(\gamma^*B)}{\M_i(\gamma^*C)}
  \end{align*}
  over $\M_i(\gamma^*B)$, which when
  composed with
  $\ev \colon \M_i(\gamma^*B) \times_{\M_i(\gamma^*A)}
  \Pi_{\M_i(\gamma^*B)}{\M_i(\gamma^*C)} \to \M_i(\gamma^*C)$ is the same as
  $\ev \cdot \kappa_{C,i}$.
  %
  \begin{equation}\label{eqn:rho-trans}\tag{\textsc{$\rho$-trans}}
    \begin{tikzcd}[column sep=0em]
      {\oM_{[i]_n}C} && {\oM_{[i]_n}B \times_{\oM_{[i]_n}A} \Pi_{\oM_{[i]_n}B}{\oM_{[i]_n}C} } \\
      & {\M_i(\gamma^*C)} && {\M_i(\gamma^*B) \times_{\M_i(\gamma^*A)} \Pi_{\M_i(\gamma^*B)}{\M_i(\gamma^*C)}} \\
      && {\oM_{[i]_n}(\gamma^*B)} \\
      &&& {\M_i(\gamma^*B)}
      \arrow["{\kappa_{B,i}}"{description}, from=3-3, to=4-4]
      \arrow[from=1-3, to=3-3]
      \arrow[from=2-4, to=4-4]
      \arrow["{\M_i(\gamma^*f)^*\rho_{C,i}}"{description}, from=1-3, to=2-4]
      \arrow["\ev"{description}, from=1-3, to=1-1]
      \arrow["{\kappa_{C,i}}"{description}, from=1-1, to=2-2]
      \arrow[bend right=30, from=1-1, to=3-3]
      \arrow["\ev"{description, pos=0.4}, from=2-4, to=2-2, crossing over]
      \arrow[bend right=30, from=2-2, to=4-4, crossing over]
    \end{tikzcd}
  \end{equation}

  Now, by \Cref{constr:sigmaCi}, taking the pullback of the top face of the
  diagram in the statement of this lemma along $f_i \colon Bi \to Ai$ yields the
  upper slanted face of the following diagram (over $Bi$).
  Further composing the slanted face by the counit
  $Bi \times_{Ai} \Pi_{Bi} (Bi \times_{M_i(\gamma^*B)} \M_i(\gamma^*C)) \to Bi
  \times_{M_i(\gamma^*B)} \M_i(\gamma^*C)$, we obtain the transpose
  of the top face of the diagram in the statement of this lemma under
  $f_i^* \dashv \Pi_{Bi}$ (again over $Bi$).
  Hence, it suffices to check that the middle layer below commutes.
  %
  \begin{equation*}
    \begin{tikzcd}[column sep=0em]
      {B[i]_n \times_{A[i]_n} \Pi_{B[i]_n}(\bullet)} \\
      \\
      {B[i]_n \times_{\oM_{[i]_n}B} \oM_{[i]_n}C}
      &&
      {B[i]_n \times_{\oM_{[i]_n}A} \Pi_{\oM_{[i]_n}B}{\oM_{[i]_n}C}} \\
      & {Bi \times_{Ai} \Pi_{Bi}(\bullet)} \\
      B[i]_n \\
      & {Bi \times_{\M_i(\gamma^*B)} \M_i(\gamma^*C)} && {Bi \times_{\M_i(\gamma^*A)} \Pi_{\M_i(\gamma^*B)}{\M_i(\gamma^*C)}} \\
      \\
      & Bi
      \arrow["\ev"{description}, from=1-1, to=3-1]
      \arrow["{(\id, \M_i(\gamma^*f)^*\rho_{C,i}) = Bi \times_{Ai} \trho_{C,i}}", from=3-3, to=6-4]
      \arrow["{B[i]_n \times_{A[i]_n} (A[i]_n \times_{\oM_{[i]_n}A} \ev\relax)^\ddagger}"', from=3-3, to=1-1]
      \arrow["{(\id,\ev\relax)}"{description, pos=0.5}, from=3-3, to=3-1]
      \arrow["{(\id, \ev\relax)}"{description}, from=6-4, to=6-2]
      \arrow[from=3-1, to=5-1]
      \arrow[from=6-2, to=8-2]
      \arrow[from=5-1, to=8-2, "="{description}]
      \arrow[bend left=30, from=6-4, to=8-2]
      \arrow[bend left=30, from=3-3, to=5-1]
      \arrow["{(\id, \kappa_{C,i})}"{description}, from=3-1, to=6-2, crossing over]
      \arrow["{Bi \times_{Ai} \sigma_{C,i}}"{description, pos=0.3}, from=1-1, to=4-2, crossing over]
      \arrow["\ev"{description, pos=0.25}, from=4-2, to=6-2, crossing over]
      \arrow["{B_i \times_{Ai} (Ai \times_{\M_i(\gamma^*A)} \ev)^\ddagger}"{description}, from=6-4, to=4-2, crossing over]
    \end{tikzcd}
  \end{equation*}
  But the middle layer above is exactly the pullback of the top face of
  \Cref{eqn:rho-trans} under $\om_iB$.
  The result now follows.
\end{proof}

\begin{lemma}\label{lem:right-face}
  The map $\sigma_{C,i}$ constructed in \Cref{constr:sigmaCi} is under
  $\Pi_{Bi}{Ci}$.
  \begin{equation*}
    \begin{tikzcd}
      & \Pi_{B[i]_n}{C[i]_n} = \Pi_{Bi}{Ci}
      \ar[rd, "{\Pi_{Bi}(g_i, \m_i(\gamma^*C))}"]
      \ar[ld, "{\Pi_{Bi}(g_i, \om_iC)}"']
      \\
      \Pi_{Bi}(Bi \times_{\oM_iB} \oM_iC) \ar[rr, "\sigma_{C,i}"']
      &
      &
      \Pi_{Bi}(Bi \times_{\M_i(\gamma^*B)} \M_i(\gamma^*C))
    \end{tikzcd}
  \end{equation*}
\end{lemma}
\begin{proof}
  As observed in \Cref{constr:sigmaCi}, $\sigma_{C,i}$ is the functorial action
  of $\Pi_{Bi}$ on
  $(\id, \kappa_{C,i})
  \colon
  B[i]_n \times_{\oM_{[i]_n}B} \oM_{[i]_n}C
  \to
  Bi \times_{\M_i(\gamma^*B)} \M_i(\gamma^*C)$ induced by the map of cospans
  \begin{equation*}
    \begin{tikzcd}
      B[i]_n \ar[r, "{\om_{[i]_n}B}"] \ar[d, "="']
      & \oM_{[i]_n}B \ar[d, "\kappa_{B,i}"{description}]
      & \ar[l, "\oM_{[i]_n}g"'] \oM_{[i]_n}C \ar[d, "\kappa_{C,i}"]
      \\
      Bi \ar[r, "\m_i(\gamma^*B)"'] & \M_i(\gamma^*B) & \ar[l, "\om_i(\gamma^*g)"] \M_i(\gamma^*C)
    \end{tikzcd}
  \end{equation*}
  Therefore, it suffices to check that under $Ci$, one has
  $(\id,\kappa_{C,i}) \cdot (g_i, \om_iC) = (g_i, m_i(\gamma^*C))$.
  But this is obvious because composing with the limiting leg
  $Bi \times_{\M_i(\gamma^*B)} \M_i(\gamma^*C) \to Bi$, both maps
  $(\id,\kappa_{C,i}) \cdot (g_i, \om_iC)$ and $(g_i, m_i(\gamma^*C))$
  give rise to $g_i$ (as the map of cospans inducing $\sigma'$ is identity on $Bi$).
  And composing both maps with the limiting leg
  $Bi \times_{\M_i(\gamma^*B)} \M_i(\gamma^*C) \to \M_i(\gamma^*C)$, one obtains
  $\kappa_{C,i} \cdot \om_{[i]}C = \m_i(\gamma^*C)$ because of \Cref{lem:mat-comp}.
\end{proof}

\begin{theorem}\label{thm:pi-comp-iso}
  Under \Cref{asm:E-enough-limits,asm:Ho-I-epi-E,asm:ho-powerful,asm:bd-initial}, the canonical
  comparison map between dependent products from \Cref{constr:pi-comp} is a
  natural isomorphism so
  $\gamma^* \colon \McE^{\Ho\McI_\infty} \to \McE^{\McI_\infty}$ preserves
  internal products.

  Explicitly, this means that for $(\McN,<,\McI,\partial\McI,\McI^\circ)$ the
  data for an iterated gluing diagram with $\McI_\infty$ equipped with a set of
  weak equivalences $\McW$ and a category $\McE$ such that the following
  assumptions hold:
  \begin{itemize}
    \item All limits indexed by each $\oint_{j \in \McI_{<n}} \McI^\circ(i,j)$
    and $([i]_n/\Ho\McI)^\circ$ for $n \in \McN$ and $i \in \partial\McI_n$
    exists in $\McE$.
    \item All maps in $\Ho\McI_\infty$ are epis.
    \item If one puts $\Ho\McI_\infty \coloneqq \McW^{-1}\McI_\infty$ and
    $\gamma \colon \McI_\infty \to \Ho\McI_\infty$ for the homotopical
    localisation then $\gamma$ restricts to
    $\gamma|_i \colon i/\McI_{<n} \to (i/\Ho\McI)^\circ$ that is initial
    (i.e. for each $\id \neq f \colon [i]_n \to j \in \Ho\McI$, the comma
    category $\gamma|_i \downarrow f$ is non-empty and connected) for each
    $i \in \partial\McI_n$.
  \end{itemize}
  then for a map $f \colon B \to A \in \McE^{\Ho\McI_\infty}$ such
  that for each $n \in \McN$ and $i \in \partial\McI_n$,
  \begin{itemize}
    \item The restriction
    $f|_{=n} \colon B|_{=n} \to A|_{=n} \in \McE^{\partial\McI_n}$
    \item The functorial action of the $n$-th absolute matching object functor
    $\M_n(\gamma^*f) \colon \M_n(\gamma^*B) \to \M_n(\gamma^*A) \in
    \McE^{\partial\McI_n}$
    \item The component $f_{[i]_n} \colon B[i]_n \to A[i]_n \in \McE$
    \item The map $\oM_{[i]_n}f \colon \oM_{[i]_n}B \to \oM_{[i]_n}A \in \McE$
    as from \Cref{def:local-boundary}
  \end{itemize}
  are powerful, along with $g \colon C \to B \in \McE^{\Ho\McI}/B$, there is a
  canonical isomorphism
  \begin{align*}
    \gamma^*(\Pi_BC) \cong \Pi_{\gamma^*B}{\gamma^*C} \in \sfrac{\McE^{\McI_\infty}}{\gamma^*A}
  \end{align*}
  given by the map $\varphi_C$ from \Cref{constr:pi-comp}.
\end{theorem}
\begin{proof}
  Fix $g \colon C \to B \in \McE^{\Ho\McI_\infty}/B$.
  For each $n \in \McN$ and $i \in \partial\McI_n$,
  we have a map between pullbacks constructed as follows:
  \begin{equation*}
    \begin{tikzcd}[column sep=-4em, nodes={font=\small}]
      (\Pi_BC)[i]_n
      \ar[dd] \ar[rrrr]
      \ar[rrdd, phantom, "\lrcorner"{pos=0}]
      \ar[rd, "\phi_{C,i}"{description}]
      &
      &
      &
      &
      \Pi_{Bi}{Ci}
      \ar[dd]
      \ar[rd, "="]
      \\
      &[-4em]
      (\Pi_{\gamma^*B}{\gamma^*C})i
      \ar[rrrr, crossing over]
      \ar[rrdd, phantom, "\lrcorner"{pos=0}]
      &
      &
      &
      &
      \Pi_{B[i]_n}{C[i]_n}
      \ar[dd]
      \\
      A[i]_n \times_{\oM_{[i]_n} A} \oM_{[i]_n}(\Pi_B C)
      \ar[rr]
      \ar[rd, "{\psi_{C,i}}"{description}]
      &
      &
      A[i]_n \times_{\oM_{[i]_n} A} \Pi_{\oM_{[i]_n} B}{\oM_{[i]_n}C}
      \ar[rr]
      \ar[rd, "\trho_{C,i}", "\cong"']
      &
      &
      \Pi_{B[i]_n}(B[i]_n \times_{\oM_{[i]_n} B} \oM_{[i]_n}C)
      \ar[rd, "\sigma_{C,i}", "\cong"']
      \\
      &
      Ai \times_{\M_i{\gamma^*A}} \M_i(\Pi_{\gamma^*B}\gamma^*C)
      \ar[uu, crossing over, leftarrow]
      \ar[rr]
      &
      &
      A \times_{\M_i{(\gamma^*A)}} \Pi_{\M_i(\gamma^*B)}\M_i(\gamma^*C)
      \ar[rr]
      &
      &
      \Pi_{Bi}(Bi \times_{\M_i(\gamma^*B)} \M_i(\gamma^*C))
    \end{tikzcd}
  \end{equation*}
  where:
  \begin{itemize}
    \item The front face is by \Cref{thm:E-iter-glue-Pi}.

    \item The back face is by \Cref{lem:ho-I-pi}.
    \item The left face is by \Cref{lem:left-face}.
    \item The bottom left and right faces are respectively by
    \Cref{lem:bot-face-left,lem:bot-face-right}.
    \item The right face is by \Cref{lem:right-face}.
  \end{itemize}
  And moreover, the maps $\tphi_{C,i}$ and $\sigma_{C,i}$ are seen to be
  isomorphisms by \Cref{constr:pi-bd-iso,constr:sigmaCi}.
  We now show that for each $n \in \McN$ and $i \in \partial\McI_n$, we have an
  isomorphism
  $\phi_{C,i} \colon (\Pi_BC)[i]_n \to (\Pi_{\gamma^*B}{\gamma^*C})i$ by way of
  levelwise induction.

  Fix $n \in \McN$ and $i \in \partial\McI_n$.
  Assume that $\psi_{C,j}$ is an isomorphism for each $j \in \McI_{<n}$.
  Per \Cref{constr:pi-comp}, we see that
  $\tphi_{C,i} \colon \oM_{[i]_n}(\Pi_BC) \to \M_i(\Pi_{\gamma^*B}{\gamma^*C})$ is
  constructed as the unique map such
  that for each $u^- \colon i \to j \in i/\McI_{<n}$,
  \begin{equation*}
    \begin{tikzcd}
      \oM_{[i]_n}(\Pi_BC)
      \ar[r, "\tphi_{C,i}"]
      \ar[d, "{\overline\pi}_{[u^-]_n}"']
      &
      \M_i(\Pi_{\gamma^*B}{\gamma^*C})
      \ar[d, "\pi_{u^-}"]
      \\
      (\Pi_BC)j
      \ar[r, "\varphi_{C,j}"']
      &
      (\Pi_{\gamma^*B}{\gamma^*C})j
    \end{tikzcd}
  \end{equation*}
  In particular, for each $u^- \colon i \to j \in i/\McI_{<n}$, by induction,
  the bottom map
  $\varphi_{C,j} \colon (\Pi_BC)[j]_n \cong (\Pi_{\gamma^*B}{\gamma^*C})j$
  is an isomorphism.
  Hence,
  $\tphi_{C,i} \colon \oM_{[i]_n}(\Pi_BC) \cong
  \Pi_i(\Pi_{\gamma^*B}{\gamma^*C})$ is also an isomorphism.
  But then by \Cref{constr:psiCi}, the map
  $\psi_{C,i}
  \colon
  A[i]_n \times_{\oM_{[i]_n}A} \oM_{[i]_n}(\Pi_BC)
  \to
  Ai \times_{\M_i(\gamma^*A)} \M_i(\Pi_{\gamma^*B}{\gamma^*C})$
  is induced by the map of cospans
  \begin{equation*}
    \begin{tikzcd}[row sep=large, column sep=huge]
      A[i]_n
      \ar[r, "\om_{[i]_n}A"]
      \ar[d, "="']
      &
      \oM_{[i]_n}A
      \ar[d, "\kappa_{A,i}"{description}]
      &
      \oM_{[i]_n}(\Pi_BC)
      \ar[l, "\oM_i(\Pi_Bg)"']
      \ar[d, "\tphi_{C,i}", "\cong"']
      \\
      Ai
      \ar[r, "\m_i(\gamma^*A)"']
      &
      \M_i(\gamma^*A)
      &
      \M_i(\Pi_{\gamma^*B}{\gamma^*C})
      \ar[l, "\M_i(\Pi_{\gamma^*B}{\gamma^*g})"]
    \end{tikzcd}
  \end{equation*}
  And further \Cref{lem:mat-comp-iso} gives $\kappa_{A,i}$ as an isomorphism, so
  $\psi_{C,i}$ is an isomorphism.
  From this, it follows that
  $\varphi_{C,i} \colon (\Pi_BC)[i]_n \cong (\Pi_{\gamma^*B}{\gamma^*C})i$.

  Because $\varphi_{C,i}$ is natural in $C$ and $i$ and $\varphi$ is over
  $\gamma^*A$ by \Cref{constr:pi-comp}, this shows that $\gamma^*$ preserves
  internal products.
\end{proof}

Specialising to the case of inverse diagrams, we conclude the following.
\begin{corollary}\label{cor:inv-pi-comp-iso}
  Suppose
  $(\McN,\McI,\partial\McI,\McI^\circ) = (\bN, \McI_{\leq -},\bG_-(\McI),
  \McI_{\leq -}(-,-))$ is the iterated gluing data for a small generalised inverse
  category $\McI$.

  Let $\McI$ be equipped with a wide subcategory of weak equivalences
  $\McW \subseteq \McI$ and put $\Ho\McI \coloneqq \McW^{-1}\McI$ with
  homotopical localisation $\gamma \colon \McI \to \Ho\McI$.
  Assume that
  \begin{itemize}
    \item All maps in $\Ho\McI$ are epis.
    \item The restriction $\gamma|_i \colon \McI^-(i) \to (i/\Ho\McI)^\circ$ is
    an initial functor
  \end{itemize}

  Let $\McE$ be a complete category with an initial object where all maps are
  powerful.
  Denote by $\gamma^* \colon \McE^{\Ho\McI} \to \McE^\McI$ the inclusion of
  $\McI$-shaped homotopical diagrams into the category of all $\McI$-shaped
  diagrams.
  Then, for any maps of homotopical categories
  $f \colon B \to A \in \McE^{\Ho\McI}$, one has an isomorphism
  \begin{align*}
    \gamma^*(\Pi_B-) \cong \Pi_{\gamma^*B}\gamma^*(-) \colon \McE^{\Ho\McI}/B \rightrightarrows \McE^{\McI}/\gamma^*A
  \end{align*}
\end{corollary}
\begin{proof}
  By direct application of \Cref{thm:pi-comp-iso,thm:all-invert-Pi-htpy}.
\end{proof}
